\documentclass[10pt]{amsart}

\usepackage{amssymb, amsmath, amsthm, mathrsfs}

\usepackage{diagrams}
\author{Bram Mesland}

\title[Bivariant $K$-theory and correspondences]{Unbounded bivariant $K$-theory and correspondences in noncommutative geometry}

\date\today
\address{Mathematisch Instituut \\ Universiteit Utrecht \\ Budapestlaan 6 \\ 3584 CD Utrecht \\ The Netherlands}
\email{brammesland@gmail.com}
\theoremstyle{plain}

\newtheorem{theorem}{Theorem}[subsection]
\newtheorem{corollary}[theorem]{Corollary}
\newtheorem*{thm}{Theorem}
\newtheorem{proposition}[theorem]{Proposition}

\newtheorem{lemma}[theorem]{Lemma}

\theoremstyle{definition}

\newtheorem{definition}[theorem]{Definition}

\newtheorem{remark}[theorem]{Remark}
\renewcommand{\thedefinition}{\arabic{subsection}.\arabic{definition}}
\renewcommand{\thetheorem}{\arabic{section}.\arabic{subsection}.\arabic{theorem}}
\renewcommand{\thecorollary}{\arabic{section}.\arabic{subsection}.\arabic{corollary}}
\renewcommand{\theexample}{\arabic{section}.\arabic{subsection}.\arabic{example}}
\renewcommand{\thelemma}{\arabic{section}.\arabic{subsection}.\arabic{lemma}}
\renewcommand{\theproposition}{\arabic{section}.\arabic{subsection}.\arabic{proposition}}
\renewcommand{\theequation}{\arabic{section}.\arabic{equation}}
\renewcommand{\theremark}{\arabic{section}.\arabic{subsection}.\arabic{remark}}

\DeclareFontFamily{OT1}{pzc}{}
\DeclareFontShape{OT1}{pzc}{m}{it}{<-> s * [1.20] pzcmi7t}{}
\DeclareMathAlphabet{\mathpzc}{OT1}{pzc}{m}{it}

\newcommand{\R}{\mathbb{R}}

\newcommand{\Z}{\mathbb{Z}}

\newcommand{\N}{\mathbb{N}}
\newcommand{\K}{\mathbb{K}}
\newcommand{\C}{\mathbb{C}}

\newcommand{\Hom}{\textnormal{Hom}}
\newcommand{\Homst}{\textnormal{Hom}^{*}}

\newcommand{\Mor}{\textnormal{Mor}}

\newcommand{\Aut}{\textnormal{Aut}}

\newcommand{\im}{\mathfrak{Im}}

\newcommand{\id}{\textnormal{id}}

\newcommand{\pr}{\textnormal{pr}}

\newcommand{\e}{\varepsilon}

\newcommand{\Dom}{\mathfrak{Dom}}

\newcommand{\tildeotimes}{\tilde{\otimes}}

\newcommand{\minotimes}{\overline{\otimes}}

\newcommand{\lrh}{\leftrightharpoons}

\newcommand{\End}{\textnormal{End}}
\newcommand{\Endst}{\textnormal{End}^{*}}

\newcommand{\Fin}{\textnormal{Fin}}

\newcommand{\Cor}{\mathfrak{Cor}}

\newcommand{\Sp}{\textnormal{Sp}}
\newcommand{\Sob}{\textnormal{Sob}}
\begin{document}
\begin{abstract}
By introducing a notion of smooth connection for unbounded $KK$-cycles, we show that the Kasparov product of such cycles can be defined directly, by an algebraic formula. 
In order to achieve this it is necessary to develop a framework of smooth algebras and a notion of differentiable $C^{*}$-module. The theory of operator spaces provides the required tools. 
Finally, the above mentioned $KK$-cycles with connection can be viewed as the morphisms in a category whose objects are spectral triples. \newline\newline
\end{abstract}

\maketitle
\small

\normalsize
\tableofcontents
\section*{Introduction}
Spectral triples \cite{Conspec} are a central notion in Connes' noncommutative geometry. The data for a spectral triple consist of a $\Z/2$-graded $C^{*}$-algebra $A$, acting on a likewise graded Hilbert space $\mathpzc{H}$, and a selfadjoint unbounded odd operator $D$ in $\mathpzc{H}$, with compact resolvent, such that the subalgebra
\[\mathcal{A}:=\{a\in A:[D,a]\in B(\mathpzc{H})\},\]
is dense in $A$. The above commutator is understood to be graded. The motivating example is the Dirac operator acting on the Hilbert space of $L^{2}$-sections of a compact spin manifold $M$. 
The $C^{*}$-algebra in question is then just $C(M)$. Over the years, many noncommutative examples of this structure have arisen, in particular in foliation theory \cite{CS} and examples dealing with non-proper group actions.\newline

Shortly after Connes introduction of spectral triples as cycles for $K$-homology \cite{Con}, Baaj and Julg \cite{BJ} generalized this notion to a bivariant setting, by replacing the Hilbert space $\mathpzc{H}$ by a 
$C^{*}$-module $\mathpzc{E}$ over a second $C^{*}$-algebra $B$. The notion of unbounded operator with compact resolvent extends to $C^{*}$-modules, and the commutator condition is left unchanged. Such an object 
$(\mathpzc{E},D)$ can be thought of as a field of spectral triples parametrized by $B$. Baaj and Julg showed, moreover, that such objects can be taken as the cycles for Kasparov's $KK$-theory \cite{Kas}, and the external 
product in $KK$-theory simplifies in this picture. It is given by an algebraic formula.\newline

The main topic of this paper is the construction of a category $\Psi$ of unbounded $KK$-cycles, together with a functor $\Psi\rightarrow KK$, i.e. composition of morphisms in $\Psi$ corresponds to the Kasparov product in $KK$-theory. In order to achieve this, a notion of smoothness for spectral triples is introduced, and this notion is weaker than that of regularity \cite{Conspec} (also known in the literature as $QC^{\infty}$). It is based on the fact that a selfadjoint operator in a Hilbert space $\mathpzc{H}$ is again selfadjoint viewed as an operator in its own graph. Thus, it induces an inverse system of Hilbert spaces
\[\cdots\rightarrow\mathfrak{G}(D_{n})\rightarrow\mathfrak{G}(D_{n-1})\rightarrow\cdots\mathfrak{G}(D)\rightarrow\mathpzc{H},\]
its \emph{Sobolev chain}.
The algebra $\mathcal{A}$ mentioned above can be given an operator space topology by realizing it as matrices through the representation
\[\pi_{1}^{D}:a\mapsto\begin{pmatrix}a & 0 \\ [D,a] & (-1)^{\partial a}a\end{pmatrix}\in B(\mathpzc{H}\oplus\mathpzc{H}).\]
These matrices preserve the graph of $D$, and as such, one can commute them with $D$. This leads one to consider the *-algebra of elements for which the commutators $[D,\pi_{1}^{D}(a)]$ are bounded in $\mathfrak{G}(D)$.  Proceeding inductively, this leads to an inverse system
\[\cdots\rightarrow\mathcal{A}_{k}\rightarrow\mathcal{A}_{k-1}\rightarrow\cdots\mathcal{A}\rightarrow A,\]
acting on the Sobolev chain of $D$. These are \emph{involutive operator algebras}, meaning that the involution $a\mapsto a^{*}$ is completely bounded. Note that this involution is different from that in the containing $C^{*}$-algebra. The definition of $k$-smoothness now entails that the algebra $\mathcal{A}_{k}$ be dense in $A$. In that case, the algebras $\mathcal{A}_{k}$ turn out to be stable under holomorphic functional calculus in $A$. A $C^{k}$-\emph{algebra} will be a $C^{*}$-algebra together with a fixed $C^{k}$-spectral triple in the above sense. This mimicks the definition of a manifold as a topological space equipped with extra structure. \newline

Subsequently we study a class of smooth modules for such algebras. Given a $C^{*}$-module $\mathpzc{E}$ over a sufficiently smooth $C^{*}$-algebra $B$, the existence of an approximate unit which is well behaved with respect to the topology on $\mathcal{B}_{k}$, allows for the resolution of $\mathpzc{E}$ by differentiable submodules
\[\cdots\subset E^{k}\subset E^{k-1}\subset\cdots\subset E^{1}\subset \mathpzc{E}.\]
The notions of adjointable and unbounded regular operators make sense on such modules, and yield properties analoguous to those in $C^{*}$-modules. In particalar, the algebras $\Endst_{\mathcal{B}_{k}}(E^{k})$ and $\K_{\mathcal{B}_{k}}(E^{k})$ are involutive operator algebras. A similar type of module has been studied extensively by Blecher (\cite{Blech2}, \cite{Blech}) and the theory developed here makes essential use of his results. The Haagerup tensor product plays a crucial r\^{o}le. It linearizes the multiplication in algebras of operators on Hilbert spaces. As such, we base the definition of $\Omega^{1}(\mathcal{B}_{k})$, the noncommutative differential forms, on it and we consider connections \[\nabla:E^{k}\rightarrow E^{k}\tildeotimes_{\mathcal{B}_{k}}\Omega^{1}(\mathcal{B}_{k}),\]
on the smooth submodules of $\mathpzc{E}$. When $(\mathpzc{H}, D)$ is a spectral triple for $B$, such that $B$ acts on the Sobolev chain of $D$ up to degree $k$, we can form the operator
\[1\otimes_{\nabla}D: (e\otimes f)\mapsto (-1)^{\partial e}(e\otimes Df+\nabla_{D}(e)f).\]
Its $k$-th Sobelev space is isomorphic to $E^{k}\tildeotimes_{\mathcal{B}_{k}}\mathfrak{G}(T_{k})$. The notion of smoothness also allows us to deal with sums of selfadjoint operators. When the module $\mathpzc{E}$ comes equipped with a selfadjoint regular operator $S$ in $E^{k}$ and the connection is 1-smooth with respect to $S$, then the operator
\[S\otimes 1+1\otimes_{\nabla}D,\]
is selfadjoint in $E^{k}\tildeotimes_{\mathcal{B}_{k}}\mathpzc{H}$. Moreover, we show it has compact resolvent whenever both $D$ and $S$ do so, and thus that this operator defines a spectral triple for $A$ whenever $(\mathpzc{E},S,\nabla)$ is a sufficiently smooth $KK$-cycle with connection. 

More generally, $C^{k}$-cycle is a triple $(E^{k},S,\nabla)$ which is a $(\mathcal{A}_{k},\mathcal{B}_{k})$-bimodule $E^{k}$ with unbounded regular operator $S$ such that $(S\pm i)^{-1}\in\K_{\mathcal{B}_{k}}(E^{k})$ and a sufficiently smooth connection $\nabla:E^{k}\rightarrow E^{k}\tildeotimes_{\mathcal{B}_{k}}\Omega^{1}(\mathcal{B}_{k})$. The isomorphism classes of such cycles are denoted $\Psi_{0}^{k}(A,B)$. We show that such cycles can be composed by the following algebraic formula:
\begin{thm}[ \ref{smregular} ]The composition of $C^{k}$ cycles with connection
\[(E^{k},S,\nabla)\circ(F^{k},T,\nabla')=(E^{k}\tildeotimes_{\mathcal{B}_{k}} F^{k},S\otimes 1 +1\otimes_{\nabla}T,1\otimes_{\nabla}\nabla'),\]
yields a $C^{k}$-cycle with connection, and is associative up to isomorphism.
\end{thm}
That is, this composition \emph{preserves all smoothness properties}. Note that this product is defined on the level of the involutive operator algebras $\mathcal{A}_{k}$ coming from the spectral triple on $A$, and that the $\mathcal{A}_{k}$ are not $C^{*}$-algebras.

Smooth bimodules can then be interpreted as morphisms of spectral triples. This can be captured in a diagram:
\begin{diagram} A &\rightarrow & (\mathpzc{H},D) & \leftrightharpoons & \C \\
\downarrow & & & &\parallel  \\
(E^{k}, S,\nabla)&&&&\C \\
\downharpoonleft\upharpoonright &&&&\parallel\\
B &\rightarrow&(\mathpzc{H}',D')&\leftrightharpoons&\C. \end{diagram}
We use the notation $\mathpzc{E}\leftrightharpoons B$ to indicate that $\mathpzc{E}$, the $C^{*}$-completion of $E^{k}$,  is a $C^{*}$-module over $B$. This also emphasizes the asymmtery, and hence the direction, of the morphisms. It seems appropriate to refer to a bimodule with connection $(E^{k}, S,\nabla)$ as a \emph{geometric correspondence}. 

The  composition of geometric correspondences is the unbounded version of the Kasparov product in $KK$-theory. Recall that the Kasparov product (\cite{Kas})
\[KK_{i}(A,B)\otimes KK_{j}(B,C)\rightarrow KK_{i+j}(A,C),\]
allows one to view the $KK$-groups as morphisms in a category whose objects are all $C^{*}$-algebras. $KK$ is a triangulated category and is universal for $C^{*}$-stable, split-exact functors on the category of $C^{*}$-algebras \cite{Hig}. The degree of a $KK$-cycle is determined by the action of a Clifford algebra. In particular spectral triples can be assigned a degree. Denote the set of unitary isomorphism classes of $k$-smooth geometric correspondences of  the above spectral triples, which we assume to have degrees $i$ and $j$, respectively, by $\mathfrak{Cor}_{k}(D,D')$. The main result of this paper states that 
\begin{thm}[ \ref{functor} ] The bounded transform $\mathfrak{b}:D\mapsto D(1+D^{2})^{-\frac{1}{2}}$ defines a functor
\[\begin{split}\mathfrak{b}:\mathfrak{Cor}_{k}(D,D')&\rightarrow KK_{i-j}(A,B)\\
(E^{k},S,\nabla) &\mapsto [(\mathpzc{E},\mathfrak{b}(D))].\end{split}\]
\end{thm}
This is done by taking $C^{*}$-completions, and forgetting all smoothness and the connection.
In particular it follows that the map $K^{j}(B)\rightarrow K^{i}(A)$ defined by the correspondence maps the $K$-homology class of $(B,\mathpzc{H}',D')$ to that of $(A,\mathpzc{H},D)$.  \newline

The structure of the paper is as follows. In the first three sections we review the theory of $C^{*}$-modules, unbounded operators, $KK$-theory and operator modules. Some of this material is well known, but we introduce 
several constructions that will be used extensively later in the paper. We describe some results that are not stated explicitly in the literature, or emphasize the interconnection of the theories. This should make the second part of the paper an easier read. In section 4 we introduce smoothness
for spectral triples and describe the properties of smooth algebras, smooth modules, and operators thereon. For theoretical purposes this notion is easier to work with and it allows for the definition of a general notion of smooth $C^{*}$-module.
In section 5 we adapt the  theory of connections to the operatore module setting and obtain results on the structure of the graphs of unbounded operators twisted by such a connection. This is used in
section 6 to show that the twisting construction is in fact the Kasparov product in disguise. That in turn leads to the definition of the category of spectral triples described above.
\subsection*{Acknowledgements} This paper was conceived during my Ph.D. studies at the Max Planck Institut f\"{u}r Mathematik in Bonn, Germany. The support of Matilde Marcolli during this period has been of great value. The work was finalized during my stay at Utrecht University, the Netherlands. I thank both intitutions for their support. I am grateful to Florida State University and the California Institute of Technology for their hospitality and support. Many thanks as well to Nigel Higson, for useful and motivating correspondence and conversations. I thank Saad Baaj, Alain Connes, Andre Henriqu\'{e}s, Matthias Lesch, Uuye Otgonbayer and Walter van Suijlekom for useful correspondence and discussions. I am indebted to Nikolay Ivankov for carefully reading the manuscript and numerous useful conversations. Finally I thank Javier Lopez for several conversations we had in the early stages of this project.

\section{$C^{*}$-modules}
From the Gelfand-Naimark theorem we know that $C^{*}$-algebras are a natural
generalization of locally compact Hausdorff topological spaces. In the same vein,
the Serre-Swan theorem tells us that finite projective modules are analogues of locally
trivial finite-dimensional complex vector bundles over a topological
space.The subsequent theory of $C^{*}$-modules, pioneered by Paschke and
Rieffel, should be viewed in the light of these theorems. They are like
Hermitian vector bundles over a space.
\subsection{$C^{*}$-modules and their endomorphism algebras}
In the subsequent review of the established theory, we will assume all
$C^{*}$-algebras and Hilbert spaces to be separable, and all modules to be
countably generated. This last assumption means that there exists a countable
set of generators whose algebraic span is dense in the module. 
\begin{definition}\label{C*mod} Let $B$ be a $C^{*}$-algebra. A \emph{right} $C^{*}$-$B$-\emph{module} is a complex vector space $\mathpzc{E}$ which is also a right $B$-module, equipped with a bilinear pairing
\[\begin{split}\mathpzc{E}\times\mathpzc{E} &\rightarrow B \\
(e_{1},e_{2}) &\mapsto \langle e_{1},e_{2}\rangle, \end{split}\]
such that
\begin{itemize}\item $\langle e_{1},e_{2}\rangle=\langle e_{2},e_{1}\rangle^{*},$
\item $\langle e_{1},e_{2}b\rangle=\langle e_{1},e_{2}\rangle b,$ 
\item $\langle e,e\rangle\geq 0$ and $\langle e,e\rangle=0\Leftrightarrow e=0,$\item $\mathpzc{E}$ is complete in the norm $\|e\|^{2}:=\|\langle e,e\rangle\|.$\end{itemize}
We use Landsman's notation (\cite{Klaas2}) $\mathpzc{E}\leftrightharpoons B$ to indicate this structure. The closure of the linear span of elements of the form $\langle e_{1},e_{2}\rangle$ is an ideal in $\mathpzc{E}$.The module $\mathpzc{E}$ is said to be \emph{full} if this ideal is all of $B$.
\end{definition}
For two such modules, $\mathpzc{E}$ and $\mathpzc{F}$, one can consider operators 
$T:\mathpzc{E}\rightarrow\mathpzc{F}$. As opposed to the case of a Hilbert space 
($B=\C$), such operators need not always have an adjoint with respect to the
inner product. Therefore let 
\[\Homst_{B}(\mathpzc{E},\mathpzc{F}):=\{T:\mathpzc{E}\rightarrow\mathpzc{F}:\quad\exists T^{*}:\mathpzc{F}\rightarrow\mathpzc{E}, \quad\langle Te_{1},e_{2}\rangle=\langle e_{1},T^{*}e_{2}\rangle\}.\]
Elements of $\Homst_{B}(\mathpzc{E},\mathpzc{F})$ are called \emph{adjointable
operators}. When $\mathpzc{E}=\mathpzc{F}$, $\Endst_{B}(\mathpzc{E})$ denote the
 adjointable endomorphisms of the $C^{*}$-module
$\mathpzc{E}$.
It is a $C^{*}$-algebra and contains the canonical $C^{*}$-subalgebra of $B$-\emph{compact operators} denoted by
$\mathbb{K}_{B}(\mathpzc{E})$, constructed as follows. The involution on $B$ allows for considering $\mathpzc{E}$ as a left
$B$-module via $be:=eb^{*}$. The inner product can be used to turn the algebraic tensor product
$\mathpzc{E}\otimes_{B}\mathpzc{E}$ into a $*$-algebra:
\[ e_{1}\otimes e_{2}\circ f_{1}\otimes f_{2}:=e_{1}\langle
e_{2},f_{1}\rangle\otimes f_{2},\quad (e_{1}\otimes e_{2})^{*}:=e_{2}\otimes
e_{1}.\] This algebra is denoted by $\Fin_{B}(\mathpzc{E})$, and $\K_{B}(\mathpzc{E})$ is its norm closure.
\newline\newline
A \emph{grading} on a $C^{*}$-algebra $B$ is an element
$\hat{\gamma}\in\Aut^{*} B$ (a *-automorphism), of order 2. If such a grading is present, $B$ decomposes as $B^{0}\oplus
B^{1}$, where $B^{0}$ is the $C^{*}$-subalgebra of \emph{even} elements, and
$B^{1}$ the closed subspace of \emph{odd} elements. We have $B^{i}B^{j}\subset
B^{i+j}$ for $i,j\in \Z/2\Z$. For $b\in B^{i}$, we denote the \emph{degree} of
$b$
by $\partial b\in\Z/2\Z$. A \emph{graded *-homomorphism} $\phi:A\rightarrow B$ between graded $C^{*}$-algebras, is a *-homomorphism that respects the gradings, i.e. $\phi\circ\hat{\gamma}_{A}=\hat{\gamma}_{B}\circ\phi$.
From now on, we assume all $C^{*}$-algebras to be
graded, possibly trivially, i.e. $\hat{\gamma}=1$.
\begin{definition}\label{graded} A $C^{*}$-module $\mathpzc{E}\lrh B$ is
\emph{graded} if it comes equipped with an element
$\gamma\in\Aut_{\C}(\mathpzc{E})$, of order 2, such that 
\begin{itemize}\item$\gamma(eb)=\gamma(e)\hat{\gamma}(b),$
\item $\langle\gamma(e_{1}),\gamma(e_{2})\rangle=\hat{\gamma}\langle
e_{1},e_{2}\rangle.$\end{itemize} \end{definition}
In this case $\mathpzc{E}$ also decomposes as
$\mathpzc{E}^{0}\oplus\mathpzc{E}^{1}$, and we have $\mathpzc{E}^{i}B^{j}\subset
\mathpzc{E}^{i+j}$ for $i,j\in \Z/2\Z$. The algebras
$\End_{B}(\mathpzc{E}),\Endst_{B}(\mathpzc{E})$ and $\K_{B}(\mathpzc{E})$
inherit a natural grading from $\mathpzc{E}$ by setting
$(\hat{\gamma}T)(e):=\gamma (T\gamma(e))$. For $e\in \mathpzc{E}^{i}$, we denote the
\emph{degree} of $e$
by $\partial e\in\Z/2\Z$.From now on we assume all $C^{*}$-modules to
be graded, possibly trivially.

\subsection{Tensor products}
For a pair of $C^{*}$-modules $\mathpzc{E}\lrh A$ and $\mathpzc{F}\lrh B$, the
vector space tensor product $\mathpzc{E}\otimes\mathpzc{F}$ can be made into
a $C^{*}$-module over the minimal $C^{*}$-tensor product $A\minotimes B$. The
minimal or \emph{spatial}  $C^{*}$-tensor product is obtained as
the closure of $A\otimes B$ in
$\mathbb{B}(\mathpzc{H}\otimes\mathpzc{K})$, where $\mathpzc{H}$ and
$\mathpzc{K}$ are graded Hilbert spaces that carry faithful graded representations of $A$ and $B$
respectively. In order to make $A\minotimes B$ into a graded algebra, the
multiplication law is defined as
\begin{equation}\label{gradedtensor} (a_{1}\otimes b_{1})(a_{2}\otimes b_{2})=(-1)^{\partial
b_{1}\partial a_{2}}a_{1}a_{2}\otimes b_{1}b_{2}.\end{equation}
The completion of $\mathpzc{E}\otimes\mathpzc{F}$ in the inner product
\begin{equation}\nonumber\langle e_{1}\otimes f_{1},e_{2}\otimes f_{2}\rangle:=\langle
e_{1},e_{2}\rangle\otimes\langle f_{1},f_{2}\rangle,\end{equation}
is a $C^{*}$-module denoted by $\mathpzc{E}\minotimes\mathpzc{F}$. It inherits a
grading by setting $\gamma:=\gamma_{\mathpzc{E}}\otimes\gamma_{\mathpzc{F}}.$

The graded module so obtained is
the \emph{exterior tensor product} of $\mathpzc{E}$ and
$\mathpzc{F}$. The \emph{graded tensor product} of maps
$\phi\in\Endst_{A}(\mathpzc{E})$ and $\psi\in\Endst_{B}(\mathpzc{F})$ is defined
by
\begin{equation}\nonumber\phi\otimes\psi(e\otimes
f):=(-1)^{\partial(e)\partial(\psi)}\phi(e)\otimes\psi(f),\end{equation}
gives a graded inclusion
\[\Endst_{A}(\mathpzc{E})\minotimes\Endst_{B}(\mathpzc{F})\rightarrow
\Endst_{A\minotimes B}(\mathpzc{E}\minotimes\mathpzc{F}),\] which restricts to an
isomorphism
\[\K_{A}(\mathpzc{E})\minotimes\K_{B}(\mathpzc{F})\rightarrow \K_{A\minotimes
B}(\mathpzc{E}\minotimes\mathpzc{F}).\]

A *-homomorphism $A\rightarrow\End^{*}_{B}(\mathpzc{E})$ is said to be \emph{essential} if
\[A\mathpzc{E}:=\{\sum_{i=0}^{n}a_{i}e_{i}:a_{i}\in A,e_{i}\in\mathpzc{E}, n\in\N\},\]
is dense in $\mathpzc{E}$. If a graded essential *-homomorphism
$B\rightarrow\Endst_{C}(\mathpzc{F})$ is given, one can complete the algebraic
tensor product $\mathpzc{E}\otimes_{B}\mathpzc{F}$ to a $C^{*}$-module
$\mathpzc{E}\tildeotimes_{B}\mathpzc{F}$ over $C$. The norm in which to complete
comes from the $B$-valued inner product
\begin{equation}\langle e_{1}\otimes f_{1},e_{2}\otimes f_{2}\rangle:=\langle
f_{1},\langle e_{1},e_{2}\rangle f_{2}\rangle.\end{equation}
There is a *-homomorphism
\begin{eqnarray}\nonumber\Endst_{B}(\mathpzc{E})&\rightarrow &\Endst_{C}(\mathpzc{E}\tildeotimes_{B}\mathpzc{F})\\
\nonumber T&\mapsto & T\otimes 1,\end{eqnarray}
which restricts to a homomorphism $\K_{B}(\mathpzc{E})\rightarrow \K_{C}(\mathpzc{E}\tildeotimes_{B}\mathpzc{F})$.If $\mathpzc{E}$ carries an (essential) $A$-representation, then so does $\mathpzc{E}\tildeotimes_{B}\mathpzc{F}$.\newline

We write $\mathpzc{H}_{B}$ for the graded tensor product $\mathpzc{H}\tildeotimes_{\C} B$, where $\mathpzc{H}=\ell^{2}(\Z\setminus\{0\})\cong \ell^{2}(\N)\oplus\ell^{2}(\N)$ with its usual grading. For nonunital $B$ one sets $\mathpzc{H}_{B}:=\mathpzc{H}_{B_{+}}B$. $\mathpzc{H}_{B}$ absorbs any countably generated
$C^{*}$-module. The direct sum $\mathpzc{E}\oplus\mathpzc{F}$ of
$C^{*}$-$B$-modules becomes a $C^{*}$-module in the inner product \[\langle
(e_{1},f_{1}),(e_{2},f_{2})\rangle:=\langle e_{1},e_{2}\rangle +\langle
f_{1},f_{2}\rangle.\]
\begin{theorem}[Kasparov \cite{Kas}] \label{stab}Let $\mathpzc{E}\lrh B$ be a countably generated
graded $C^{*}$-module. Then there exists a graded unitary isomorphism
$\mathpzc{E}\oplus\mathpzc{H}_{B}\xrightarrow{\sim}\mathpzc{H}_{B}$.
\end{theorem}

\subsection{Unbounded operators}
Similar to the Hilbert space setting, there is a notion of unbounded operator on
a $C^{*}$-module. Many of the already subtle issues in the theory of unbounded
operators should be handled with even more care. This is mostly due to the fact
that closed submodules of a $C^{*}$-module need not be orthogonally
complemented. We refer to \cite{Baaj}, \cite{Lan} and \cite{Wor} for detailed
expositions of this theory.
\begin{definition}[\cite{BJ}]\label{reg} Let $\mathpzc{E}$ be a 
$C^{*}$-$B$-module. A densely defined closed operator
$D:\mathfrak{Dom}D\rightarrow\mathpzc{E}$ is called
\emph{regular} if 
\begin{itemize} \item $D^{*}$ is densely defined in $\mathpzc{E}$
\item $1+D^{*}D$ has dense range.
\end{itemize}\end{definition}
Such an operator is automatically $B$-linear, and $\mathfrak{Dom}D$ is a $B$-submodule of $\mathpzc{E}$. 
There are two operators,
$\mathfrak{r}(D),\mathfrak{b}(D)\in\Endst_{B}(\mathpzc{E})$
canonically associated with a regular operator $D$. They are the \emph{inverse modulus} of $D$
\begin{equation} \mathfrak{r}(D):=(1+D^{*}D)^{-\frac{1}{2}},\end{equation}
and the \emph{bounded transform}
\begin{equation}\mathfrak{b}(D):=D(1+D^{*}D)^{-\frac{1}{2}}.\end{equation}
A regular operator $D$ is $\emph{symmetric}$ if $\mathfrak{Dom} D\subset\mathfrak{Dom}D^{*}$ and $D=D^{*}$ on $\mathfrak{Dom} D$. It is \emph{selfadjoint} if it is symmetric and $\mathfrak{Dom}D=\mathfrak{Dom} D^{*}$.
\begin{proposition}\label{core} If $D:\mathfrak{Dom}D\rightarrow\mathpzc{E}$ is regular,
then $D^{*}D$ is selfadjoint and regular. Moreover, $\Dom D^{*}D$ is a core for $D$ and $\im
\mathfrak{r}(D)=\mathfrak{Dom} D$.
\end{proposition}
It follows that $D$ is completely determined by $\mathfrak{b}(D)$, as $\mathfrak{r}(D)^{2}=1-\mathfrak{b}(D)^{*}\mathfrak{b}(D)$.  
Recall that a submodule $\mathpzc{F}\subset\mathpzc{E}$ is \emph{complemented} if $\mathpzc{E}\cong\mathpzc{F}\oplus\mathpzc{F}^{\perp}$, where
\[\mathpzc{F}^{\perp}:=\{e\in\mathpzc{E}:\forall f\in\mathpzc{F}\quad \langle e,f\rangle=0\}.\] Contrary to the Hilbert space case, closed submodules of a $C^{*}$-module need not be
complemented. The complemented submodules of a $C^{*}$-module $\mathpzc{E}$ are precisely those of the form $p\mathpzc{E}$, with $p$ a projection in $\Endst_{B}(\mathpzc{E})$. \newline\newline The \emph{graph of}
$D$ is the closed submodule
\[\mathfrak{G}(D):=\{(e,De):e\in\mathfrak{Dom}(D)\}\subset\mathpzc{E}\oplus\mathpzc{E}.\]
There is a canonical unitary $v\in\Endst_{B}(\mathpzc{E}\oplus\mathpzc{E})$, defined by $v(e,f):=(-f,e)$. Note that $\mathfrak{G}(D)$ and $v\mathfrak{G}(D^{*})$ are orthogonal submodules of $\mathpzc{E}\oplus\mathpzc{E}.$
The following algebraic characterization of regularity is due to Woronowicz.  
\begin{theorem}[\cite{Wor}] \label{Worthm} A densely defined closed operator
$D:\mathpzc{E}\rightarrow\mathpzc{E}$, with densely defined adjoint is regular if and only if 
$\mathfrak{G}(D)\oplus v\mathfrak{G}(D^{*})\cong\mathpzc{E}\oplus\mathpzc{E}.$
\end{theorem}
The isomorphism is given by coordinatewise addition. Moreover, the operator
\begin{equation}\label{Wor}p_{D}:=\begin{pmatrix}\mathfrak{r}(D)^{2} &
\mathfrak{r}(D)\mathfrak{b}(D)^{*}\\\mathfrak{b}(D)\mathfrak{r}(D)&
\mathfrak{b}(D)\mathfrak{b}(D)^{*}\end{pmatrix}\end{equation}
satisfies $p_{D}^{2}=p_{D}^{*}=p_{D}$, i.e. it is a projection, and
$p_{D}(\mathpzc{E}\oplus\mathpzc{E})=\mathfrak{G}(D).$ When $D$ is an odd operator, the grading $\gamma\oplus(-\gamma)$ on $\mathpzc{E}\oplus\mathpzc{E}$ respects the decompositon from theorem \ref{Worthm}. We will always consider $\mathpzc{E}\oplus\mathpzc{E}$ with this grading. In case $D$ is selfadjoint, the above projection takes the form
\[ p_{D}=\begin{pmatrix}(1+D^{2})^{-1} & D(1+D^{2})^{-1} \\ D(1+D^{2})^{-1} & D^{2}(1+D^{2})^{-1}\end{pmatrix},\] so the components are algebraic functions of $D$. Moreover $vpv^{*}=1-p$ in this case. These two facts will play a crucial r\^{o}le in this paper.\newline

The module $\mathfrak{G}(D)$, which
is naturally in bijection with $\mathfrak{Dom}(D)$, 
inherits the structure of a $C^{*}$-module from $\mathpzc{E}\oplus\mathpzc{E}$. We denote its inner product by $\langle\cdot,\cdot\rangle_{1}$. It is a well known fact that 
\[\mathfrak{r}(D)^{2}=(D+i)^{-1}(D-i)^{-1},\]
and the operators $D\pm i$ are bijections $\mathfrak{Dom} D\rightarrow \mathpzc{E}$. We refer to the operators $(D\pm i)^{-1}$ as the \emph{resolvents} of $D$. Since $D$ commutes with $(D\pm i)^{-1}$, $D$ maps $(D\pm i)^{-1}\mathfrak{G}(D)$ into $\mathfrak{G}(D)$. We denote this operator by $D_{2}$.
\begin{proposition}\label{D2} Let $D:\Dom D\rightarrow\mathpzc{E}$ be a selfdajoint regular operator. Then $D_{2}:(D\pm i)^{-1}\mathfrak{G}(D)\rightarrow \mathfrak{G}(D)$ is a selfadjoint regular operator. When $D$ is odd, so is $D_{2}$.
\end{proposition}
\begin{proof} From proposition \ref{core} it follows that \[ (D\pm i)^{-1}\mathfrak{G}(D)=\mathfrak{r}(D)^{2}\mathpzc{E}=\Dom D^{2}.\] $D_{2}$ is closed as an operator in $\mathfrak{G}(D)$ for if $\mathfrak{r}(D)^{2}e_{n}\rightarrow
\mathfrak{r}(D)^{2}e$ and $D\mathfrak{r}(D)^{2}e_{n}\rightarrow e'$ in the topology of $\mathfrak{G}(D)$, then it follows immediately that \[e'=D(D\mathfrak{r}(D)^{2}e)=D^{2}\mathfrak{r}(D)^{2}e.\] It is straightforward to check that
$D_{2}$ is symmetric for the inner product of $\mathfrak{G}(D)$. Hence it is regular, because $(1+D^{2})\mathfrak{r}(D)^{4}\mathpzc{E}=\mathfrak{r}(D)^{2}\mathpzc{E}$. To prove selfadjointness, suppose
$y\in\Dom D$ is such that there exists $z\in\Dom D$ such that for all $x\in\mathfrak{r}(D)^{2}\mathpzc{E}$ $\langle D_{2}x,y \rangle_{1}=\langle x,z\rangle _{1}$. Then $z=Dy$, because
\[\begin{split}\langle Dx,y \rangle_{1} &= \langle Dx,y\rangle +\langle D^{2}x, Dy\rangle \\
&=\langle D\mathfrak{r}(D)^{2}e,y\rangle +\langle D^{2}\mathfrak{r}(D)^{2}e, Dy\rangle \\
&=\langle \mathfrak{r}(D)^{2}e,Dy\rangle +\langle D^{2}\mathfrak{r}(D)^{2}e, Dy\rangle \\
&=\langle e, Dy\rangle. \end{split}\]
A similar computation shows that $\langle x,z\rangle_{1}=\langle e,z\rangle$. Since $\mathfrak{r}(D)^{2}$ is injective this holds for all $e\in\mathpzc{E}$, and hence $z=Dy$. Therefore
\[\Dom D_{2}^{*}=\{y\in\Dom D:Dy\in\Dom D\}=\Dom D^{2}=\mathfrak{r}(D)^{2}\mathpzc{E}=\Dom D_{2},\] so $D_{2}$ is selfadjoint.\end{proof}
\begin{corollary}\label{Sobolev} A selfadjoint regular operator $D:\Dom D\rightarrow\mathpzc{E}$ induces a morphism of inverse systems of $C^{*}$-modules:
\begin{diagram}\cdots &\rTo & \mathpzc{E}_{i+1} &\rTo & \mathpzc{E}_{i} &\rTo & \mathpzc{E}_{i-1} & \rTo & \cdots &\rTo & \mathpzc{E}_{1} &\rTo & \mathpzc{E} \\
                      &\rdTo^{D_{i+1}}&               &\rdTo^{D_{i}}&             &\rdTo^{D_{i-1}}             &  &\rdTo^{D_{i-2}}&  &\rdTo^{D_{2}} &            &\rdTo^{D_{1}=D}&  \\ 
  \cdots &\rTo & \mathpzc{E}_{i+1} &\rTo & \mathpzc{E}_{i} &\rTo & \mathpzc{E}_{i-1} & \rTo & \cdots &\rTo & \mathpzc{E}_{1} &\rTo & \mathpzc{E} \end{diagram}
  \end{corollary}    
  \begin{proof} Set $\mathpzc{E}_{i}=\mathfrak{G}(D_{i}).$ Then the maps $\mathpzc{E}_{i}\rightarrow \mathpzc{E}_{i-1}$ are just projection on the first coordinate, whereas the maps
  $D_{i}:\mathpzc{E}_{i}\rightarrow\mathpzc{E}_{i-1}$ are the projections on the second coordinates. These maps are adjointable, and we have 
  \[D_{i}^{*}(e_{i})=(D_{i}\mathfrak{r}(D_{i})^{2}e_{i}, D^{2}_{i}\mathfrak{r}(D_{i})^{2}e_{i}),\quad \phi_{i}^{*}(e_{i})=(\mathfrak{r}(D_{i})^{2},D_{i}\mathfrak{r}(D_{i})^{2}).\] These are exactly the components of the
Woronowicz projection \ref{Wor}.\end{proof}
We will refer to this inverse system as the \emph{Sobolev chain} of $D$. Almost selfadjoint operators were introduced by Kucerovsky in \cite{Kuc2}. They are adjointable perturbations of selfadjoint operators.
\begin{definition}\label{almostdef} Let $D$ be a regular operator in a $C^{*}$-$B$-module $\mathpzc{E}$. $D$ is \emph{almost selfadjoint} if $\mathfrak{Dom} D=\mathfrak{Dom}D^{*}$ and  $D-D^{*}$ extends to an element in $\Endst_{B}(\mathpzc{E}).$
\end{definition}
 The following result is implicit in \cite{Kuc2}.
\begin{proposition}\label{almost} Let $D$ be an almost selfadjoint operator on a $C^{*}$-$B$-module $\mathpzc{E}$ and $b=D^{*}-D\in\Endst_{B}(\mathpzc{E})$. For $|\lambda|>\|b\|$, the operators $D+\lambda i$, $D^{*}-\lambda i$ are bijections $\mathfrak{Dom} D\rightarrow\mathpzc{E}$.
\end{proposition} 
\begin{proof} The operator $T:=D+D^{*}$ is selfadjoint, and $D=T+b$.  The operators $T+\lambda i$ are bijections $\mathfrak{Dom} D\rightarrow \mathpzc{E}$, and $\|(T+\lambda i)^{-1}\|\leq  \frac{1}{\lambda}$.  Since
\[(D+\lambda i)(T+\lambda i)^{-1}= 1 + b(T+\lambda i)^{-1},\]
and $1 + b(T+\lambda i)^{-1}$ is invertible whenever $|\lambda|>\|b\|$,  we see that $D+\lambda i$ is surjective.
It is injective because
\[\begin{split}\langle (D+\lambda i )e,(D+\lambda i)e\rangle &=\langle De, De\rangle -\lambda i\langle e, De\rangle +\lambda i\langle De, e\rangle +\lambda^{2}\langle e,e\rangle \\
&= \langle De, De\rangle -\lambda i\langle be, e\rangle +\lambda^{2}\langle e,e\rangle \\
&\geq\langle(\lambda ib+\lambda^{2})e,e\rangle +\lambda^{2}\langle e,e\rangle\\
&\geq\lambda^{2}\langle e,e\rangle.\end{split}\]

Reversing the roles of $D$ and $D^{*}$ shows that $D^{*}-\lambda i$ is bijective as well.

\end{proof}
\begin{corollary}\label{idempotent} Let $D$ be an almost selfadjoint regular operator in $\mathpzc{E}$. Then $\Dom D^{*}D=\Dom D^{2}$, and the operator
\[1+\frac{D^{2}}{\lambda^{2}}:\Dom D^{2}\rightarrow \mathpzc{E},\] is bijective for $\lambda$ sufficiently large. Moreover, define
\[p:=\begin{pmatrix}(1+\frac{D^{2}}{\lambda^{2}})^{-1} & \frac{D}{\lambda^{2}}(1+\frac{D^{2}}{\lambda^{2}})^{-1}\\
D(1+\frac{D^{2}}{\lambda^{2}})^{-1} &  \frac{D^{2}}{\lambda^{2}}(1+\frac{D^{2}}{\lambda^{2}})^{-1}\end{pmatrix},\quad v_{\lambda}:=\begin{pmatrix} 0&-\lambda^{-1} \\ \lambda &0\end{pmatrix}, \]
then $p$ is an idempotent and $v_{\lambda}$ an invertible in $\Endst_{B}(\mathpzc{E})$ such that $\mathfrak{Im} p=\mathfrak{G}(D)$ and $v_{\lambda}pv^{-1}_{\lambda}=1-p$.
\end{corollary}
\begin{proof}Since $\Dom D^{*}=\Dom D$, we have $\Dom D^{*}D=\Dom D^{2}$. By proposition \ref{almost} $D\pm \lambda i$ are bijections $\Dom D\rightarrow \mathpzc{E}$. Thus, 
\[\lambda^{2}+D^{2}=(D+\lambda i)(D-\lambda i):\Dom D^{2}\rightarrow \mathpzc{E},\]
bijectively as well. Moreover, the inverse $(\lambda^{2}+D^{2})^{-1}=(D+\lambda i)^{-1}(D-\lambda i)^{-1}$ is bounded and adjointable. That $p$ is idempotent is now easily checked, as well the property $v_{\lambda}pv^{-1}_{\lambda}=1-p$. It is immediate that $\mathfrak{Im} p \subset\mathfrak{G}(D)$ and $\mathfrak{Im} p^{*} \subset\mathfrak{G}(D^{*})$. Therefore
\[\ker p=\mathfrak{Im}(1-p^{*})=\mathfrak{Im}vp^{*}v^{*}\subset v\mathfrak{G}(D^{*}),\]
which implies that $\mathfrak{Im} p=\mathfrak{G}(D)$.
\end{proof}
Thus, for an almost selfadjoint operator there is an invertible adjointable operator
\[g:\mathfrak{G}(D)\oplus v_{\lambda}\mathfrak{G}(D)\xrightarrow{\sim}\mathpzc{E}\oplus\mathpzc{E},\] 
which is a key example of the following definition.
\begin{definition}\label{topiso} Two $C^{*}$-modules $\mathpzc{E}\lrh B$ and $\mathpzc{F}\lrh B$ are \emph{topologically isomorphic} if there are $g\in\Homst_{B}(\mathpzc{E},\mathpzc{F})$ and $g^{-1}\in\Homst_{B}(\mathpzc{F},\mathpzc{E})$ with \[gg^{-1}=1_{\mathpzc{F}},\quad g^{-1}g=1_{\mathpzc{E}}.\] Such a $g$ is called a \emph{topological isomorphism}.
\end{definition}

\begin{proposition} An almost selfadjoint operator $D$ is almost selfadjoint in its own graph, and hence induces a Sobolev chain as in the selfadjoint case.
\end{proposition}
\begin{proof} We define $D$ in its own graph on the domain $(D+\lambda i)^{-1}\mathfrak{G}(D)$. This is dense since $\Dom D^{*}D=\Dom D^{2}$ is a core for $D$. It is straightforward to check that $D_{2}$ is closed on this domain, and that $2R=D_{2}-D_{2}^{*}$ is bounded adjointable. Moreover, by definition $D_{2}+\lambda i$ is surjective and has adjointable inverse. Therefore
\[\frac{1}{2}(D_{2}+D_{2}^{*}+\lambda i)(D_{2}+\lambda i)^{-1}=1+R(D_{2}+\lambda i)^{-1},\] is invertible for $\lambda $ sufficiently large, and $D_{2}+D_{2}^{*}$ is selfadjoint.
\end{proof}
\section{$KK$-theory}
Kasparov's bivariant $K$-theory $KK$ \cite{Kas} has become a central tool in noncommutative geometry since its creation. It is a bifunctor on pairs of $C^{*}$-algebras, associating to $(A,B)$ a $\Z/2\Z$-graded group $KK_{*}(A,B)$.
It unifies $K$-theory and $K$-homology in the sense that \[KK_{*}(\C,B)\cong K_{*}(B) \textnormal{  and  } KK_{*}(A,\C)\cong K^{*}(A).\]
Much of its usefulness comes from the existence of internal and external product structures, by which $KK$-elements induce homomorphisms between $K$-theory and $K$-homology groups. In Kasparov's original approach, the
definition and computation of the products is very complicated. In order to simplify the external product, Baaj and Julg \cite{BJ} introduced another model for $KK$, in which the external product is given by a simple
algebraic formula. The price one has to pay is working with unbounded operators.
\subsection{The bounded picture}
The main idea behind Kasparov's approach to $K$-homology and $KK$-theory is that of a family of abstract elliptic operators. This was an idea pioneered by Atiyah, in his construction of $K$-homology for spaces and the
family index theorem. 
\begin{definition}[\cite{Kas}]\label{Kasparov} Let $A\rightarrow\mathpzc{E}\leftrightharpoons B$ be a graded bimodule and $F\in\Endst_{B}(\mathpzc{E})$ an odd operator. $(\mathpzc{E},F)$ is a \emph{Kasparov $(A,B)$-bimodule} if, for all $a\in A$, 
\begin{itemize}\item $ [F,a], a(F^{2}-1),a(F-F^{*})\in\K_{B}(\mathpzc{E})$.\end{itemize}
\end{definition}
The set of Kasparov modules up to unitary equivalence is denoted $\mathbb{E}_{0}(A,B)$, and $\mathbb{E}_{j}(A,B):=\mathbb{E}_{0}(A,B\minotimes\C_{j})$, where $\C_{j}$ is the $j$-th complex Clifford algebra.
The set of \emph{degenerate} elements consists of bimodules for which
 \[\forall a\in A:[F,a]=a(F^{2}-1)=a(F-F^{*})=0.\] Denote by $e_{i}:C[0,1]\minotimes B\rightarrow B$ the evalution map at $i\in[0,1]$. Two Kasparov $(A,B)$-bimodules $(\mathpzc{E}_{i},F_{i})\in\mathbb{E}_{j}(A,B)$, $i=0,1$ 
 are \emph{homotopic} if there exists a Kasparov $(A,C[0,1]\minotimes B)$-module $(\mathpzc{E},F)\in\mathbb{E}_{j}(A,C[0,1]\otimes B)$ for which $(\mathpzc{E}\otimes_{e_{i}}B,F\otimes 1)$ is unitarily equivalent to $(\mathpzc{E}_{i},F_{i})$, $i=0,1$. It is an equivalence relation, denoted $\sim$. Define
\[KK_{j}(A,B):=\mathbb{E}_{j}(A,B)/\sim.\] 
$KK_{j}$ is a bifunctor, contravariant in $A$, covariant in $B$, taking values in abelian groups. It is not hard to show that $KK_{*}(\C,A)$ and $KK_{*}(A,\C)$ are naturally isomorphic to the $K$-theory and $K$-homology of
$A$, respectively. Moreover, Kasparov proved the following deep theorem.
\begin{theorem}[\cite{Kas}] For any $C^{*}$-algebras $A,B,C$ there exists an associative bilinear pairing
\[KK_{i}(A,B)\otimes_{\Z} KK_{j}(B,C)\xrightarrow{\otimes_{B}} KK_{i+j}(A,C).\] Therefore, the groups $KK_{*}(A,B)$ are the morphism sets of a category $KK$ whose objects are all $C^{*}$-algebras.
\end{theorem}

There also is a notion of external product in $KK$-theory.
\begin{theorem}[\cite{Kas}]For any $C^{*}$-algebras $A,B,C,D$ there exists an associative bilinear pairing
\[KK_{i}(A,C)\otimes_{\Z} KK_{j}(B,D)\xrightarrow{\minotimes} KK_{i+j}(A\minotimes B,C\minotimes D).\]
The external product makes $KK$ into a symmetric monoidal category
\end{theorem}
The category $KK$ has more remarkable properties. Although we will not use them in this paper, we do believe they deserve a brief mention. It was shown by Cuntz and Higson (\cite{Cuntz},\cite{Hig}) that the category $KK$ is universal in the sense that any split exact stable functor from the category of $C^{*}$-algebras to, say, that of
abelian groups, factors through the category $KK$. Altough it fails to be abelian, $KK$ is a triangulated category. This allows for the development of homological algebra in it, which has special interest in relation to the
Baum-Connes conjecture, an approach pursued by Nest and Meyer \cite{Nest}.
\subsection{The unbounded picture}
One can define $KK$-theory using unbounded operators on $C^{*}$-modules. As the
bounded definition corresponds to abstract order zero elliptic
pseudodifferential operators, the unbounded version corresponds to order one
operators.
\begin{definition}[\cite{BJ}]\label{unbounded} Let $A\rightarrow\mathpzc{E}\leftrightharpoons B$ be a graded bimodule and $D:\Dom D\rightarrow\mathpzc{E}$ an odd selfadjoint regular operator. The pair $(\mathpzc{E},D)$ is an \emph{$KK$-cycle} for  $(A,B)$ if, 
for all $a\in \mathcal{A}$, a dense subalgebra of $A$ 
\begin{itemize}\item $a\Dom D\subset\Dom D$ and $ [D,a]$ extends to an operator in $\Endst_{B}(\mathpzc{E})$\item $a\mathfrak{r}(D)\in\K_{B}(\mathpzc{E})$.\end{itemize}
\end{definition}
Denote the set of KK-cycles for $(A,B\tildeotimes\C_{i})$ modulo unitary equivalence by $\Psi_{i}(A,B)$. As in the bounded case, we will refer to elements of $\Psi_{0}$
as \emph{even} unbounded bimodules. In \cite{BJ} it is shown that $(\mathpzc{E},\mathfrak{b}(D))$ is a Kasparov bimodule, and that every element in $KK_{*}(A,B)$ can be represented by an unbounded bimodule. The motivation for introducing unbounded modules is
the following result.
\begin{theorem}[\cite{BJ}]\label{ext} Let $(\mathpzc{E}_{i},D_{i})$ be unbounded bimodules for $(A_{i},B_{i})$, $i=1,2$. The operator
\[D_{1}\otimes 1+1\otimes D_{2}:\Dom D_{1}\otimes \Dom D_{2}\rightarrow \mathpzc{E}\otimes\mathpzc{F},\]
extends to a selfadjoint regular operator with compact resolvent. Moreover,
the diagram
\begin{diagram}\Psi_{i}(A_{1},B_{1})\times\Psi_{j}(A_{2},B_{2})&\rTo &\Psi_{i+j}(A_{1}\minotimes A_{2},B_{1}\minotimes B_{2})\\
\dTo^{\mathfrak{b}} & & \dTo^{\mathfrak{b}}\\
KK_{i}(A_{1},B_{1})\times KK_{j}(A_{2},B_{2})&\rTo^{\minotimes} & KK_{i+j}(A_{1}\minotimes A_{2},B_{1}\minotimes B_{2})\end{diagram}
commutes.\end{theorem}
Consequently, we can \emph{define} the external product in this way, using
unbounded modules. In \cite{Kuc}, Kucerovsky gives sufficient conditions for an unbounded module
$(\mathpzc{E}\tildeotimes_{A}\mathpzc{F},D)$ to be the internal product of $(\mathpzc{E},S)$ and $(\mathpzc{F},T)$. For each $e\in\mathpzc{E}$, we have an operator
\[\begin{split} T_{e}:\mathpzc{F}&\rightarrow \mathpzc{E}\tildeotimes_{B}\mathpzc{F}\\
f &\mapsto e\otimes f.\end{split}\]
Its adjoint is given by $T^{*}_{e}(e'\otimes f)=\langle e,e'\rangle f$. Kucerovsky's result now reads as follows.
\begin{theorem}[\cite{Kuc}]\label{Kuc} Let
$(\mathpzc{E}\tildeotimes_{B}\mathpzc{F},D)\in\Psi_{0}(A,C)$. Supppose that
$(\mathpzc{E},S)\in\Psi_{0}(A,B)$ and $(\mathpzc{F},T)\in\Psi_{0}(B,C)$ are such
that
\begin{itemize}\item For $e$ in some dense subset of $A\mathpzc{E}$, the operator
\[\left[\begin{pmatrix}D & 0\\ 0& T\end{pmatrix},\begin{pmatrix} 0 &
T_{e}\\T^{*}_{e} & 0\end{pmatrix}\right],\]
is defined on $\Dom(D\oplus T)$ and extends to an operator in $\Endst_{C}(\mathpzc{E}\tildeotimes_{B}\mathpzc{F}\oplus\mathpzc{F});$
\item $\Dom D\subset \Dom S\tildeotimes 1$ ;
\item For some $\kappa\in\R$, $\langle Sx,Dx\rangle + \langle Dx,Sx\rangle \geq\kappa\langle x,x\rangle$
for all $x$ in the domain.\end{itemize} Then
$(\mathpzc{E}\tildeotimes_{B}\mathpzc{F},D)\in\Psi_{0}(A,C)$ represents the
internal Kasparov product of $(\mathpzc{E},S)\in\Psi_{0}(A,B)$ and $(\mathpzc{F},T)\in\Psi_{0}(B,C)$.\end{theorem}
This theorem only gives sufficient conditions, and gives an indication about the actual form of the product of two given cycles. By equipping unbounded bimodules with some extra differential structure, we will obtain an algebraic description of the product cycle. To this end, we need to extend our scope from $C^{*}$-modules to a class of similar objects, defined over a larger class of topological algebras.
 \section{Operator modules}
When dealing with unbounded operators, it becomes necessary to deal with dense subalgebras of $C^{*}$-algebras and modules over these. The theory of $C^{*}$-modules, which is the basis of Kasparov's approach to
bivariant $K$-theory for $C^{*}$-algebras, needs to be extended in an appropriate way. The framework of operator spaces and the Haagerup tensor product provides with a class of modules and algebras which is
sufficiently rich to accomodate for the phenomena occurring in the Baaj-Julg picture of $KK$-theory.
\subsection{Operator spaces}
We will frequently deal with algebras and modules that are not $C^{*}$. In this section we discuss the basic notions of the theory of operator spaces, in which all of our
examples will fit. There is an intrinsic approach presented in \cite{Helem}. The link between the theory we describe here and the aforementioned intrinsic approach can be found in \cite{Ruan}. 
\begin{definition}\label{opsp} An \emph{operator space} $X$ is a closed linear subspace of some $C^{*}$-algebra. As such there are canonical norms on the matrix spaces $M_{n}(X)$ and the space $\K\minotimes X$. A linear map $\phi:X\rightarrow Y$ between operator spaces is called \emph{completely bounded}, resp. \emph{completely contractive}, resp. \emph{completely isometric} if the induced map 
\[1\otimes\phi:\K\otimes X\rightarrow \K\otimes Y,\] is bounded,
resp. contractive, resp isometric for the minimal tensorproduct norm. The norm of $1\otimes\phi$ is denoted $\|\phi\|_{cb}$ and equals $\sup_{n}\|1_{n}\otimes\phi\|$, where 
\[1_{n}\otimes\phi:M_{n}(\C)\otimes X\rightarrow M_{n}(\C)\otimes Y.\]
\end{definition}
Any $C^{*}$-module $\mathpzc{E}$ over a $C^{*}$-algebra $B$ is an operator space, as it is isometric to $\K(B,\mathpzc{E})$, which is a closed subspace of $\K(B\oplus\mathpzc{E})$, the \emph{linking algebra} of $\mathpzc{E}$.\newline\newline
Let $\mathpzc{E}$ be an $(A,B)$ bimodule and $D$ an odd regular operator in $\mathpzc{E}$. Define
\[\mathcal{A}_{1}:=\{a\in A:[D,a]\in\Endst_{B}(\mathpzc{E}). \}\]
Let $\delta:\mathcal{A}_{1}\rightarrow \Endst_{B}(\mathpzc{E})$ the closed derivation $a\mapsto [D,a]$. 
Then $\mathcal{A}_{1}$ can be made into an operator space via
\begin{eqnarray}\label{Der}\pi_{1}:\mathcal{A}_{1}&\rightarrow & M_{2}(\Endst_{B}(\mathpzc{E}))\\
a &\mapsto & \begin{pmatrix}a & 0 \\ \delta(a) & \gamma a\gamma \end{pmatrix}.\end{eqnarray}
Here $\gamma$ is the grading on $\mathpzc{E}$ this construction in particular applies to $KK$-cycles for $(A,B)$ $(\mathpzc{E},D)$, in which case $\mathcal{A}_{1}$ is dense in $A$.Equipped with this operator space structure, $\mathcal{A}$ admits a completely contractive algebra homomorphism $\mathcal{A}\rightarrow \Endst_{B}(\mathfrak{G}(D)).$ 
\subsection{The Haagerup tensor product}
For operator spaces $X$ and $Y$, one can define their spatial tensor product $X\minotimes Y$ as the norm closure of the algebraic tensor product in the spatial tensor product of some containing $C^{*}$-algebras. This gives
rise to an exterior tensor product of operator modules.\newline\newline
The internal tensor product of $C^{*}$-modules is an example of the Haagerup tensor product for operator spaces. This tensor product will be extremely important in what follows.
\begin{definition} Let $X,Y$ be operator spaces. The \emph{Haagerup norm} on $\K\otimes X\otimes Y$ is defined by
\[\|u\|_{h}:=\inf\{\sum_{i=0}^{n}\|x_{i}\|\|y_{i}\|:u=m(\sum x_{i}\otimes y_{i}), x_{i}\in\K\otimes X, y_{i}\in\K\otimes Y\}.\]
Here $m:\K\otimes X\otimes\K\otimes Y\rightarrow \K\otimes X\otimes Y$ is the linearization of the map $(a\otimes x,b\otimes y)\mapsto (ab\otimes x\otimes y)$.\end{definition}
\begin{theorem}\label{hnorm}  If $X\subset B(\mathpzc{H})$ and $Y\subset B(\mathpzc{K})$, the norm on $X\otimes Y$ induced by the Haagerup norm is given by
\[\|\sum_{j} x_{j}\otimes y_{j}\|_{h}=\inf\{\|\sum v_{i}v_{i}^{*}\|^{\frac{1}{2}}\|\sum w^{*}_{i}w_{i}\|^{\frac{1}{2}}:\sum_{i} v_{i}\otimes w_{i}=\sum_{j} x_{j}\otimes y_{j}\}.\]
and the completion of $X\otimes Y$ in this norm is an operator space denoted $X\tildeotimes Y$.
\end{theorem}
The completion $X\tildeotimes Y$ and is called the \emph{Haagerup tensor product} of $X$ and $Y$.
From this theorem we deduce the following useful property. Whenever $x_{i}\in X, y_{i}\in Y$ are sequences such that \[ \|\sum y^{*}_{i}y_{i}\|,\|\sum x_{i}x_{i}^{*}\|\leq 1,\]
then $\sum x_{i}\otimes y_{i}$ is convergent for the Haagerup norm and defines an element $w\in X\tildeotimes Y$, with $\|w\|\leq 1$. Another consequence of this is the following.
\begin{proposition} For a closed subalgebra $\mathcal{A}\subset B(\mathpzc{H})$, operator multiplication induces a completely contractive map $m:\mathcal{A}\tildeotimes\mathcal{A}\rightarrow \mathcal{A}$. 
\end{proposition}
\begin{proof} First define $m$ on the algebraic tensor product via $a\otimes b\mapsto ab$. Then estimate
\[\begin{split}\|m(\sum_{i=1}^{n} a_{i}\otimes b_{i})\|&=\|\sum_{i=1}^{n} a_{i}b_{i}\| \\
&=\|\langle (a_{i}^{*})_{i=1}^{n},(b_{i})_{i=1}^{n}\rangle\| \\ 
&\leq \|(a_{i}^{*})\|\|(b_{i})\| \\
&= \|\sum a_{i}a_{i}^{*}\|^{\frac{1}{2}}\|\sum b_{i}^{*}b_{i}\|^{\frac{1}{2}},\end{split} \]
where we viewed the expression $\sum_{i=1}^{n} a_{i}b_{i}$ as an inner product of two vectors in the $C^{*}$-module $\bigoplus_{i=1}^{n}B(\mathpzc{H})$. Since this inequality holds for any representative of $\sum_{i=1}^{n} a_{i}b_{i},$ it follows from theorem \ref{hnorm} that $\|m(\sum_{i=1}^{n} a_{i}\otimes b_{i})\|\leq \|\sum_{i=1}^{n}a_{i}\otimes b_{i}\|_{h}$, so $m$ is continuous.
\end{proof}
\begin{definition}\label{opalg} An \emph{operator algebra} is an operator space $\mathcal{A}$ which is an algebra, such that the multiplication induces a completely bounded map $\mathcal{A}\tildeotimes\mathcal{A}\rightarrow\mathcal{A}$.  A (right) \emph{operator module} is an operator space $M$ which is a right module over an operator algebra $\mathcal{A}$, such that the module multiplication induces a completely bounded map $M\tildeotimes\mathcal{A}\rightarrow M$.
\end{definition} 
Note that the multiplication in the above definition is only abstractly defined and need not coincide with operator multiplication. However, in \cite{blechercb} it is proved that such operator algebras are completely boundedly isomorphic to a subalgebra of $B(\mathpzc{H})$ for some $\mathpzc{H}$.\newline

The module $\mathfrak{G}(D)\subset\mathpzc{E}\oplus\mathpzc{E}$ from example \ref{Der} is a (left)-operator module over the operator algebra $\mathcal{A}$.
The natural choice of morphisms between operator modules are the completely bounded module maps. If $E$ and $F$ are operator modules over an operator algebra $\mathcal{A}$, we denote the set of these maps by
$\Hom^{c}_{\mathcal{A}}(E,F).$\newline

A \emph{countable approximate unit} for an operator algebra $\mathcal{A}$ is a sequence $\{u_{n}\}\subset\mathcal{A}$ such that $\sup_{n}\|u_{n}\|_{cb}< \infty$ and \[\lim_{n\rightarrow\infty}\|au_{n}-a\|=\lim_{n\rightarrow\infty}\|u_{n}a-a\|=0,\]
for all $a\in\mathcal{A}$. 
We use the completely bounded version of operator algebras and modules as the completely contractive picture is too restrictive for our purposes. Surprisingly, the cb-theory is more complicated than the contractive theory, in some aspects, especially when dealing with nonunital algebras. An excellent reference for operator algebra and module theory is \cite{Blechbook}. \newline
\begin{definition}Let $\mathcal{A},\mathcal{B}$ be an operator algebras. A \emph{completely bounded anti-isomorphism} is an antilinear bijection $\phi:\mathcal{A}\rightarrow \mathcal{B}$ such that $\phi(ab)=\phi(b)\phi(a)$, for which the norms of the matrix extensions $\phi(a_{ij}):=(\phi(a_{ji}))$ are uniformly bounded. An \emph{involutive operator algebra} is an operator algebra which carries an involution $a\mapsto a^{*}$, which is a completely bounded anti-isomorphism. If $M$ is a right- and $N$ a left operator module over an involutive operator algebra $\mathcal{A}$, a \emph{completely bounded anti-isomorphism} is an antilinear bijection $\phi:M\rightarrow N$ such that $\phi(ma)=a^{*}\phi(m)$.
\end{definition}
Of course, $C^{*}$-algebras and -modules are examples that fit this definition. The algebra $\mathcal{A}_{1}$ from example \ref{Der} is an involutive operator algebra since $\pi_{1}(a^{*})=v\pi_{1}(a)^{*}v^{*}$, and hence $\|a\|=\|a^{*}\|$. \newline\newline 
Now suppose $M$ is a right operator $\mathcal{A}$-module, and $N$ a left operator $\mathcal{A}$-module. Denote by $I_{\mathcal{A}}\subset M\tildeotimes N$ the closure of the linear span of the expressions $(ma\otimes n-m\otimes
an).$ The \emph{module Haagerup tensor product} of $M$ and $N$ over $\mathcal{A}$ (\cite{BMP}) is
\[M\tildeotimes_{\mathcal{A}}N:=M\tildeotimes N/I_{\mathcal{A}},\]
equipped with the quotient norm, in which it is obviously complete. Moreover, if $M$ also carries a left $\mathcal{B}$ operator module structure, and $N$ a right $\mathcal{C}$ operator module structure, then
$M\tildeotimes_{\mathcal{A}}N$ is an operator $\mathcal{B},\mathcal{C}$-bimodule.  Graded operator algebras and -modules can be defined by the same conventions as in definition \ref{graded} and the discussion
preceeding it. If the modules and operator algebras are graded, so are the Haagerup tensor products, again in the same way as in the $C^{*}$-case, as in the discussion around equation \ref{gradedtensor}. 
The following theorem resolves the ambiguity in the notation for the interior tensor product of $C^{*}$-modules and the Haagerup tensor
product of operator spaces.
\begin{theorem}[\cite{Blech}]\label{int} Let $\mathpzc{E},\mathpzc{F}$ be $C^{*}$-modules over the $C^{*}$-algebras $B$ and $C$ respectively, and $\pi:B\rightarrow \Endst_{C}(\mathpzc{F})$ a nondegenrate *-homomorphism. Then the interior tensor product and the Haagerup
tensor product of $\mathpzc{E}$ and $\mathpzc{F}$ are completely isometrically isomorphic.
\end{theorem}
This result provides us with a convenient description of algebras of compact operators on $C^{*}$-modules. The \emph{dual module} of a $C^{*}$-module $\mathpzc{E}$ is anti-isomorphic to $\mathpzc{E}$ as a linear space, and we equip
it with a left $C^{*}$-$B$-module structure using the involution: \[be:=eb^{*},\quad (e_{1},e_{2})\mapsto\langle e_{1},e_{2}\rangle^{*}.\]
\begin{theorem}[\cite{Blech}] There is a completely isometric isomorphism \[\K_{C}(\mathpzc{E}\tildeotimes\mathpzc{F})\xrightarrow{\sim} \mathpzc{E}\tildeotimes_{B}\K_{C}(\mathpzc{F})\tildeotimes_{B}\mathpzc{E}^{*}.\] In
particular $\K_{B}(\mathpzc{E})\cong \mathpzc{E}\tildeotimes_{B}\mathpzc{E}^{*}$.
\end{theorem}

\subsection{Stably rigged modules}
The work of Blecher \cite{Blech} provides a metric description of
$C^{*}$-modules which is useful in extending the theory to non $C^{*}$-algebras.
The algebra $\K_{B}(\mathpzc{E})$ associated to a $\Z/2$-graded countably generated $C^{*}$-$B$-module $\mathpzc{E}$, admits an  
approximate unit $\{u_{n} \}_{n\in \N }$ consisting of elements in $\Fin_{B}(\mathpzc{E})$. 
Replacing $u_{n}$ by
$u_{n}^{*}u_{n}$ if necessary, we may assume
\begin{equation}\label{approx}u_{n}=\sum_{1\leq|i|\leq n} x_{i}\otimes x_{i},\end{equation}
by invoking Kasparov's stabilization theorem.
For each $n$ we get operators $\phi_{n}\in\K_{B}(\mathpzc{E},B^{2n})$, defined by
\begin{equation}\label{rig}\phi_{n}:e \mapsto  (\langle x^{\alpha}_{i},e\rangle)_{1\leq|i|\leq n}.  \end{equation} 
We have
\begin{equation}\label{rig*}\phi_{n}^{*}:(b_{i})_{i=-n}^{n}\mapsto  \sum_{1\leq|i|\leq n}x_{i}b_{i}, \end{equation}
and hence $\phi^{*}_{n}\circ \phi_{n}\rightarrow \id_{\mathpzc{E}}$ pointwise. This structure determines the $\mathpzc{E}$ completely as a
$C^{*}$-module.
\begin{theorem}[\cite{Blech}]\label{Blech}Let $B$ be a graded separable $C^{*}$-algebra and $\mathpzc{E}$ be an operator space which is also a graded
right operator module over $B$. Then $\mathpzc{E}$ is completely isometrically isomorphic to a countably generated $C^{*}$-module if and only if there exist completely contractive module maps
\[\phi_{n}:\mathpzc{E}\rightarrow B^{2n},\quad \psi_{n}:B^{2n}\rightarrow \mathpzc{E},\]
of degree $0$, such that $\psi_{n}\circ\phi_{n}$ converges pointwise to the identity on $\mathpzc{E}$. In this case the inner product on $\mathpzc{E}$ is given by
\[\langle e, f\rangle = \lim_{n\rightarrow\infty}\langle \phi_{n}(e),\phi_{n}(f)\rangle.\]
\end{theorem}
For this reason we can think of $C^{*}$-modules as approximately finitely generated projective modules. Also note that the maps $\phi_{n},\psi_{n}$ are by no means unique, and that different maps can thus give rise to the same
inner product on $\mathpzc{E}$. The description of $C^{*}$-modules in theorem \ref{Blech} is metric, and hence generalizes to non-selfadjoint operator algebras with contractive approximate unit.
\begin{definition}[cf. \cite{Blech2}]\label{rigged} Let $\mathcal{B}$ be an operator algebra with completely contractive approximate identity, and $E$ a right $\mathcal{B}$-operator module. $E$ is a countably generated $\mathcal{B}$-\emph{rigged module} if there exist completely contractive $\mathcal{B}$-module maps
\[\phi_{n}:E\rightarrow\mathcal{B}^{2n},\quad \psi_{n}:\mathcal{B}^{2n}\rightarrow E,\]
such that $\psi_{n}\circ\phi_{n}\rightarrow\id_{E}$ strongly on $E$. Subsequently define the \emph{dual module} of $E$ by
\[E^{*}:=\{e^{*}\in\Hom^{c}_{\mathcal{B}}(E,\mathcal{B}):e^{*}\circ\psi_{n}\circ\phi_{n}\rightarrow e^{*}\},\]
and the algebra of $\mathcal{B}$-\emph{compact operators} as $\K_{\mathcal{B}}(E):=E\tildeotimes_{\mathcal{B}} E^{*}.$
\end{definition} 
\begin{remark} In \cite{Blech2}, three more conditions appear in the definition of rigged module. The first one is that the module $E$ be  \emph{essential}, i.e. $E\mathcal{B}$ is dense in $E$. Moreover it was required that $\phi_{n}\psi_{k}\phi_{k}\rightarrow \phi_{n}$ and $\psi_{n}u_{i}\rightarrow \psi_{n}$ in norm. Here $u_{i}$ is a bounded approximate identity for $\mathcal{B}$. All of these conditions were shown to be superfluous in \cite{her}.
\end{remark}
\begin{remark}\label{dual}It is immediate from this definition that $E^{*}=\K_{\mathcal{B}}(E,\mathcal{B})$. This module satisfies the transposed version of \ref{rigged}, i.e. it is a left rigged $\mathcal{B}$-module \cite{Blech2}. The module structure comes from the left module structure on $\mathcal{B}$ itself, $(be^{*})(e)=be^{*}(e)$. For the rigged structure on a $C^{*}$-module, coming from the approximate unit \eqref{approx}, the structural maps $\psi_{n}^{*}:E^{*}\rightarrow(\mathcal{B}^{2n})^{t}$ and $\phi_{n}^{*}:(\mathcal{B}^{2n})^{t}\rightarrow E^{*}$ are given by
\[\psi_{n}^{*}(e^{*}):= (e^{*}(x_{i}))_{1\leq |i|\leq n}^{t},\quad\phi_{n}^{*}(b_{i})_{1\leq |i|\leq n}^{t}:=\sum_{1\leq |i|\leq n} b_{i}x_{i}.\]

\end{remark}
There is an analogue of adjointable operators on rigged modules. Their definition is straightforward. 
\begin{definition}[\cite{Blech2}]\label{rigad} A completely bounded operator $T:E\rightarrow F$ between rigged modules is called \emph{adjointable} if there exists an operator $T^{*}:F^{*}\rightarrow E^{*}$ such that
\[\forall e\in E, f^{*}\in F^{*}:\quad \langle f^{*}, Te \rangle=\langle T^{*}f^{*},e\rangle.\]
Here we used the suggestive notation $\langle f^{*}, Te \rangle$ for $f(Te)$.
The space of adjointable operators from $E$ to $F$ is denoted $\End^{*}_{\mathcal{B}}(E,F)$.
\end{definition}
When $\mathcal{B}$ has a contractive approximate unit it is a rigged module over itself, and $\K_{\mathcal{B}}(\mathcal{B})\cong\mathcal{B}$ completely isometrically. The compact and adjointable operators satisfy the usual relation $\End^{*}_{\mathcal{B}}(E)=\mathscr{M}(\K_{\mathcal{B}}(E))$, where $\mathscr{M}$ denotes the multiplier algebra. We take this as the definition of $\mathscr{M}(\mathcal{B})$. 
Given an operator algebra and a completely contractive algebra homomorphism $\mathcal{A}\rightarrow \Endst_{\mathcal{B}}(E)$, $E$ is an $(\mathcal{A},\mathcal{B})$ rigged bimodule. As can be expected from theorem \ref{int}, the Haagerup tensor product of rigged modules behaves like the interior tensor product of $C^{*}$-modules.

\begin{theorem}[\cite{Blech2}]\label{rigint}Let $E$ be a right $\mathcal{B}$- rigged module and $F$ an $(\mathcal{B},\mathcal{C})$  rigged bimodule. Then $E\tildeotimes_{\mathcal{B}} F$ is a $\mathcal{C}$- rigged module and
$\K_{\mathcal{C}}(E\tildeotimes_{\mathcal{B}} F)\cong E\tildeotimes_{\mathcal{B}}\K_{\mathcal{C}}(F)\tildeotimes_{\mathcal{B}} E^{*}$ completely isometrically.
\end{theorem}

For our purposes, we are only considered with countably generated $C^{*}$-modules. The particular form of the approximate unit \eqref{approx} implies the maps $\phi_{n},\psi_{n}$ from \eqref{rig} can be assembled into two maps
\[\phi:\mathpzc{E}\rightarrow \mathpzc{H}_{B},\quad\psi:\mathpzc{H}_{B}\rightarrow \mathpzc{E},\]
given by $\phi(e)=(\langle x_{i},e\rangle)_{i\in\Z}$ and $\psi(b_{i})_{i\in\Z}=\sum_{i\in\Z}x_{i}b_{i}$. Then we have $\psi\phi=\id$ and $\phi\psi$ is a projection. In \cite{Blech2}, these modules are called CCGP (countably column generated projective) modules. As noted by Blecher in \cite{Blech2}, rigged modules seem too restrictive for $K$-theoretic considerations, as it is unlikely that every finite projective module over an operator algebra may be rigged. However, if we allow the maps $\phi_{n},\psi_{n}$  from definition \ref{rigged} to be completely bounded, we obtain a theory that is flexible enough. \newline\newline
Let $\mathpzc{H}:=\ell^{2}(\Z\setminus\{0\})\cong \ell^{2}(\N)\oplus\ell^{2}(\N)$ be an infinite dimensional separable graded Hilbert column space and $\mathcal{B}$ a graded operator algebra. Then the $\mathpzc{H}_{\mathcal{B}}:=\mathpzc{H}\tildeotimes\mathcal{B}$ is the \emph{standard rigged module} over
$\mathcal{B}$.
\begin{definition} A right $\mathcal{B}$ operator module $E$ is \emph{stably rigged} if there are completely bounded maps $\phi:E\rightarrow \mathpzc{H}_{\mathcal{B}}$ and $\psi:\mathpzc{H}_{\mathcal{B}}\rightarrow E$ such that $\psi\phi=\id$.
\end{definition}
A stably rigged module need not be rigged itself. This will be the case if $\phi,\psi$ can be chosen completely contractive. For this reason, we will always consider stably rigged modules \emph{up to cb-isomorphism}. In general a stably rigged module is a completely bounded direct summand in $\mathpzc{H}_{\mathcal{B}}$, which is an actual rigged module. The maps $\phi_{n},\psi_{n}$ defined by composing $\phi$ and $\psi$ with the projections $\mathpzc{H}_{\mathcal{B}}\rightarrow \mathcal{B}^{n}$ and inclusions $\mathcal{B}^{n}\rightarrow \mathpzc{H}_{\mathcal{B}}$ will be uniformly completely bounded as opposed to completely contractive. They can be used to define the algebras $\K_{\mathcal{B}}(E)$ and $\Endst_{\mathcal{B}}(E)$ as above. In the presence of a countable approximate unit $\{u_{n}\}\subset \mathcal{B}$, $\mathcal{B}$ is stably rigged over itself and $\K_{\mathcal{B}}(\mathcal{B})\cong\mathcal{B}$ completely boundedly.
\begin{definition}\label{mult} The \emph{multiplier algebra} of an operator algebra $\mathcal{B}$ with a countable approximate unit is $\mathscr{M}(\mathcal{B}):=\Endst_{\mathcal{B}}(\mathcal{B})$.
\end{definition} 
Note that this defines $\mathscr{M}(\mathcal{B})$ up to cb-isomorphism, which suffices for our purposes.
\begin{theorem}\label{stabten} Let $E$ be a stably rigged $\mathcal{B}$-module and $F$ a stably rigged $\mathcal{C}$ module. Given a completely bounded algebra homomorphism $\pi:\mathcal{B}\rightarrow\Endst_{\mathcal{C}}(F)$, the Haagerup tensor product $E\tildeotimes_{\mathcal{B}}F$ is a stably rigged module and $\K_{\mathcal{C}}(E\tildeotimes_{\mathcal{B}} F)\cong E\tildeotimes_{\mathcal{B}}\K_{\mathcal{C}}(F)\tildeotimes_{\mathcal{B}} E^{*}$ completely boundedly. Moreover, if $\mathcal{C}=C$ is a $C^{*}$-algebra, then both $F$ and $E\tildeotimes_{\mathcal{B}}F$ are completely isomorphic to $C^{*}$-modules.
\end{theorem} 
\begin{proof}We only prove the last statement. First note that the module $F$ is completely isomorphic to $p\mathpzc{H}_{C}$, with $p=\phi\psi$ an idempotent in $\Endst_{C}(\mathpzc{H}_{C})$, which is a $C^{*}$-module. Secondly, denote by $\tilde{\mathcal{B}}$ the algebra $\mathcal{B}$ with the completely isomorphic operator space structure given by the representation
\[\id\oplus\pi:\mathcal{B}\rightarrow\mathcal{B}\oplus\Endst_{C}(\mathcal{F}).\] 
Then $\pi:\tilde{\mathcal{B}}\rightarrow\Endst_{\mathcal{C}}(F)$ is completely contractive, and $\mathpzc{H}_{\tilde{\mathcal{B}}}$ remains rigged for this operator space structure, and is completely isomorphic to $\mathpzc{H}_{\mathcal{B}}$. Thus, 
we see that $E\tildeotimes_{\mathcal{B}}\mathpzc{F}$ is completely isomorphic to a $C^{*}$-module, by theorem \ref{rigint}. 
\end{proof}

The next theorem shows that the Haagerup tensor product of stably rigged modules behaves well with respect to adjointable operators.
\begin{theorem}[cf. \cite{Blech2}]\label{tensop} Let $E,E'$ be stably rigged $\mathcal{B}$-modules, $F,F'$ stably rigged $(\mathcal{B},\mathcal{C})$- bimodules. If $S\in\Endst_{\mathcal{B}}(E,E')$, $T\in\Endst_{\mathcal{C}}(F,F')$, and $T$ is also a left $\mathcal{B}$-module map, then $S\otimes T\in\Endst_{\mathcal{C}}(E\tildeotimes_{\mathcal{B}} F,E'\tildeotimes_{\mathcal{B}} F')$. Moreover the map $S\mapsto S\otimes 1$ is a completely bounded algebra homomorphism.
\end{theorem}
The direct sum of a family $\{E_{\alpha}\}$ of rigged modules is canonically defined in \cite{Blech2}. This is done by embedding the algebra $\mathcal{B}$ isometrically in a $C^{*}$-algebra $B$. The modules $E_{\alpha}$ are completely isometrically isomorphic to closed submodules of the $C^{*}$-modules $\mathpzc{E}_{\alpha}:=E_{\alpha}\tildeotimes_{\mathcal{B}}B$, and $\bigoplus E_{\alpha}$ is constructed as the natural closed submodule of the $C^{*}$-direct sum $\bigoplus_{\alpha}\mathpzc{E}_{\alpha}$. For stably rigged modules the situation is slightly more complicated.
\begin{definition}\label{directsum} Let $\{E_{\alpha}\}_{\alpha\in\N}$ be a countable family of stably rigged modules, with structural maps $\phi_{\alpha}:E_{\alpha}\rightarrow \mathpzc{H}_{\mathcal{B}}$ and $\psi_{\alpha}:\mathpzc{H}_{\mathcal{B}}\rightarrow E_{\alpha}$. Suppose \[\sup_{\alpha}\{ \|\psi_{\alpha}\|_{cb},\|\phi_{\alpha}\|_{cb}\}<\infty.\] The direct sum $E:=\bigoplus_{\alpha} E_{\alpha}$ is defined up to cb-ismorphism by identifying it with the submodule $\bigoplus_{\alpha}\phi_{\alpha}(E_{\alpha})\subset\bigoplus_{\alpha}\mathpzc{H}_{\mathcal{B}}$. The maps \[\phi:=\bigoplus\phi_{\alpha}:\bigoplus_{\alpha} E_{\alpha}\rightarrow\mathpzc{H}_{\mathcal{B}},\]
\[\psi:=\bigoplus\psi_{\alpha}:\mathpzc{H}_{B}\rightarrow \bigoplus_{\alpha}E_{\alpha},\]
give a completely bounded factorization of the identity, making $E$ into a stably rigged module.
\end{definition}
Note that we used the isomorphism $\bigoplus_{\alpha\in\N}\mathpzc{H}_{\mathcal{B}}\cong\mathpzc{H}_{\mathcal{B}}$ in the definition of the maps $\phi,\psi$. The choice of operator space structure on the direct sum is natural, but that it is more natural to think about the direct sum as being defined only up to complete isomorphism. For our purposes this suffices.
\section{Smoothness}
We adopt the philosophy that spectral triples should be a source of smooth structures $C^{*}$-algebras. The most important feature of a smooth subalgebra is stability under holomorphic functional calculus, implying $K$-equivalence. We will show our smooth
algebras satisfy this property. Moreover, we show that regular spectral triples \cite{Conspec} are smooth in our sense, providing us with numerous examples. Subsequently, we turn to the notion of a smooth $C^{*}$-module over a $C^{*}$-algebra equipped
with a smooth structure. All operator algebras are assumed to have a completely bounded countable approximate unit.

\subsection{Sobolev algebras}
We construct now a nested sequence of algebras
\[\dots\subset\mathcal{A}_{i+1}\subset\mathcal{A}_{i}\subset\mathcal{A}_{i-1}\subset\cdots\subset\mathcal{A}_{1}\subset A,\]
for any graded $(A,B)$-bimodule $\mathpzc{E}$ equipped with an odd selfadjoint regular operator $D$. Each $\mathcal{A}_{i}$ will admit a completely contractive representation on the $i$-th Sobolev module of $D$.

The representation $\pi_{1}:\mathcal{A}_{1}\rightarrow M_{2}(\Endst_{B}(\mathpzc{E}))$ (equation \ref{Der}), associated to an  $(A,B)$-bimodule $\mathpzc{E}$ equipped with an odd regular operator $D$, induces a representation \begin{eqnarray}\nonumber\mathcal{A}_{1}&\rightarrow&\Endst_{B}(\mathfrak{G}(D))\\
\nonumber a&\mapsto &p\pi(a)p,\end{eqnarray} 
with $p=p^{D}$ the Woronowicz projection. This is an algebra homomorphism due to the identity $p\pi_{1}(a)p=\pi_{1}(a)p$. From this it follows that 
\[\begin{split}\mathcal{A}_{1}&\rightarrow\Endst_{B}(v\mathfrak{G}(D))\\
a &\mapsto p^{\perp}\pi(a)p^{\perp},\end{split}\] 
where $p^{\perp}:=1-p$, is a homomorphism as well. Thus we can define a map
\begin{eqnarray}\nonumber\theta_{1}:\mathcal{A}_{1}&\rightarrow &M_{2}(\Endst_{B}(\mathpzc{E}))\\
a&\mapsto &p\pi_{1}(a)p+p^{\perp}\pi_{1}(a)p^{\perp}.\nonumber\end{eqnarray}
Recall from the discussion preceding proposition \ref{D2}, that the natural grading to consider on  $\bigoplus_{j=1}^{2^{i+1}}\mathpzc{E}$ is defined inductively by 
\[\gamma_{i+1}:=\begin{pmatrix}\gamma_{i} &0\\0&-\gamma_{i}\end{pmatrix}.\]
\begin{definition} Let $\mathcal{A}_{1}$, $\pi_{1}$ and $\theta_{1}$ be as above. 
For $i>0$, abusively denote by $D$ the odd selfadjoint regular operator on $\bigoplus_{j=1}^{2^{i}}\mathpzc{E}$ given by the diagonal action of $D$, and by $p_{i}$ its Woronowicz projection. For $i<k$, $p_{i}$ will denote the corresponding diagonal matrix in $\bigoplus_{j=1}^{2^{k}}\mathpzc{E}$. 
Inductively define
\begin{equation}\label{alg}\mathcal{A}_{i+1}:=\{a\in\mathcal{A}_{i}:[D,\theta_{i}(a)]\in\Endst_{B}(\bigoplus_{j=1}^{2^{i}}\mathpzc{E})\},\end{equation}
\begin{eqnarray}\label{prep}\pi_{i+1}:\mathcal{A}_{i+1}&\rightarrow & M_{2^{i+1}}(\Endst_{B}(\mathpzc{E}))\\
\nonumber a &\mapsto &\begin{pmatrix}\theta_{i}(a) & 0\\
[D,\theta_{i}(a)] & \gamma_{i}\theta_{i}(a)\gamma_{i}\end{pmatrix},\end{eqnarray}
\begin{eqnarray}\label{trep}\theta_{i+1}:\mathcal{A}_{i+1}&\rightarrow & M_{2^{i+1}}(\Endst_{B}(\mathpzc{E}))\\
\nonumber a &\mapsto &p_{i+1}p_{i}\pi_{i+1}(a)p_{i}p_{i+1}+p_{i+1}^{\perp}p_{i}^{\perp}\pi_{i+1}(a)p_{i}^{\perp}p_{i+1}^{\perp}\end{eqnarray}
\end{definition}
The notion of smoothness introduced in the next section will entail that the $\mathcal{A}_{i}$'s are dense in $A$. In the current section, no such assumption is present. We will refer to $\mathcal{A}_{i}$ as the $i$-th \emph{Sobolev subalgebra} of $A$. In case $A=\Endst_{B}(\mathpzc{E})$, we denote the $i$-th \emph{full Sobolev algebra of $D$} by $\Sob_{i}(D)$. Clearly, $\mathcal{A}_{i}=A\cap\Sob_{i}(D)$.
\begin{remark}Note that we have defined $\pi_{i}$ and $\theta_{i}$ on the same domain $\mathcal{A}_{i}$. 

However, a priori, we have \[\mathfrak{Dom} \pi_{i},\mathfrak{Dom} \pi_{i+1}\subset\mathfrak{Dom} \theta_{i}\subset\Endst_{B}(\mathpzc{E}).\]

It is important to think of these representations in this way when one considers density of the domains. 
\end{remark}

Taking $\pi_{0}$ to be the original representation of $A$ on $\mathpzc{E}$, the direct sums \begin{equation}\label{directpi}\pi_{[i]}=\bigoplus_{j=0}^{i}\pi_{j}:\mathcal{A}_{i}\rightarrow \bigoplus_{j=0}^{i}\Endst_{B}(\bigoplus_{k=1}^{2^{i}}\mathpzc{E}) \end{equation}give each $\mathcal{A}_{i}$ the structure of an operator space, and this
yields an inverse system of operator algebras
\[\cdots \rightarrow\mathcal{A}_{i+1}\rightarrow \mathcal{A}_{i}\rightarrow\mathcal{A}_{i-1}\rightarrow\cdots\rightarrow\mathcal{A}_{1}\rightarrow A,\]
in which all maps are completely contractive.\newline

Consider the unitaries \[v_{n+1}:=\begin{pmatrix} 0 & -I_{2^{n}}\\I_{2^{n}} & 0\end{pmatrix}\in\Endst_{B}(\bigoplus_{j=1}^{2^{n+1}}\mathpzc{E}),\] where $I_{2^{n}}$ is the $2^{n}\times 2^{n}$-identity matrix. For $j<i$ we identify $v_{i}$ with $v_{i}I_{2^{j}}$ and as such consider it as an element of $\Endst_{B}(\bigoplus_{i=1}^{2^{j}})$. As such, $v_{i}$ and $v_{k}$ commute for all $k,i\leq j$. 
For $i\in\N$, denote by $[i]$ the set $\{1,\cdots, i\}$ and by $\mathscr{P}([i])$ the powerset of $[i]$.
Define
\[v_{F}:=\prod_{j\in F}v_{j}\in M_{2^{i}}(\Endst_{B}(\mathpzc{E})),\]
which is well defined since the $v_{j}$ commute. Note that $v_{[0]}=v_{\emptyset}=1$.
\begin{proposition}\label{involutive} The $\mathcal{A}_{i}$ are involutive operator algebras.
\end{proposition}
\begin{proof}
 To prove that the involution $a\mapsto a^{*}$ is a complete anti isometry for the norm $\|\cdot\|_{i}$ (cf. definition \ref{opalg}) we show that 
\begin{equation}\label{inveven}\pi_{i}(a^{*})=v_{[i]}\pi_{i}(a)^{*}v_{[i]}^{*},\quad\theta_{i}(a^{*})=v_{[i]}\theta_{i}(a)^{*}v^{*}_{[i]},\quad i \textnormal{ even;}\end{equation}
\begin{equation}\label{invodd}\pi_{i}(a^{*})=v_{[i]}\gamma_{i}\pi_{i}(a)^{*}\gamma_{i}v_{[i]}^{*},\quad\theta_{i}(a^{*})=v_{[i]}\gamma_{i}\theta_{i}(a)^{*}\gamma_{i}v^{*}_{[i]},\quad i \textnormal{ odd}.\end{equation} 
In order to achieve this, recall that the grading on $\Endst_{B}(\bigoplus_{j=0}^{2^{i}}\mathpzc{E})$ is given by $T\mapsto \gamma_{i}T\gamma_{i}$, and hence that $[D,T]=DT-\gamma_{i}T\gamma_{i}D$. From this, it is immediate that 
\[(\gamma_{i}T\gamma_{i})^{*}=\gamma_{i}T^{*}\gamma_{i},\quad [D,T]^{*}=\gamma_{i}[D,T^{*}]\gamma_{i}=-[D,\gamma_{i}T^{*}\gamma] ,\]
which will be used in the computation below.\newline\newline
We have $v_{[0]}=1$ and the $v_{i}$ commute with $D$. For $\pi_{0}=\theta_{0}$, \ref{inveven} is trivial. Suppose \ref{inveven} holds for some even number $i$. Then,
\[\begin{split}\pi_{i+1}(a^{*})&=\begin{pmatrix}v_{[i]}\theta_{i}(a)^{*}v_{[i]}^{*}& 0 \\ [D, v_{[i]}\theta_{i}(a)^{*}v_{[i]}^{*}]&  v_{[i]}\gamma_{i}\theta_{i}(a)^{*}\gamma_{i}v_{[i]}^{*}\end{pmatrix}\\
&=\begin{pmatrix} 0 &-v_{[i]} \\ v_{[i]} &0 \end{pmatrix}\begin{pmatrix}\gamma_{i}\theta_{i}(a)^{*}\gamma_{i}& -[D, \theta_{i}(a)^{*}] \\ 0 &\theta_{i}(a)^{*}\end{pmatrix}\begin{pmatrix} 0 &v_{[i]}^{*} \\ -v_{[i]}^{*} &0 \end{pmatrix} \\
&=\begin{pmatrix} 0 &-v_{[i]} \\ v_{[i]} &0 \end{pmatrix}\begin{pmatrix}\gamma_{i}\theta_{i}(a)\gamma_{i}& 0  \\ -\gamma_{i}[D, \theta_{i}(a)]\gamma_{i} &\theta_{i}(a)\end{pmatrix}^{*}  \begin{pmatrix} 0 &v_{[i]}^{*} \\ -v_{[i]}^{*} &0 \end{pmatrix} \\
&=v_{[i+1]}\gamma_{i+1}\pi_{i+1}(a)^{*}\gamma_{i+1}v_{[i+1]}^{*}.\end{split}\]
Since $v_{[i]}$ commutes with $D$, we have $v_{[i+1]}p_{i+1}p_{i}v_{[i+1]}^{*}=p_{i+1}^{\perp}p_{i}^{\perp}$, and the projections $p_{i}, p_{i+1}$ are even. Thus, \ref{invodd} holds for $i+1$. \newline

Now suppose \ref{invodd} holds for some odd $i$. Note that for all $i$, $\gamma_{i}v_{[i]}=(-1)^{i}\gamma_{i}v_{[i]}$, i.e. $v_{[i]}$ is homogeneous of degree $i\mod 2$. Then, 
\[\begin{split}\pi_{i+1}(a^{*})&=\begin{pmatrix}v_{[i]}\gamma_{i}\theta_{i}(a)^{*}\gamma_{i}v_{[i]}^{*}& 0 \\ [D, v_{[i]}\gamma_{i}\theta_{i}(a)^{*}\gamma_{i}v_{[i]}^{*}]&  v_{[i]}\theta_{i}(a)^{*}v_{[i]}^{*}\end{pmatrix}\\
&=\begin{pmatrix} 0 &-v_{[i]} \\ v_{[i]} &0 \end{pmatrix}\begin{pmatrix}\theta_{i}(a)^{*}& \gamma_{i}[D, \theta_{i}(a)^{*}]\gamma_{i} \\ 0 &\gamma_{i}\theta_{i}(a)^{*}\gamma_{i}\end{pmatrix}\begin{pmatrix} 0 &v_{[i]}^{*} \\ -v_{[i]}^{*} &0 \end{pmatrix}\\
&=\begin{pmatrix} 0 &-v_{[i]} \\ v_{[i]} &0 \end{pmatrix}\begin{pmatrix}\theta_{i}(a)& 0  \\ [D, \theta_{i}(a)] &\gamma_{i}\theta_{i}(a)\gamma_{i}\end{pmatrix}^{*}\begin{pmatrix} 0 &v_{[i]}^{*} \\ -v_{[i]}^{*} &0 \end{pmatrix}\\
&=v_{[i+1]}\pi_{i+1}(a)^{*}v_{[i+1]}^{*}.\end{split}\]
Since $v_{[i]}$ commutes with $D$, we have $v_{[i+1]}p_{i+1}p_{i}v_{[i+1]}^{*}=p_{i+1}^{\perp}p_{i}^{\perp}$, and hence it follows that \ref{inveven} holds for $i+1$.
\end{proof}

\begin{proposition} For each $n\in\N$, there is a decomposition
\begin{equation}\label{decomp}\bigoplus_{i=1}^{2^{n}}\mathpzc{E}\cong\bigoplus_{F\in\mathscr{P}([n])}v_{F}
\mathfrak{G}(D_{n}),\end{equation}
and for $a\in\mathcal{A}_{n}$, $\theta_{n}(a)$ respects this decomposition. In fact it is nonzero only on $\mathfrak{G}(D_{n})$ and $v_{[n]}\mathfrak{G}(D_{n})$.
\end{proposition}
\begin{proof} The decomposition is proved by induction. Clearly it holds for $n=1$ (this is Woronowicz's theorem \ref{Wor}). Suppose we have the decomposition for $n=k$. Then
\[\bigoplus_{i=1}^{2^{k+1}}\mathpzc{E}\cong\bigoplus_{F\in\mathscr{P}([k])}v_{F}\mathfrak{G}(D_{k})\oplus\bigoplus_{F\in\mathscr{P}([k])}v_{F}\mathfrak{G}(D_{k}),\]
and since 
\[v_{F}(\mathfrak{G}(D_{k})\oplus\mathfrak{G}(D_{k}))\cong v_{F}(\mathfrak{G}(D_{k+1})\oplus v_{k+1}\mathfrak{G}(D_{k+1})),\]
we get the desired decomposition for $n=k+1$. To prove the $\mathcal{A}_{n}$-invariance, observe that for $n=1$, this holds by construction. Suppose the statement has been proven for $n=i$. The graph of $D$ as a diagonal operator in $\bigoplus_{i=1}^{2^{i}}\mathpzc{E}$ is a submodule of $ \bigoplus_{i=1}^{2^{i+1}}\mathpzc{E}$ and under the isomorphism \ref{decomp} it gets mapped to $\bigoplus_{F\in\mathscr{P}([i+1])}v_{F}\mathfrak{G}(D_{i+1})$. Thus, preservation of the decomposition \ref{decomp} is equivalent to preservation of the graph of $D$ and its complement. This is immediate from the definition of $\theta_{i+1}$.
\end{proof} 
\begin{corollary}\label{graphrep}
 Each $\mathcal{A}_{n}$ admits a completely contractive representation $\chi_{n}:\mathcal{A}_{n}\rightarrow\Endst_{B}(\mathfrak{G}(D_{n}))$.
\end{corollary}
\begin{proof}Denote by $p_{[n]}=\prod_{i=1}^{n}p_{i}\in\Endst_{B}(\bigoplus_{i=1}^{2^{n}}\mathpzc{E})$ the projection onto $\mathfrak{G}(D_{n})$. From the previous proposition it follows that 
\[\chi_{n}(a):=p_{[n]}\theta_{n}(a)p_{[n]}=\theta_{n}(a)p_{[n]},\]
and hence is a completely contractive algebra homomorphism.
\end{proof}
Note that in fact we have $\theta_{n}(a)=p_{[n]}\theta_{n}(a)p_{[n]}+v_{[n]}p_{[n]}v_{[n]}^{*}\theta_{n}(a)v_{[n]}p_{[n]}v_{[n]}^{*}$ for even $n$, and $\theta_{n}(a)=p_{[n]}\theta_{n}(a)p_{[n]}+v_{[n]}p_{[n]}v_{[n]}^{*}\gamma_{n}\theta_{n}(a)\gamma_{n}v_{[n]}p_{[n]}v_{[n]}^{*}$ for odd $n$.
\begin{corollary}\label{astar} $a\in\mathcal{A}_{n+1}$ if and only if $a\in\mathcal{A}_{n}$ and $[D,\chi_{n}(a)],[D,\chi_{n}(a^{*})]\in\Endst_{B}(\bigoplus_{i=1}^{2^{n}}\mathpzc{E})$.
\end{corollary}
\begin{proof} We have \[\chi_{n}(a^{*})=p_{[n]}v_{[n]}\theta_{n}(a)^{*}v_{[n]}^{*}p_{[n]},\] for even $n$ and \[\chi_{n}(a^{*})=p_{[n]}v_{[n]}\gamma_{n}\theta_{n}(a)^{*}\gamma_{n}v_{[n]}^{*}p_{[n]},\] for odd $n$. Therefore, for odd $n$
\[ \begin{split} \| [D,\theta_{n}(a)] \| &=\max \{\| [D,\chi_{n}(a)] \|,\|v_{[n]} [D,p_{[n]} v_{[n]}^{*}\gamma_{n}\theta_{n}(a)\gamma_{n} v_{[n]} p_{[n]}] v_{[n]}^{*}\| \}\\ &=\max \{\| [D,\chi_{n}(a)] \|,\|[D,\chi_{n}(a^{*})] \| \}.\end{split}\]
The same works for even $n$.
\end{proof}
Lastly, we note that the constructions associated with Sobolev algebras can be done for almost selfadjoint operators, using the nonselfadjoint idempotents from corollary \ref{idempotent}. The price for doing this is that the involution will not be completely isometric, but still a complete anti isomorphism. This is good enough for our purposes, and fits the idea of working with nonselfadjoint algebras and homomorphisms.

\subsection{Holomorphic stability}
Now we turn to spectral invariance of the $\mathcal{A}_{i}$. The following definition is a modification of \cite{BC}, definition 3.11:
\begin{definition} Let $\mathscr{A}$ be an algebra with Banach norm $\|\cdot\|$, and $\mathcal{A}$ its closure in this norm. A norm $\|\cdot\|_{\alpha}$ on $\mathscr{A}$ is said to be \emph{analytic} with respect to $\|\cdot\|$ if for each $x\in\mathcal{A}$, with $\|x\|<1$ we have
\[\limsup_{n\rightarrow\infty}\frac{\ln \|x^{n}\|_{\alpha}}{n}\leq 0.\]
\end{definition}
The reason for introducing the concept of analyticity is that analytic inclusions are spectral invariant.
\begin{proposition}[\cite{BC}] Let $\mathcal{A}_{\beta}\rightarrow \mathcal{A}_{\alpha}$ be a continuous dense inclusion of unital Banach algebras. If $\|\cdot\|_{\beta}$ is analytic with respect to $\|\cdot\|_{\alpha}$, then for
all $a\in\mathcal{A}_{\beta}$ we have $\Sp_{\beta}(a)=\Sp_{\alpha}(a)$.
\end{proposition}
\begin{proof} It suffices to show that if $x\in\mathcal{A}_{\beta}$ is invertible in $\mathcal{A}_{\alpha}$, then $x^{-1}\in\mathcal{A}_{\beta}$. To this end choose $y\in\mathcal{A}_{\beta}$ with $\|x^{-1}-y\|<
\frac{1}{2\|x\|_{\alpha}}$. Then $\|2-2xy\|_{\alpha}<1$. By analyticity, there exists $n$ such that $\|(2-2xy)^{n}\|_{\beta}<1$, and hence $2\notin\Sp_{\beta}(2-2xy)$. But then $0\notin\Sp_{\beta}(2xy)$, hence $2xy$ has an
inverse $u\in\mathcal{A}_{\beta}$. Therefore $x^{-1}=2yu$.\end{proof} 
In order to prove spectral invariance of the inclusions $\mathcal{A}_{i+1}\rightarrow\mathcal{A}_{i}$ we need the following straightforward result, whose proof we include for the sake of completeness.
\begin{lemma}\label{an}Let $\mathcal{A}$ be a graded Banach algebra and $\delta:\mathcal{A}_{\alpha}\rightarrow M$  a densely defined closed graded derivation into a Banach $\mathcal{A}$-bimodule $M$. Then $\|a\|_{\alpha}:=\|a\|+\|\delta(a)\|$ is analytic with
respect to $\|\cdot\|$.
\end{lemma}
\begin{proof} Let $\|x\|<1$. We have $\|\delta(x^{n})\|\leq n\|\delta(x)\|,$ by an obvious induction. Then
\[\begin{split}\limsup_{n\rightarrow\infty}\frac{\ln\|x^{n}\|_{\alpha}}{n}&=\limsup_{n\rightarrow\infty}\frac{\ln(\|x^{n}\|+\|\delta(x^{n})\|)}{n}\\
&\leq\limsup_{n\rightarrow\infty}\frac{\ln (1+n\|\delta(x)\|)}{n}\\
&\leq \limsup_{n\rightarrow\infty}\frac{\ln n}{n}+\frac{\ln( 1+\|\delta(x)\|)}{n}\\ &=0.\end{split}\]
\end{proof}
\begin{theorem}\label{hol}Let $(\mathpzc{E},D)$ be an unbounded $(A,B)$ bimodule. Then all inclusions $\mathcal{A}_{i+1}\rightarrow\mathcal{A}_{i}$ are spectral invariant, and hence the $\mathcal{A}_{i}$
are stable under holomorphic functional calculus in $A$.
\end{theorem}
\begin{proof} Observe that
\[\|a\|_{i+1}\leq \|a\|_{i}+\|[D,\theta_{i}(a)]\|,\]
thus, by lemma \ref{an}$ ,\|\cdot\|_{i+1}$ is majorized by a norm analytic with respect to $\|\cdot\|_{i}$, and hence is itself analytic with respect to $\|\cdot\|_{i}$. Now $\mathcal{A}_{1}$ is dense in its $C^{*}$-closure which is a $C^{*}$-subalgebra of $A$, so $\mathcal{A}_{1}$ is spectral invariant in $A$. Suppose now $\mathcal{A}_{i}$ is spectral invariant in $A$. By the above argument, $\mathcal{A}_{i+1}$ is spectral invariant in its $i$-closure, which is spectral invariant in $A$.
\end{proof}
\begin{corollary}\label{rangeproj} Let $q\in \mathcal{A}_{k}$ be an idempotent and
\[p:=qq^{*}(1+(q-q^{*})(q^{*}-q))^{-1}.\] Then $p\in\mathcal{A}_{k}$ and $p=p^{2}=p^{*}$ is a projection such that $pq=q$ and $qp=p$. In particular $q\mathcal{A}_{k}=p\mathcal{A}_{k}$.
\end{corollary}
\begin{proof} The element $(q-q^{*})(q^{*}-q)=(q-q^{*})(q-q^{*})^{*}$ is positive and hence $x=1+(q-q^{*})(q^{*}-q)$ is invertible. By theorem \ref{hol} $x^{-1}\in\mathcal{A}_{k}$ and thus $p\in\mathcal{A}_{k}$. We have
\[qq^{*}x=qq^{*}(1+(q-q^{*})(q^{*}-q))=(qq^{*})^{2}=(1+(q-q^{*})(q^{*}-q))qq^{*}=xqq^{*},\]
so $qq^{*}x^{-1}=x^{-1}qq^{*}$, which shows that $p^{*}=p$ and also
\[p^{2}=(qq^{*})^{2}x^{-2}=qq^{*}xx^{-2}=qq^{*}x^{-1}=p.\]
The identity $qp=p$ is immediate, and 
\[pq=(1+qq^{*}+q^{*}q-q^{*}-q)^{-1}qq^{*}q=(1-(1+qq^{*}+q^{*}q-q^{*}-q)^{-1})(1+q^{*}q-q^{*}-q)q=q.\]
\end{proof}

In the sequel, by a \emph{$C^{k}$-structure} on a $C^{*}$-algebra $A$ we shall mean an inverse system of operator algebras
\[\mathcal{A}_{k}\rightarrow\mathcal{A}_{k-1}\rightarrow\cdots\rightarrow A\]
where the maps are spectral invariant completely bounded *-homomorphisms with dense range. 
\begin{definition}\label{Ckalg} Let $A$ and $B$ be $C^{*}$-algebras, $\mathpzc{E}$ be an $(A,B)$ bimodule, $D$ a selfadjoint regular operator in $\mathpzc{E}$ and $k>0$. The pair $(\mathpzc{E},D)$ is said to be \emph{$C^{k}$} if the subalgebra $\mathcal{A}_{k}$ (\ref{alg}) is dense in $A$ and there is a countable positive increasing approximate unit $u_{n}$ such that $\sup_{n}\|u_{n}\|_{k}<\infty$. The bimodule $(\mathpzc{E},D)$ is \emph{smooth} if it is $C^{k}$ for all $k$. A \emph{$C^{k}$-algebra} shall be a $C^{*}$-algebra equipped with a fixed $C^{k}$-spectral triple. As such it has a natural $C^{k}$-structure.
\end{definition}
Note that if a module is $C^{k}$ for some $k$, then it is $C^{i}$ for all $i\leq k$. In particular $KK$-cycles are $C^{1}$ by definition. The above notion of smoothness is weaker than the one defined \cite{Conspec}. We refer to the appendix for a proof of this. In what follows (especially in section 6) it is crucial that we work relative to a fixed spectral triple. Notice the parallel with the definition of a manifold as a topological space with extra structure defined on it.

\subsection{Smooth $C^{*}$-modules}
We will define $C^{k}$-structures on $C^{*}$-modules over a $C^{k}$-algebra by requiring the existence of an appropriate approximate unit. We use this to construct a chain of stably rigged submodules
\[ E^{k}\subset E^{k-1}\subset \cdots \subset E^{1}\subset \mathpzc{E},\]
up to the smoothness degree of the module. Then we show that the smooth structure is compatible with tensor products, and we address the case of nonunital algebras.
\begin{definition}\label{smoothmodule} Let $B$ be a smooth $C^{*}$-algebra, with smooth structure $\{\mathcal{B}_{i}\}.$ A $C^{*}$-$B$-module $\mathpzc{E}$ is a $C^{k}$-$B$-module, if there is an approximate unit 
\[u_{n}:=\sum_{1\leq |i| \leq n}x_{i}\otimes x_{i} \in\Fin_{B}(\mathpzc{E}),\]
with $x_{i}$ homogeneous elements such that the matrices $(\langle x_{i},x_{j}\rangle )\in M_{n}(\mathcal{B}_{k})$, and 
\[\|(\langle x_{i},x_{j}\rangle)\|_{k}\leq C_{k}.\] It is a \emph{smooth}
$C^{*}$-module if there is such an approximate unit that makes it a $C^{k}$-module for all $k$.\end{definition}
From this definition, the definition of a nonunital smooth $C^{*}$-algebra is forced. In order that $B$ be smooth over itself, the existence of a positive, contractive approximate unit that restricts to a bounded one in each $\mathcal{B}_{k}$ is required. This is in line with definition \ref{Ckalg}. 
\begin{proposition}\label{useful} Let $B$ be a  $C^{k}$-algebra and $\mathpzc{E}$ a smooth $C^{*}$-$B$-module, with corresponding approximate unit $u_{n}:=\sum_{1\leq |i| \leq n}x_{i}\otimes x_{i}.$ Then
\[E^{k}:=\{e\in\mathpzc{E}:\langle x_{i},e\rangle\in\mathcal{B}_{k},\quad\|(\langle x_{i},e\rangle)_{i\in\Z}\|_{k}<\infty\},\]
is a stably rigged $\mathcal{B}_{k}$-module. When $C\leq 1$, it is an actual rigged module. Moreover, the inclusions $E^{k+1}\rightarrow E^{k}$ are completely contractive with dense range, and $E^{k+1}\tildeotimes_{\mathcal{B}_{k+1}}\mathcal{B}_{k}\cong E^{k}$, completely boundedly. When $C\leq 1$, this isomorphism is completely isometric.
\end{proposition}
\begin{proof} The maps 
\[\begin{split}\phi :E^{k} &\rightarrow \mathpzc{H}_{\mathcal{B}_{k}}\\
e &\mapsto  (\langle x_{i},e\rangle)_{i\in\Z\setminus\{0\}},\end{split}\]
and
\[\begin{split} \psi:\mathpzc{H}_{\mathcal{B}_{k}} &\rightarrow  E^{k} \\
(b_{i})_{i\in\Z}&\mapsto  \sum_{i\in\Z\setminus\{0\}}x_{i}b_{i},\end{split}\]
will give the desired factorization of the identity. These maps are completely bounded of norm $\leq C$ for the matrix norms on $E^{k}$ given by
\[\|(e_{ij})\|_{k}:=\|(\phi^{k}(e_{ij}))\|_{k},\]
and $E^{k}$ is (by definition) complete in these matrix norms. To check that $E^{k}$ is a stably rigged-$\mathcal{B}_{k}$-module, we have to show that 
\[\sum_{n\leq |i|\leq m}x_{i}\langle x_{i},e\rangle \rightarrow 0,\]
in $k$-norm, as $n\rightarrow\infty$. 
\[\begin{split}\|\sum_{n\leq |i|\leq m}x_{i}\langle x_{i},e\rangle\|_{k}&=\|(\sum_{n\leq |i|\leq m}\langle x_{j},x_{i}\rangle\langle x_{i},e\rangle)_{j\in\Z}\|_{k}\\
&=\|(\langle x_{j},x_{\ell}\rangle)_{j,\ell\in\Z}(\langle x_{i},e\rangle)_{n\leq |i|\leq m}\|_{k} \\
&\leq C\|(\langle x_{i},e\rangle)_{n\leq |i|\leq m}\|_{k}\rightarrow 0,\end{split}\]
because $\| (\langle x_{i},e \rangle )_{i\in\Z} \|_{k}<\infty$.
To see that $E^{k}$ is dense in $\mathpzc{E}$, it suffices to show that all the $x_{j}$ are in $E^{k}$, because they form a generating set for $\mathpzc{E}$. Thus we have to show that $\|x_{j}\|_{k}<\infty$. To this end we may assume that $\mathcal{B}_{k}$ is unital, and we denote by $\{e_{j}\}_{j\in\Z}$ the standard orthonormal basis of $\mathpzc{H}_{\mathcal{B}_{k}}$.
\[\begin{split}\|x_{j}\|_{k}&=\|\phi(x_{j})\|_{k}\\
&=\|(\langle x_{i},x_{j}\rangle)_{i\in\Z}\|_{k}\\
&=\|(\langle x_{i},x_{\ell}\rangle)_{i,\ell\in\Z}\cdot e_{j}\|_{k}\\
&\leq\|(\langle x_{i},x_{\ell}\rangle)_{i,\ell\in\Z}\|_{k} \\
&\leq C.
\end{split}\]
For the last statement,  the isomorphism will be implemented by the
multiplication map
\[\begin{split} m:E^{k+1}\tildeotimes_{\mathcal{B}_{k+1}}\mathcal{B}_{k}&\rightarrow E^{k}\\
e\otimes b &\mapsto eb.\end{split}\]
Write $p_{k}=\phi_{k}\psi_{k}$. The map $m:E^{k+1}\tildeotimes_{\mathcal{B}_{k+1}}\mathcal{B}_{k}\rightarrow E^{k}$ fits into a commutative diagram 
\begin{diagram}E^{k+1}\tildeotimes_{\mathcal{B}_{k+1}}\mathcal{B}_{k} &\rTo & p_{k+1}\mathpzc{H}_{\mathcal{B}_{k+1}}\tildeotimes_{\mathcal{B}_{k+1}}\mathcal{B}_{k} \\
\dTo & & \dTo\\
E^{k} &\rTo & p_{k}\mathpzc{H}_{\mathcal{B}_{k}},\end{diagram}
in which all other arrows are complete isomorphisms.
\end{proof}
\begin{remark}\label{different} There may very well be other approximate units satisfying definition \ref{smoothmodule}. They need not define the same $C^{k}$ -submodules. Two $C^{k}$ -approximate units $u_{n}=\sum_{1\leq|i|\leq n}x_{i}\otimes x_{i}$ and $v_{n}=\sum_{1\leq|i|\leq n}y_{i}\otimes y_{i}$ are equivalent if the matrix $(\langle x_{i},y_{j}\rangle)$ has finite $k$-norm. In this case, $u_{n}$ and $v_{n}$ define the same $C^{i}$-submodules, $i\leq k$, and the operator space topologies from proposition \ref{useful} are cb-isomorphic. Therefore we think of the  $C^{k}$-submodules up to cb-isomorphism.
\end{remark}
Now that we have constructed $C^{k}$-submodules as stably rigged modules, they come with canonical endomorphism algebras. This allows for a definition of $C^{k}$-bimodule.
\begin{definition}\label{bismooth} Let $A,B$ be $C^{k}$-algebras, $\mathpzc{E}\leftrightharpoons B$ a  $C^{k}$-module and $A\rightarrow\Endst_{B}(\mathpzc{E})$ a *-homomorphism. $\mathpzc{E}$ is a \emph{ $C^{k}$-$(A,B)$-bimodule} if the $A$-module structure restricts to a completely bounded homomorphism $\mathcal{A}_{k}\rightarrow\Endst_{\mathcal{B}_{k}}(E^{k})$.
\end{definition}
Note that a $C^{k}$-bimodule is automatically $C^{i}$ for $i\leq k$.       
\subsection{Inner products, stabilization and tensor products}
For a smooth $C^{*}$-algebra $B$ with smooth structure $\{\mathcal{B}_{i}\}$, any right rigged $\mathcal{B}_{i}$-module has a canonically associated left rigged $\mathcal{B}_{i}$-module $\tilde{E}$.  As a set, this is
\[\tilde{E}:=\{\overline{e}:e\in E\},\] equipped with the canonical conjugate linear structure and the left module structure $a\overline{e}:=\overline{ea^{*}}$. The left-stably rigged structure comes from the completely isometric anti isomorphism between row- and column modules
 \[\begin{split}\mathpzc{H}_{\mathcal{B}_{k}}&\rightarrow\mathpzc{H}_{\mathcal{B}_{k}}^{t}\\
(a_{j})&\mapsto(a_{j}^{*})^{t},\end{split}\] induced by the involution on $\mathcal{B}_{k}$. The structural maps are given by
\[\tilde{\phi}(\overline{e}):=(\phi(e)_{j}^{*})^{t}=(\langle e,x_{i}\rangle)_{i\in\Z},\quad \tilde{\psi}((b_{j})^{t}):=\overline{\psi((b_{j}^{*}))}=\overline{\sum_{\in\Z}x_{i}b_{i}^{*}}=\sum_{i\in\Z}b_{i}\overline{x}_{i},\] 
and are left-module maps having the desired properties.
\begin{lemma}\label{Riesz} Let $\mathpzc{E}$ be a smooth $C^{*}$-module over a smooth $C^{*}$-algebra $B$ with smooth structure $\{\mathcal{B}_{i}\}$. There is a cb-isomorphism of rigged modules $E^{i*}\cong \tilde{E}^{i}$ given by
restriction of the inner product pairing on $\mathpzc{E}$.
\end{lemma}
\begin{proof}The inner product on $\mathpzc{E}$ induces an injection $\tilde{E}^{k}\rightarrow E^{k*}$, which we denote $\overline{e}\mapsto e^{*}$.  For such elements we have \begin{equation}\label{gelijk}\phi^{*}(e^{*})=(\langle e,x_{i})\rangle_{i\in\Z}=\tilde{\phi}(\overline{e}).\end{equation} This follows from the definition of $\tilde{\phi}$ and remark \ref{dual}. An element $f\in E^{k*}$ is by definition a norm limit 
\[f=\lim_{n\rightarrow\infty}\sum_{i=-n}^{n}f(x_{i})x_{i}^{*},\]
and by \eqref{gelijk} the sequence $\sum_{i=-n}^{n}f(x_{i})\overline{x}_{i}$ is convergent in $\tilde{E}^{k}$. Therefore the map $\overline{e}\mapsto e^{*}$ is an isomorphism.
\end{proof}
As a consequence, $C^{k}$-modules over a $C^{k}$-algebra $\{\mathcal{B}_{k}\}$ are pre- $C^{*}$-modules, i.e. they come with a nondegenerate $\mathcal{B}_{k}$-valued innerproduct pairing satisfying all the properties of definition \ref{C*mod}. It should be noted that this inner product does not generate the operator space topology on $E^{k}$. Completing both $\mathcal{B}_{k}$ and $E^{k}$ yield the $C^{*}$-module $\mathpzc{E}\leftrightharpoons B$.

The type of self-duality expressed in lemma \ref{Riesz} allows us to remove the requirement of complete boundedness in the definition of adjointable operator (\ref{rigad}).
\begin{theorem}\label{automatic} Let $B$ be a $C^{k}$-algebra and $\mathpzc{E}\leftrightharpoons B$ a  $C^{k}$-module. If $T,T^{*}:E^{k}\rightarrow E^{k}$ are mappings satisfying $\langle Te,f\rangle=\langle e, T^{*}f\rangle$ for all $e,f\in E^{i}$, then $T,T^{*}$ are completely bounded  and $\mathcal{B}_{k}$-linear, i.e. $T,T^{*}\in \Endst_{\mathcal{B}_{k}}(E^{k})$. Moreover, the cb-norm and the operator norm are equivalent to one another and  $T\mapsto T^{*}$ is a well defined complete anti isomorphism of $\Endst_{\mathcal{B}_{k}}(E^{k})$.
\end{theorem}
\begin{proof} We first prove the statement for the case where $E^{k}$ is actually rigged. Uniqueness of the adjoint and $\mathcal{B}_{k}$-linearity are straightforward to show. To show $T,T^{*}$ are bounded, first note that lemma \ref{Riesz} implies that $\K_{\mathcal{B}_{k}}(E^{k},\mathcal{B}_{k})$ is anti isometric to $E^{k}$ via $e\mapsto e^{*}$. Now let $T,T^{*}$ be as stated in the theorem, and take $e\in E^{k}$ with $\|e\|_{k}=1$. Then $T_{e}:=(Te)^{*}\in \K_{\mathcal{B}_{k}}(E^{k},\mathcal{B}_{k})$ and
\begin{equation}\label{T}\|T_{e}(f)\|_{k}=\|\langle Te,f\rangle\|_{k}=\|\langle e, T^{*}f\rangle\|_{k}\leq\|T^{*}f\|_{k}.\end{equation}
From the Banach-Steinhaus theorem we conclude that the set
\[\{\|T_{e}\|_{k}:\|e\|_{k}=1\},\]
is bounded, which implies that $\|T\|_{k}<\infty$. By reversing $T$ and $T^{*}$, we find $\|T^{*}\|_{k}< \infty$ as well. Morever, now \ref{T} implies that $\|T\|\leq\|T^{*}\|$, and again, reversing gives $\|T\|=\|T^{*}\|$. Complete boundedness follows by estimating (cf. \cite{Blech2}, theorem 3.5)
\[\begin{split}\| (Te_{ij})\|_{k}&=\lim_{n}\|\psi_{n}\phi_{n}T\psi_{n}\phi_{n}(e_{ij})\|_{k}\\
&\leq (\sup_{n}\|\phi_{n}T\psi_{n}\|_{cb})\sup_{n}\|\psi_{n}(e_{ij})\|_{k}\\
&=(\sup_{n}\|\phi_{n}T\psi_{n}\|)\|(e_{ij})\|_{k}\\
&\leq \|T\|\|(e_{ij})\|_{k}.\end{split}\]
Here we used that $\phi_{n}T\psi_{n}:\mathcal{B}_{k}^{2n}\rightarrow \mathcal{B}_{k}^{2n}$ is completely bounded, which follows from the fact that it comes from left multiplication by a matrix, since $\mathcal{B}_{k}$ has a bounded approximate unit. Note that this estimate also shows $\|T\|_{cb}=\|T\|$.\newline\newline
For the general stably rigged case, embed $E^{k}$ in the rigged module $\mathpzc{H}_{\mathcal{B}_{k}}$, i.e. choose an isomorphism $E^{k}\cong p_{k}\mathpzc{H}_{\mathcal{B}_{k}}$, with $p_{k}\in\Endst_{\mathcal{B}_{k}}(\mathpzc{H}_{\mathcal{B}_{k}})$ a projection. The equalities $\|T\|=\|T\|_{cb}=\|T^{*}\|$, valid for $T\in \Endst_{\mathcal{B}_{k}}(\mathpzc{H}_{\mathcal{B}_{k}})$ then yield equivalences of these three norms for $\Endst_{\mathcal{B}_{k}}(E^{k})$.
\end{proof}
Note that this result implies that unitary operators (in the usual inner product sense), need not be isometric, but they will be cb-isomorphisms. Passing to an equivalent approximate unit (cf. remark \ref{different}) yields a unitary isomorphism of $C^{k}$-submodules.
\begin{theorem}\label{smoothstab} Let $B$ be a smooth graded $C^{*}$-algebra, and $\mathpzc{E}$ a countably generated smooth graded $C^{*}$-module. Then $\mathpzc{E}\oplus\mathpzc{H}_{B}$ is $C^{k}$ unitarily isomorphic to $\mathpzc{H}_{B}$. That is, there is a unitary isomorphism of graded inverse systems
\begin{diagram}\cdots &\rTo & E^{i+1}\oplus \mathpzc{H}_{\mathcal{B}_{i+1}} &\rTo & E^{i}\oplus \mathpzc{H}_{\mathcal{B}_{i}} &\rTo &\cdots &\rTo & \mathpzc{E}\oplus\mathpzc{H}_{B}\\
 & & \dTo & &\dTo & & & &\dTo \\
\cdots  &\rTo &\mathpzc{H}_{\mathcal{B}_{i+1}} &\rTo & \mathpzc{H}_{\mathcal{B}_{i}} &\rTo &\cdots &\rTo & \mathpzc{H}_{B}\end{diagram}
\end{theorem}
\begin{proof} The proof of lemma \ref{Riesz} shows that the map $\psi:E^{k}\rightarrow \mathpzc{H}_{\mathcal{B}_{k}}$ preserves the inner product. Let $p=\phi\psi$, so $\mathpzc{H}_{\mathcal{B}_{k}}\cong(1-p)\mathpzc{H}_{\mathcal{B}_{k}}\oplus p\mathpzc{H}_{\mathcal{B}_{k}}$ is an inner product preserving cb-isomorphism. Now use the Eilenberg swindle
\[E^{k}\oplus\mathpzc{H}_{\mathcal{B}_{k}}\cong E^{k}\oplus (1-p)\mathpzc{H}_{\mathcal{B}_{k}}\oplus (p\mathpzc{H}_{\mathcal{B}_{k}}\oplus (1-p)\mathpzc{H}_{\mathcal{B}_{k}}\oplus\cdots)\cong \mathpzc{H}_{\mathcal{B}_{k}},\]
to obtain an innerproduct preserving isomorphism. Note that the infinite direct sum is well defined cf.\ref{directsum}, since only two different modules appear in it.
\end{proof}
\begin{lemma}\label{standardsob} Let $B$ be a $C^{k}$-algebra with spectral triple $(\mathcal{H}, D)$. The algebras $\K_{\mathcal{B}_{k}}(\mathpzc{H}_{\mathcal{B}_{k}})$ and $\Endst_{\mathcal{B}_{k}}(\mathpzc{H}_{\mathcal{B}_{k}})$ are completely *-isomorphic to closed subalgebras of $\Sob_{k}(1\otimes D)$ of the selfadjoint operator $1\otimes D$ in $\mathpzc{H}\tildeotimes\mathpzc{H}$. In particular they are spectral invariant in their $C^{*}$-closures.
\end{lemma}\begin{proof} It is immediate that  $\K_{\mathcal{B}_{k}}(\mathpzc{H}_{\mathcal{B}_{k}})\cong\K\minotimes\mathcal{B}_{k}$ is a closed subalgebra of $\Sob_{k}(1\otimes D)$, since the Woronowics projections satisfy $p^{1\otimes D}=1\otimes p^{D}$. By theorem \ref{multipliersob} this extends to an involutive respresentation of $\mathscr{M}(\K_{\mathcal{B}_{k}}(\mathpzc{H}_{\mathcal{B}_{k}}))\cong\Endst_{\mathcal{B}_{k}}(\mathpzc{H}_{\mathcal{B}_{k}})$. \end{proof}

\begin{lemma}\label{approxim} Let $B$ be a $C^{k}$-algebra and $\mathpzc{E}\lrh B$ a $C^{*}$-module. Then $\mathpzc{E}$ is a $C^{k}$-module if and only if there is an approximate unit
\begin{equation}\label{approx2} u_{n}=\sum_{1\leq |i|\leq n} x_{i}\otimes x'_{i}\in\Fin_{B}(\mathpzc{E}),\end{equation}
such that $\|\langle x_{i}',x_{j}\rangle \|_{k}\leq C$.
\end{lemma}
\begin{proof} The implication $\Rightarrow$ is trivial, as $x_{i}=x'_{i}$ in definition \ref{smoothmodule}. For the other direction, note that $q=(\langle x_{i}',x_{j}\rangle)$ is an idempotent in $\Endst_{\mathcal{B}_{k}}(\mathpzc{H}_{\mathcal{B}_{k}})$, and by lemma \ref{standardsob} the range projection $p$ (corollary \ref{rangeproj}) is an element of $\Endst_{\mathcal{B}_{k}}(\mathpzc{H}_{\mathcal{B}_{k}})$ and $p\mathpzc{H}_{\mathcal{B}_{k}}=q\mathpzc{H}_{\mathcal{B}_{k}}$. Therefore $\mathpzc{E}$ is a $C^{k}$-module with approximate unit $u_{n}=\sum_{1\leq |i|\leq n} pe_{i}\otimes pe_{i}$, with $\{e_{i}\}$ the standard basis.
\end{proof}
\begin{proposition} Let $\mathpzc{E}\lrh B$ and $\mathpzc{F}\lrh C$ be $C^{k}$-modules with approximate units
\[u_{n}=\sum_{1\leq |i|\leq n}x_{i}\otimes x_{i}, \quad v_{n}=\sum_{1\leq |i|\leq n}y_{i}\otimes y_{i},\]
respectively. If $\pi:\mathcal{B}_{k}\rightarrow \Endst_{\mathcal{B}_{k}}(F^{k})$ is a completely bounded homomorphism, then $E^{k}\tildeotimes_{\mathcal{B}_{k}}\mathpzc{F}\lrh C$ is completely isomorphic to a $C^{k}$-module with approximate unit 
\[u_{n,m}=\sum_{1\leq |i|\leq n} \sum_{1\leq |j|\leq m}(x_{i}\otimes y_{j})\otimes (y_{j}\otimes x_{i}),\]
inner product
\begin{eqnarray}\label{innerprod}\langle e\otimes f,e'\otimes f'\rangle :&=&\lim_{n}\sum_{1\leq |i| \leq n}\langle\langle x_{i},e\rangle f,\langle x_{i},e'\rangle f'\rangle,\end{eqnarray}
and $C^{i}$-submodules cb-isomorphic to $E^{k}\tildeotimes_{\mathcal{B}_{k}}F^{i}$.
\end{proposition}
\begin{proof} Note that $y_{j}\otimes x_{i}$ denotes the functional $e\otimes f\mapsto \langle y_{j},\pi(\langle x_{i},e\rangle ) f\rangle$. The approximate unit $u_{n,m}$ defines a stably rigged structure on $E^{k}\tildeotimes_{\mathcal{B}_{k}}F^{k}$. (Note that strictly speaking, $u_{n,m}$ should be reindexed to bring it in the form \eqref{approx2}, and that this can be done without problems). The homomorphism $\mathcal{B}_{k}\rightarrow\Endst_{\mathcal{C}_{k}}(F^{k})$ in particular gives maps $\mathcal{B}_{k}\rightarrow \Endst_{\mathcal{B}_{i}}(F^{i})$ for all $i\leq k$, and hence stably rigged structures on each $E^{k}\tildeotimes_{\mathcal{B}_{k}}F^{i}$. From theorem \ref{stabten}, it follows that $E^{k}\tildeotimes_{\mathcal{B}_{k}}\mathpzc{F}$ is completely isomorphic to a $C^{*}$-module.  The inner product corresponding to $u_{n.m}$ is
\begin{eqnarray}\nonumber\langle e\otimes f,e'\otimes f'\rangle :&=&\lim_{n,m}\sum_{1\leq |i| \leq n}\sum_{1\leq |j| \leq m}\langle\langle x_{i},e\rangle f, y_{j}\rangle\langle
y_{j},\langle x_{i},e'\rangle f'\rangle \\ \nonumber &=&\lim_{n}\sum_{1\leq |i| \leq n}\langle\langle x_{i},e\rangle f,\langle x_{i},e'\rangle f'\rangle.\end{eqnarray}
This shows that  the functional $y_{j}\otimes x_{i}$ does not coincide with the functional defined by $x_{i}\otimes y_{j}$ via this innerproduct when $\pi$ is not a *-homomrphism. However, the approximate unit $u_{n,m}$ satisfies lemma \ref{approxim} , and the $C^{i}$-submodules are clearly cb-isomorphic to $E^{k}\tildeotimes_{\mathcal{B}_{k}}F^{i}$. 
\end{proof}

\subsection{Regular operators on $C^{k}$-modules}
To develop the theory of regular operators in $C^{k}$-modules, we need to strengthen definition \ref{reg} a little bit, due to the finer topology on such modules. Moreover, due to the absence of square roots we have to develop the theory for the selfadjoint case first.
\begin{definition}\label{Ckreg}Let $B$ be a $C^{k}$-algebra and $\mathpzc{E}\lrh B$ a $C^{k}$-module. \begin{enumerate}\item A closed densely defined selfadjoint operator $D:\mathfrak{Dom} D\rightarrow E^{k}$ is \emph{regular} if the operators $(D\pm i)^{-1}$ are densely defined and have finite $k$-norm.
\item A closed densely defined operator $D:\mathfrak{Dom} D\rightarrow E^{k}$ is \emph{regular} if $D^{*}$ is densely defined and the selfadjoint operator $\begin{pmatrix} 0& D^{*}\\ D& 0\end{pmatrix}$ is regular.
\end{enumerate}
\end{definition} 
Note that this definition in particular implies that $D\pm i$ have dense range. We wish to prove the analogue of the Woronowicz theorem \ref{Worthm} for such operators. Along the way we will find cleaner, equivalent characterizations of regularity, especially for selfadjoint operators. However, these seem to be harder to verify in practice.

\begin{proposition}\label{di}Let $B$ be a $C^{k}$-algebra and $\mathpzc{E}\leftrightharpoons B$ a  $C^{k}$-module. Suppose $D$ is a selfadjoint regular operator in $E^{k}$. Then $D^{2}$ is densely defined, $\mathfrak{Dom} D^{2}$ is a core for $D$, and the operators \[1+D^{2}:\Dom D^{2}\rightarrow E^{k},\]
\[D\pm i: \Dom D\rightarrow E^{k}\] are bijective.
\end{proposition}
\begin{proof} The operators $D\pm i$ have dense range, and by assumption their inverses extend to mutually adjoint elements $r_{+},r_{-}$ of $\Endst_{\mathcal{B}_{k}}(E^{k})$. Similarly, one gets adjointable extensions of $Dr_{+}$ and $r_{-}D$, which are adjoint to one another. We have that $\mathfrak{Im}r_{+}\subset\mathfrak{Dom} D$, by taking a sequence $e_{n}\in\im (D+i)$, such that $e_{n}\rightarrow e$. Then $r_{+}e_{n}\rightarrow r_{+}e$ and $Dr_{+}e_{n}\rightarrow Dr_{+}e$ and since $D$ is closed, $r_{+}e\in \Dom D$. Also $r_{+}(D+i)e=e$ and hence $\mathfrak{Im} r_{+}=\mathfrak{Dom}D$. The same holds for $r_{-}$. \newline
Now observe that for $e\in\Dom D$ and $f\in E^{k}$ 
\[\begin{split}\langle e,f\rangle &= \langle r_{\pm}(D\pm i)e,f\rangle \\
&=\langle e, (D\mp i)r_{\mp}f\rangle,\end{split}\]
so $f=(D\pm i)r_{\pm}f$ and $D\pm i$ are surjective. This in particular implies that $1+D^{2}$ is surjective.

Let $f\in\mathfrak{Dom} D$, then $f=r_{+}e$ for some $e\in E^{k}$. Choose a sequence $e_{n}$ in $\mathfrak{Im}(1+D^{2})=E^{k}$ with $r_{-}e_{n}\rightarrow e$. Then $r_{+}r_{-}e_{n}\rightarrow f$ and $Dr_{+}r_{-}e_{n}\rightarrow Df$. But $r_{+}r_{-}e_{n}\in\mathfrak{Dom} D^{2}$ since $e_{n}\in\mathfrak{Im}(1+D^{2})$, so $\mathfrak{Dom}D^{2}$ is a core for $D$ and in particular is dense.
\end{proof}
\begin{corollary}\label{smoothreg} Let $D$ be a closed densely defined operator in $E^{k}$ with densely defined adjoint. Then $D$ is regular if and only if $1+D^{*}D, 1+DD^{*}$ are surjective. If $D$ is regular then $D^{*}$ is regular and $\mathfrak{Dom} D^{*}D$ is a core for $D$.
\end{corollary}
\begin{proof} Consider the selfadjoint operator
\[\tilde{D}=\begin{pmatrix}0 & D^{*}\\D &0\end{pmatrix},\]
which is regular in $E^{k}\oplus E^{k}$ if and only if $D$ is regular in $E^{k}$.\newline\newline 
$\Rightarrow$ This is the statement that $(1+\tilde{D}^{2})$ is surjective.\newline\newline
$\Leftarrow$
Since
\[(\tilde{D}+i)(\tilde{D}-i)=\begin{pmatrix}1+D^{*}D & 0\\ 0 & 1+DD^{*}\end{pmatrix},\]
this implies that $(\tilde{D}\pm i)$ are bijectieve, and hence by Banach-Steinhaus their inverses are bounded and adjointable, hence $D$ is regular. The other statements are immediate from proposition \ref{di}
\end{proof}

\begin{theorem}\label{diffWor}Let $B$ be a smooth $C^{*}$-algebra and $\mathpzc{E}\leftrightharpoons B$ a smooth $C^{*}$-module. Suppose $D$ is a densely defined closed operator in $E^{k}$, with densely defined adjoint. Then $D$ is regular if and only if $\mathfrak{G}(D)\oplus v\mathfrak{G}(D^{*})\cong E^{k}\oplus E^{k}$ unitarily.
\end{theorem}
\begin{proof}We may assume that $D$ is selfadjoint, using the same trick as in corollary \ref{smoothreg}. The preceeding lemmas show that the operators $(1+D^{2})^{-1}, D(1+D^{2})^{-1}$ and $D^{2}(1+D^{2})^{-1}$ are selfadjoint elements of $\Endst_{\mathcal{B}_{k}}(E^{k})$. Therefore we can write down a Woronowicz projection
\[p_{D}:=\begin{pmatrix}(1+D^{2})^{-1} & D(1+D^{2})^{-1} \\ D(1+D^{2})^{-1} & D^{2}(1+D^{2})^{-1}\end{pmatrix}.\]
It maps $E^{k}\oplus E^{k}$ into $\mathfrak{G}(D)$. From the relation $(1+D^{2})^{-1}+D^{2}(1+D^{2})^{-1}=1$ it follows that $1-p_{D}$ maps $E^{k}\oplus E^{k}$ into $v\mathfrak{G}(D)$. Since these submodules are orthogonal, their sum must be all of $E^{k}\oplus E^{k}$.\newline

The converse follows by a standard argument as in \cite{Lan}. Let $p$ be the projection onto $\mathfrak{G}(D)$,
\[p=\begin{pmatrix} a& b^{*} \\ b & d\end{pmatrix}.\]
Then $\mathfrak{Im}a\subset\mathfrak{Dom} D, b=Da$ and $\mathfrak{Im}b\subset\mathfrak{Dom} D, 1-a=Db$. Thus, $\mathfrak{Im}a\subset\mathfrak{Dom} D^{2}$ and $1-a=D^{2}a$. Then $(1+D^{2})a=1$, so $(1+D^{2})$ is surjective and $D$ is regular.
\end{proof}
\begin{corollary}\label{selfreg} A densely defined closed symmetric operator in $E^{k}$ is selfadjoint and regular if and only if $\mathfrak{G}(D)\oplus v\mathfrak{G}(D) \cong E^{k}\oplus E^{k}$.
\end{corollary}
Combining this with proposition \ref{useful}, we see that regular operators in $E^{k}$ extend to $E^{i}\cong E^{k}\tildeotimes_{\mathcal{B}_{k}}\mathcal{B}_{i}$ for all $i\leq k$ as $D\otimes 1$, and this extension preserves selfadjointness.
\begin{corollary}[cf.\cite{Kuc2}, lemma 2.3] \label{invregular} If $D$ is a regular operator in $E^{k}$ such that $D$ and $D^{*}$ have dense range, then $D^{-1}$ is regular and $D^{-1*}=D^{*-1}$. In particular, if $S,T\in\Endst_{\mathcal{B}_{k}}(E^{k})$ have dense range, and adjoints with dense range, then $S^{-1}T^{-1}$ is regular with adjoint $T^{*-1}S^{*-1}$.
\end{corollary}
\begin{proof} This follows by observing that the unitary $v$ maps the graph of $D$ to that of $-D^{-1}$.
\end{proof}
\begin{theorem} \label{dplusmini} Let $D$ be a densely defined  closed symmetric operator in $E^{k}$. The following are equivalent:
\begin{enumerate} \item $D$ is selfadjoint and regular; 
\item The operators $D\pm i:\mathfrak{Dom} D\rightarrow E^{k}$ are bijective; 
\item $\mathfrak{Im}(D+i)\cap\mathfrak{Im}(D-i)$ is dense and the operators $(D\pm i)^{-1}$ have bounded $k$-norm.
\end{enumerate}
\end{theorem}
\begin{proof} $1.)\Rightarrow 2.)$ We already saw in proposition \ref{di} that for selfadjoint regular operators, $D\pm i$ are bijective. \newline
$2.)\Rightarrow 3.)$ Follows from theorem \ref{automatic}.\newline
$3.)\Rightarrow 1.)$
The extensions $r_{\pm}\in\Endst_{\mathcal{B}_{k}}(E^{k})$ of $(D\pm i)^{-1}$ are mutually adjoint because \[r_{+}^{*}e=(D+i)^{-1*}e=(D-i)^{-1}e=r_{-}e,\] 
for $e\in\mathfrak{Im}(D+i)\cap\mathfrak{Im}(D-i)$ and this subset is dense. One then shows as in proposition \ref{di} that in fact $(D\pm i)$ are bijective and thus $r_{\pm}=(D\pm i)^{-1}$. From corollary \ref{invregular} we then get that $D\pm i$ are regular and mutually adjoint, hence $D$ is selfadjoint and regular.
\end{proof}
A regular operator in $E^{k}$ is \emph {almost selfadjoint} if it satisfies the analogue of definition \ref{almostdef}, and the proves of proposition \ref{almost} and its corollary \ref{idempotent} go through verbatim. That is, for an almost selfadjoint operator and $\lambda\in \R$ sufficiently large, $D\pm\lambda i$ and $D^{*}\pm\lambda i$ are bijections $\Dom D\rightarrow E^{k}$ and the formula 
\[p=\begin{pmatrix}(1+\frac{D^{2}}{\lambda^{2}})^{-1} & \frac{D}{\lambda^{2}}(1+\frac{D^{2}}{\lambda^{2}})^{-1}\\
D(1+\frac{D^{2}}{\lambda^{2}})^{-1} &  \frac{D^{2}}{\lambda^{2}}(1+\frac{D^{2}}{\lambda^{2}})^{-1}\end{pmatrix},\]
defines an idempotent in $\Endst_{\mathcal{B}_{k}}(E^{k})$ with range $\mathfrak{Im} p=\mathfrak{G}(D)$, satisfying $vpv^{*}=1-p$.

\subsection{Transverse smoothness}
Regular operators in a $C^{k}$-module behave similarly to those in $C^{*}$-modules. In particular, their graphs are complemented submodules given by Woronowicz projections. This means that for a subalgebra $\mathcal{A}\subset\Endst_{\mathcal{B}_{k}}(E^{k})$, the representations $\pi_{n}$ (\ref{prep}), $\theta_{n}$(\ref{trep}) and $\chi_{n}$ (corollary \ref{graphrep}) can be defined, relative to a regular operator $D$ in $E^{k}$. They have the same properties as those in a $C^{*}$-module. Strictly speaking we should denote them by $\pi_{n}^{k}$, $\theta_{n}^{k}$ and $\chi_{n}^{k}$, but we suppress this in the notation, unless it causes confusion. In particular, Sobelev algebras $\Sob^{k}_{n}(D)$ and $\mathcal{A}_{n}:=\mathcal{A}\cap\Sob^{k}_{n}(D)$  are defined for a *-subalgebra $\mathcal{A}\subset \Endst_{\mathcal{B}_{k}}(E^{k})$.

We can construct a Sobolev chain $E^{k}_{j}$ for an almost selfadjoint $D$, and view it as a morphism of inverse systems, just as in corollary \ref{Sobolev} and the proposition preceeding it.
\begin{proposition} Let $\mathpzc{E}\leftrightharpoons B$ be a $C^{k}$-module and $D$ an almost selfadjoint regular operator in $E^{k}$.  Then the Sobolev modules $\mathpzc{E}_{n}$ of $D$ are $C^{k}$ over $B$.
\end{proposition}
\begin{proof}Let $u_{n}=\sum_{1\leq |i| \leq n}x_{i}\otimes x_{i}$ be a $C^{k}$- approximate unit for $\mathpzc{E}$. One checks that the map
\[\begin{split} E^{k}&\rightarrow \mathfrak{G}(D)\subset E^{k}\oplus E^{k}\\
e&\mapsto p\begin{pmatrix} ie\\e\end{pmatrix},\end{split}\]
preserves the inner product. Therefore
\[u_{n}^{1}:=\sum_{1\leq |j| \leq n}p\begin{pmatrix} ix_{j}\\x_{j}\end{pmatrix}\otimes p\begin{pmatrix} i x_{j}\\x_{j}\end{pmatrix},\]
satisfies the requirement of definition \ref{smoothmodule}, and hence smoothens (up to degree $k$) the first Sobolev module $\mathpzc{E}_{1}$. $D_{2}$, viewed as an operator in $\mathfrak{G}(D)=\mathpzc{E}_{1}$ is $C^{k}$ for this approximate unit: Since 
\begin{equation}\label{trick}\langle p\begin{pmatrix} i x_{j}\\x_{j}\end{pmatrix},\begin{pmatrix} e\\De\end{pmatrix}\rangle=\langle \begin{pmatrix} i x_{j}\\x_{j}\end{pmatrix},\begin{pmatrix} e\\De\end{pmatrix}\rangle,\end{equation}
the modules $E^{i}_{1}\subset E^{i}\oplus E^{i}$ and $\mathfrak{G}(D)^{i}$ coincide as submodules of $\mathpzc{E}\oplus \mathpzc{E}$, for $i\leq k$, and the notation $E^{i}_{1}$ is unambiguous. Therefore $D_{2}$ restricts to a selfadjoint regular operator in each $E^{i}_{1}$, $i\leq k$. Next we proceed by induction. Given that $\mathpzc{E}_{m}$ is smooth, we use its appoximate unit $u^{m}_{n}$ and the Woronowicz projection of $D_{n+1}$ to construct a $C^{k}$ approximate unit $u^{m+1}_{n}$ for $\mathpzc{E}_{n+1}=\mathfrak{G}(D_{n+1})$. By \ref{trick}, the modules $E^{i}_{n+1}$ and $\mathfrak{G}(D_{n+1})^{i}$ coincide, and since $D_{n+1}$ is selfadjoint regular in each $E^{i}_{n}$, so is $D_{n+2}$ in $E^{i}_{n+1}$.
\end{proof}
The situation of the previous proposition can be visualized in a diagram of completely contractive injections:
\begin{diagram}&     &\vdots &        &\vdots        &     & \vdots        &     &       &     & \vdots \\  
   &     &\dTo &        &\dTo        &     & \dTo         &     &       &     & \dTo \\
\cdots &\rTo & E^{j+1}_{i+1}&\rTo &E^{j+1}_{i} &\rTo & E_{i-1}^{j+1}&\rTo &\cdots &\rTo & E^{j+1}\\
                      &     &\dTo &        &\dTo        &     & \dTo         &     &       &     & \dTo \\
        \cdots &\rTo & E^{j}_{i+1}&\rTo &E^{j}_{i} &\rTo & E_{i-1}^{j}&\rTo &\cdots &\rTo & E^{j}\\    
               &     &\dTo &      &\dTo        &     & \dTo         &     &       &     & \dTo \\
        &     &\vdots &        &\vdots        &     & \vdots        &     &       &     & \vdots \\  
   &     &\dTo &        &\dTo        &     & \dTo         &     &       &     & \dTo \\
   \cdots &\rTo & \mathpzc{E}_{i+1}&\rTo &\mathpzc{E}_{i} &\rTo & \mathpzc{E}_{i-1}&\rTo &\cdots &\rTo & \mathpzc{E}.\end{diagram} 
We can now define the notion of transversely smooth $KK$-cycle.
\begin{definition} Let $A,B$ be  $C^{k}$-algebras, $\mathpzc{E}\leftrightharpoons B$ a  $C^{k}$-bimodule and $D$ a regular operator in $E^{k}$. The pair $(\mathpzc{E},D)$ is \emph{transverse $C^{k}$} if the subalgebras $\mathcal{A}_{i}$, are mapped completely boundedly into $\Sob^{i}_{i}(D)$ for all $i\leq k$. It is a $C^{k}$-cycle when $a(D\pm i)^{-1}\in\K_{\mathcal{B}_{i}}(E^{i})$ for $a\in\mathcal{A}_{i}, i\leq k$.
\end{definition}
Note that there are completely bounded injections $\Sob_{i}^{i}(D)\rightarrow\Sob^{i-1}_{i}(D)$, and $\Sob_{i}^{i}(D)\rightarrow\Sob^{i}_{i-1}(D)$ thus that transverse smoothness implies $\mathcal{A}_{k}$ gets mapped completely boundedly into $\Sob^{i}_{j}(D)$ for all $i,j\leq k$.
\subsection{Bounded perturbations}
 The following characterization of the the domain of the representations $\pi_{n}$ is interesting in itself. It is a relative boundedness condition. A slightly weaker form of this condition turns out to be sufficient for an operator $R\in\Endst_{\mathcal{B}_{k}}(E^{k})$ to belong to the domain of $\theta_{n}$. We use the notation $\textnormal{ad} D$ for the derivation $a\mapsto [D,a]$.
\begin{proposition} \label{smoothorder} Let $\mathpzc{E}$ be a $C^{k}$-module over the $C^{k}$-algebra $B$, $D$ a regular operator in $E^{k}$, and  $a\in \Endst_{\mathcal{B}_{k}}(E^{k})$. 
\begin{enumerate}\item  $a\in\Sob^{k}_{n}(D)$ if and only if $\forall m\leq n : (\textnormal{ad} (D)^{m}a)(D\pm i)^{-m+1}\in\Endst_{\mathcal{B}_{k}}(E^{k})$ and  $(D\pm i)^{-m+1}(\textnormal{ad} (D)^{m}a)\in\Endst_{\mathcal{B}_{k}}(E^{k})$. 
\item If $\forall m\leq n : (\textnormal{ad} (D)^{m}a)(D\pm i)^{-m}\in\Endst_{\mathcal{B}_{k}}(E^{k})$ and  $(D\pm i)^{-m}(\textnormal{ad} (D)^{m}a)\in\Endst_{\mathcal{B}_{k}}(E^{k})$, then $[D,\theta_{n-1}(a)](D\pm i)^{-1}, (D\pm i)^{-1}[D,\theta_{n-1}(a)]$ are adjointable and hence $\theta_{n}(a)\in\Endst_{\mathcal{B}_{k}}(\mathfrak{G}(D_{n})\oplus v_{[n]}\mathfrak{G}(D_{n}))$. That is, $a\in\mathfrak{Dom}\theta_{n}$.
\end{enumerate}

\end{proposition}
\begin{proof}
We only prove $(1)$, as $(2)$ can be done by the same method.\newline
$\Rightarrow$ For $n=1$ the statement reduces to the boundedness of the commutators $[D,a]$. Proceeding by induction, we assume the statement proven for $m\leq n.$ Let $a\in\mathcal{A}_{n+1}$, i.e. $a\in\mathcal{A}_{n}$ and $[D,\theta_{n}(a)]\in\Endst_{\mathcal{B}_{k}}(\bigoplus_{j=1}^{2^{n}}E^{k})$. We prove that $\textnormal{ad} (D)^{n+1}a(D\pm i)^{-n}\in\Endst_{\mathcal{B}_{k}}(E^{k})$. Since $\theta_{n}(a)$ is an orthogonal sum
\[\theta_{n}(a)=p_{n}\theta_{n}(a)p_{n}+p_{n}^{\perp}\theta_{n}(a)p_{n}^{\perp}=\theta_{n}(a)p_{n}+p_{n}^{\perp}\theta_{n}(a),\]
$[D,\theta_{n}(a)p_{n}]$ and $[D,p_{n}^{\perp}\theta_{n}(a)]$ are both bounded. Now since
\[\theta_{n}(a)p_{n}=\pi_{n}(a)p_{n-1}p_{n}=\begin{pmatrix} \theta_{n-1}(a) & 0 \\ [D,\theta_{n-1}(a)] & \theta_{n-1}(a)\end{pmatrix}p_{n-1}p_{n},\]
$[D,\theta_{n}(a)p_{n}]$ is bounded if and only if 
\[[D,[D,\theta_{n-1}p_{n-1}]](1+D^{2})^{-1},\quad\textnormal{and}\quad [D,[D,\theta_{n-1}p_{n-1}]]D(1+D^{2})^{-1},\]
are adjointable. This in turn is true if and only if $[D,[D,\theta_{n-1}(a)p_{n-1}]](D\pm i)^{-1}$ is adjointable. The same argument can now be applied another $n-1$-times to yield that
\[(\textnormal{ad}D)^{n+1}(a)(D\pm i)^{-n}\in\Endst_{\mathcal{B}_{k}}(E^{k}).\]
One proves that 
\[(D\pm i)^{-k+1}\textnormal{ad} (D)^{k}a\in\Endst_{\mathcal{B}_{i}}(E^{i}),\]
in the same way by using the summand $p_{n}^{\perp}\theta_{n}(a)p_{n}^{\perp}$.\newline\newline
$\Leftarrow$
Suposse that $a\in\mathcal{A}_{n}$. Then $\theta_{n}(a)$ is adjointable. The above method shows that $[D,\theta_{n}(a)]$ is adjointable whenever $[D,\theta_{n-1}(a)]$ and $[D,[D,\theta_{n-1}(a)]](D\pm i)^{-1}$ are adjointable. As above, this argument can be repeated to find that $\forall k\leq n : \textnormal{ad} (D)^{k}a(D\pm i)^{-k+1}\in\Endst_{\mathcal{B}_{k}}(E^{k})$ and  $(D\pm i)^{-k+1}\textnormal{ad} (D)^{k}a\in\Endst_{\mathcal{B}_{k}}(E^{k})$. \newline

\end{proof}
Bounded perturbations by elements in $\Dom \theta_{n}$ are well behaved with respect to taking Sobolev chains up to degree $n+1$. Equivalently, such perturbations preserve the domain op the powers of $D$ up to degree $n+1$.
\begin{definition} Let $\mathpzc{E},\mathpzc{F}$ be $C^{k}$-modules over a $C^{k}$-algebra $B$. A \emph{topological isomorphism of $C^{k}$-modules} is an invertible element $g\in\Homst_{\mathcal{B}_{k}}(E^{k},F^{k})$.
\end{definition}
Obviously, a topological isomorphism of $C^{k}$-modules induces topological isomorphisms of $C^{i}$-modules for $i\leq k$. 
\begin{lemma}\label{perturb} Let $D$ be a selfadjoint regular operator in the $C^{k}$-module $\mathpzc{E}$ and let $R\in \Endst_{\mathcal{B}_{k}}(E^{k})$. The map
\[\begin{split}g:\mathfrak{G}(D)&\rightarrow\mathfrak{G}(D+R)\\
 (e,(D+R)e)&\mapsto (e,De) ,\end{split}\]
is a topological isomorphism of $C^{k}$-modules.
\end{lemma}
\begin{proof}On $E^{k}\oplus E^{k}$, the map $g$ can be written as
\[g=p^{D}\begin{pmatrix} 1& 0\\ -R& 1\end{pmatrix}p^{D+R},\]
and hence it is an adjointable operator. Its inverse is
\[g^{-1}=p^{D+R}\begin{pmatrix} 1& 0\\ R& 1\end{pmatrix}p^{D}.\]
\end{proof}
When $R$ preserves the domain of $D^{m}$ for all $m\leq n$, we can inductively define maps \[g_{m}:\mathfrak{G}((D+R)_{m})\rightarrow\mathfrak{G}(D_{m})\] for $m\leq n+1$, by setting
\[g_{m+1}(e,(D+R)e):=(g_{m}(e),Dg_{m}(e)).\]

\begin{theorem}\label{ckperturb} If $R\in \mathfrak{Dom}\theta_{m}$, for all $m\leq n$, then the canonical maps
\[g_{k}:\mathfrak{G}((D+R)_{m})\rightarrow\mathfrak{G}(D_{m}),\]
are topological isomorphisms of $C^{k}$-modules for all $m\leq n+1$. 
\end{theorem}
\begin{proof} For $n=0$ the above lemma applies. Proceeding by induction, we suppose the theorem proven for $n-1$. The hypothesis imply that $\theta_{n-1}(R)\in\Endst_{\mathcal{B}_{i}}(\mathfrak{G}(D_{n-1})\oplus v_{[n-1]}\mathfrak{G}(D_{n-1}))$, and hence $\chi_{n-1}(R) \in\Endst_{\mathcal{B}_{k}}(\mathfrak{G}(D_{n-1}))$. The induction hypothesis gives isomorphisms
\[g_{n}\oplus g_{n}:\mathfrak{G}((D+R)_{n})\oplus \mathfrak{G}((D+R)_{n})\rightarrow\mathfrak{G}(D_{n})\oplus \mathfrak{G}(D_{n}),\]
under which the graph \[\mathfrak{G}((D+R)_{n+1})\rightarrow\mathfrak{G}(D_{n+1}+\chi_{n}(R)),\]  bijectively. This can be seen by induction: For $m=1$, the claim is obvious.
Suppose that $g_{m}\oplus g_{m}$ maps $\mathfrak{G}((D+R)_{m+1})$ bijectively to $\mathfrak{G}(D_{m+1}+\chi_{m}(R))$. This is equivalent to saying that  
\[g_{m}(D+R)e=Dg_{m}(e)+\chi_{m}(R)g_{m}(e).\] Then
\[g_{m}((D+R) e, (D+R)^{2}e)=(Dg_{m}e+\chi_{m}(R)g_{m}(e), D^{2}g_{m}e+D\chi_{m}(R)g_{m}e),\]
and since $(\chi_{m}(R)g_{m}e,D\chi_{m}(R)g_{m}e)=\chi_{m+1}(R)(g_{m}e,Dg_{m}e)$, $g_{m+1}\oplus g_{m+1}$ is a bijection \[\mathfrak{G}((D+R)_{m+2})\rightarrow \mathfrak{G}(D_{m+2}+\chi_{m+1}(R)),\] and the claim follows. 
Since we have
\[\chi_{n}(R) \in\Endst_{\mathcal{B}_{k}}(\mathfrak{G}(D_{n})),\] by lemma \ref{perturb}, the map
\[\begin{split}\mathfrak{G}(D_{n+1}+\chi_{n}(R))&\rightarrow\mathfrak{G}(D_{n+1})\\
(e,D_{n}e+\chi_{n}(R)e) &\mapsto (e,D_{n+1}e),\end{split}\]
is a topological isomorphism, and the composition of these two maps restricted to $\mathfrak{G}((D+R)_{n+1})$ is the canonical map $g_{n+1}$.
\end{proof}
\begin{corollary} If $R\in\Sob^{k}_{n}(D)$, or if it satisfies (2) of proposition \ref{smoothorder}, the canonical maps
\[g_{m}:\mathfrak{G}((D+R)_{m})\rightarrow\mathfrak{G}(D_{m}),\]
are topological isomorphisms of $C^{k}$-modules for all $m\leq n+1$.
\end{corollary}
\begin{proof} Both conditions imply $R\in\Dom \theta_{n}$, so the theorem applies.
\end{proof}
\begin{corollary}\label{invSob}If $R\in\Sob^{k}_{n}(D)$, then $\Sob^{k}_{m}(D)=\Sob^{k}_{m}(D+R)$ for all $m\leq n+1$.
\end{corollary}
\begin{proof} The statement clearly holds for $n=0$. Suppose it holds for $n-1$ and let $R=R^{*}\in\Sob^{k}_{n}(D)$. By the induction hypothesis, $\Sob_{n}^{k}(D)=\Sob_{n}^{k}(D+R)$. The isomorphism $g_{n}:\mathfrak{G}((D+R)_{n})\rightarrow\mathfrak\mathfrak{G}(D_{n})$ intertwines the representations $\chi_{n}^{D+R}$ and $\chi_{n}^{D}$, and for $a\in\Sob_{n}^{k}(D)=\Sob_{n}^{k}(D+R)$ we have
\[g_{n}[D+R,\chi_{n}^{D+R}(a)]g_{n}^{-1}=[D+\chi_{n}^{D}(R),\chi_{n}^{D}(a)].\]
Thus, cf. corollary \ref{astar} $a\in\Sob_{n+1}^{k}(D)$ if and only if $a\in \Sob_{n+1}^{k}(D+R)$.
\end{proof}

\section{Connections}
Connections on Riemannian manifolds are a vital tool for differentiating functions and vector fields over the manifold. Cuntz and Quillen \cite{CQ} developed a purely algebraic theory of
connections on modules, which gives a beautiful characterization of projective modules. They are exactly those modules that admit a universal connection. We review their results, but will recast everything in the setting
of operator modules. This is only straightforward, because the Haagerup tensor product linearizes the multiplication in an operator algebra in a continuous way. We then proceed to construct a category of modules with
connection, and finally pass to inverse systems of modules.
\subsection{Universal forms} The notion of universal differential form is widely used in noncommutative geometry, especially in connection with cyclic homology \cite{Con}. For topological algebras, their exact definition depends on a
choice of topological tensor product. The default choice is the Grothendieck projective tensor product, because it linearizes the multiplication in a topological algebra continuously. However, when dealing with operator
algebras, the natural choice is the Haagerup tensor product.
Meyer \cite{Meyer} has shown any operator algebra $\mathcal{B}$ admits a canonical unitization $\mathcal{B}^{+}$, simply by taking the unital algebra generated by it in any completely isomorphic representation $\pi:\mathcal{B}\rightarrow B(\mathpzc{H})$, on some Hilbert space $\mathpzc{H}$. Note that, when $\mathcal{B}$ has a unit, then $q=\pi(1)$ is an idempotent, so $\pi$ restricts to a unital representation on $B(q\mathpzc{H})$ and the unitization coincides with $\mathcal{B}$ in this case. In this section we will always replace $\mathcal{B}$ by $\mathcal{B}^{+}$.
\begin{definition} Let $\mathcal{B}$ be an operator algebra. The module of \emph{universal $1$-forms} over $\mathcal{B}$ is defined as
\[\Omega^{1}(\mathcal{B}):=\ker(m:\mathcal{B}\tildeotimes\mathcal{B}\rightarrow \mathcal{B}).\]
\end{definition}
Here $m$ is the graded multiplication map $m(a\otimes b)=(-1)^{\partial b}ab$. 
By definition, there is an exact sequence of operator bimodules
\[0\rightarrow\Omega^{1}(\mathcal{B})\rightarrow\mathcal{B}\tildeotimes\mathcal{B}\xrightarrow{m}\mathcal{B}\rightarrow 0.\]
The module $\Omega^{1}(\mathcal{B})$ inherits a grading from $\mathcal{B}\tildeotimes\mathcal{B}$.
The map 
\begin{eqnarray}\nonumber d:\mathcal{B}&\rightarrow &\Omega^{1}(\mathcal{B})\\
\nonumber a &\mapsto & 1\otimes a- (-1)^{\partial a}a\otimes 1\end{eqnarray}
is an even graded bimodule derivation. 
\begin{lemma} The derivation $d$ is universal. For any completely bounded graded derivation $\delta:\mathcal{B}\rightarrow M$ into an $\mathcal{B}$ operator bimodule, there is a unique completely bounded bimodule
homomorphism $j_{\delta}:\Omega^{1}(\mathcal{B})\rightarrow M$ such that the diagram
\begin{diagram}\mathcal{B} & &\rTo^{\delta} & & M\\
&\rdTo_{d}&  &\ruTo_{j_{\delta}}  & \\ & & \Omega^{1}(\mathcal{B})& &\end{diagram}
commutes. If $\delta$ is homogeneous, then so $j_{\delta}$ and $\partial\delta=\partial j_{\delta}$.\end{lemma}
\begin{proof} Set $j_{\delta}(da)=\delta(a)$. This determines $j_{\delta}$ because $da$ generates $\Omega^{1}(\mathcal{B})$ as a bimodule.\end{proof}
Any derivation $\delta:\mathcal{B}\rightarrow M$ has its associated module of forms \[\Omega^{1}_{\delta}:=j_{\delta}(\Omega^{1}(\mathcal{B}))\subset M.\]

\begin{definition} Let $\delta:\mathcal{B}\rightarrow M$ be a graded derivation as above, and $E$ a right operator $\mathcal{B}$-module. A $\delta$-\emph{connection} on $E$ is a completely bounded even linear map
\[\nabla_{\delta}:E\rightarrow E\tildeotimes_{\mathcal{B}}\Omega^{1}_{\delta},\]
satifying the Leibniz rule
\[\nabla_{\delta}(eb)=\nabla_{\delta}(e)b+e\otimes\delta(b).\]
If $\delta=d$, the connection will be denoted $\nabla$, and referred to as a \emph{universal connection}. 
\end{definition}
Note that a universal connection $\nabla$ on a module $E$ gives rise to $\delta$-connections for any completely bounded derivation $\delta$, simply by setting $\nabla_{\delta}:=1\otimes j_{\delta}\circ\nabla$. If $\delta$ 
is of the form $\delta(a)=[S,a]$, for $S\in\Hom_{\C}^{c}(X,Y)$, where $X$ and $Y$ are left $\mathcal{A}$-operator modules, we write simply $\nabla_{S}$ for $\nabla_{\delta}.$\newline

Not all modules admit a universal connection. Cuntz and Quillen showed that universal connections characterize
algebraic projectivity. Their proof shows that stably rigged modules admit universal connections, but the class of modules admitting a connection might be larger.
\begin{proposition}[\cite{CQ}]\label{CQ} A right $\mathcal{B}$ operator module $E$ admits
a universal connection if and only if the multiplication map $m:E\tildeotimes\mathcal{B}\rightarrow E$ is $\mathcal{B}$-split.
\end{proposition}
\begin{proof} Consider the exact sequence
\begin{diagram}0 &\rTo & E\tildeotimes_{\mathcal{B}}\Omega^{1}(\mathcal{B})
&\rTo^{j} & E\tildeotimes \mathcal{B}&\rTo^{m} &E &\rTo &
0,\end{diagram} where $m$ is the multiplication map and $j(e\otimes
da)=eb\otimes 1 -e\otimes b$. A linear map \[s:E\rightarrow
E\tildeotimes\mathcal{B}\] determines a linear map
\[\nabla:E\rightarrow
E\tildeotimes_{\mathcal{B}}\Omega^{1}(\mathcal{B})\] by the formula
$s(e)=e\otimes 1-j(\nabla(e))$, since $j$ is injective. Now
\[s(eb)-s(e)b=j(\nabla(e)b+e\otimes db-\nabla(eb)),\]
whence $s$ being an $\mathcal{B}$-module map is equivalent to $\nabla$ being a
connection.\end{proof}
\begin{corollary} Any stably rigged module $E$ over $\mathcal{B}$ admits a connection.
\end{corollary}
\begin{proof} $E$ is a direct summand in $\mathpzc{H}_{\mathcal{B}}$, i.e. $\mathpzc{E}=p\mathpzc{H}_{\mathcal{B}}$, with
$p^{2}=p\in\End^{*}_{\mathcal{B}}(\mathpzc{H}_{\mathcal{B}})$. Observe that
$\mathpzc{H}_{\mathcal{B}}\tildeotimes_{\mathcal{B}}\Omega^{1}{\mathcal{B}}\cong\mathpzc{H}\tildeotimes\Omega^{1}(\mathcal{B})$. Consider the \emph{Grassmann connection} 
\[\begin{split}d:\mathpzc{H}_{\mathcal{B}}&\rightarrow \mathpzc{H}\tildeotimes\Omega^{1}(\mathcal{B})\\
h\otimes a &\mapsto h\otimes da,\end{split}\]
and define $p\nabla p:E\rightarrow E\tildeotimes_{\mathcal{B}}\Omega^{1}(\mathcal{B})$.\end{proof}

\subsection{Product connections}
We now proceed to connections on tensor products of stably rigged modules. Anticipating the use of connections on unbounded bimodules, a category of modules with connection is constructed.
\begin{proposition}\label{conn1} Let $E$ be a stably rigged $\mathcal{B}$-module with a
universal connection $\nabla$, $F$ a stably rigged
$(\mathcal{B},\mathcal{C})$-bimodule with universal connection $\nabla '$. Then
$\nabla$ and $\nabla '$ determine a universal
$\mathcal{C}$-connection \[1\otimes_{\nabla}\nabla ':E\tildeotimes_{\mathcal{B}}F\rightarrow E\tildeotimes_{\mathcal{B}}F\tildeotimes_{\mathcal{C}}\Omega^{1}(\mathcal{C}). \] 
\end{proposition}
\begin{proof}
Consider the derivation 
\[\begin{split}\delta:\mathcal{B}&\rightarrow \End_{\mathcal{C}}(F,F\tildeotimes_{\mathcal{C}}\Omega^{1}(\mathcal{C}))\\b &\mapsto[\nabla ',b].\end{split}\] By
universality there is a unique map \[j_{\delta}:\Omega^{1}(\mathcal{B})\rightarrow \Omega^{1}_{\delta},\] intertwining $d$ and $\delta$. Thus, $\nabla$
induces a connection
\[\nabla_{\delta}:E\rightarrow E\tildeotimes_{\mathcal{B}}\Omega^{1}_{\delta},\]
by composing with $j_{\delta}$. Subsequently define
 \[\begin{split}1\tildeotimes_{\nabla}\nabla':E\tildeotimes_{\mathcal{B}}F &\rightarrow E\tildeotimes_{\mathcal{B}}F\tildeotimes_{\mathcal{C}}\Omega^{1}(\mathcal{C})\\e\otimes f &\mapsto e\otimes\nabla '(f)+\nabla_{\delta}(e)f,\end{split}\]
which is a connection. 
\end{proof} We will refer to the connection of proposition \ref{conn1} as the
\emph{product connection}.
Taking product connections is in fact a special case of the following construction. Let $(E,\nabla)$ be a graded stably rigged right $\mathcal{B}$ module with connection and $D\in\Hom^{c}(X,Y)$ a homogeneous operator between graded left $\mathcal{B}$ operator modules $X,Y$. Denote by $1\tildeotimes_{\nabla} D$ the operator
\begin{equation}\label{indop}1\otimes_{\nabla} D (e\otimes x):=(-1)^{\partial D\partial e}(e\otimes D(x) + \nabla_{D}(e)x),\end{equation}
which is a well defined operator $E\tildeotimes_{\mathcal{B}} X\rightarrow E\tildeotimes_{\mathcal{B}}Y$. This construction is associative up to isomorphism.
\begin{theorem}\label{op}\label{conn2}  Let $E$ be a stably rigged $\mathcal{B}$-module, $F$ a stably rigged $(\mathcal{B},\mathcal{C})$-bimodule and 
$\nabla,\nabla '$ universal connections on $E$ and $F$ respectively. Furthermore let $X,Y$ be left operator $\mathcal{C}$-modules,
and $D\in\Hom^{c}(X,Y)$. Then
\[1\tildeotimes_{\nabla}1\tildeotimes_{\nabla '}D=
1\tildeotimes_{1 \tildeotimes_{\nabla }\nabla '}D, \] under the intertwining isomorphism.
\end{theorem}
\begin{proof} Since $\partial D=\partial 1\otimes_{\nabla} D$ and  \[(-1)^{\partial(e\otimes f)\partial  D}=(-1)^{(\partial e+\partial f)\partial D}=(-1)^{\partial e\partial 1\otimes_{\nabla} D}(-1)^{\partial f\partial D},\]
 we assume that $D$ is even, the odd case differing only by this factor. Recall the formula for the product connection
\[1\otimes_{\nabla}\nabla '(e\otimes f):=e\otimes \nabla '(f)+\nabla_{\delta}(e)f.\] Moreover, write
$\nabla_{D}$ for $\nabla_{\nabla '_{D}}$. It is straightforward to check that
\[(1\tildeotimes_{\nabla}\nabla ')_{D}(e\otimes f)=e\otimes \nabla '_{D}(f)+\nabla_{D}(e)f.\]
Therefore we have
\[\begin{split}1\otimes_{1\otimes_{\nabla}\nabla '}D(e\otimes f\otimes x)&=e\otimes f\otimes Dx+1\otimes_{\nabla}\nabla' (e\otimes f)x \\
 &=e\otimes f\otimes D x+ e\otimes \nabla'_{D}(f)x+\nabla_{D}(e)(f\otimes x).\end{split}\]
On the other hand
\[\begin{split}  1\tildeotimes_{\nabla} 1\tildeotimes_{\nabla'}D(e\otimes f\otimes x)&= 
e\otimes (1\tildeotimes_{\nabla '}D)(f\otimes x)+\nabla_{1\tildeotimes_{\nabla '}D}(e)(f\otimes x)\\
&=e\otimes f\otimes Dx+e\otimes\nabla '_{D}(f)x+\nabla_{1\tildeotimes_{\nabla '}D}(e)(f\otimes x),\end{split}\]
thus, it suffices to show that $\nabla_{D}=\nabla_{1\tildeotimes_{\nabla '}D}.$ To this end, observe that
\[[1\tildeotimes_{\nabla '}D,b]=[\nabla '_{D},b]:E\otimes_{\mathcal{B}}F\rightarrow E\otimes_{\mathcal{B}}F,\]
which gives a natural isomorphism
$\Omega^{1}_{\nabla '_{D}}\xrightarrow{\sim}\Omega^{1}_{1\tildeotimes_{\nabla '}D}$ intertwining the derivations. By universality
this gives a commutative diagram
\begin{diagram} & & \Omega^{1}(\mathcal{B}) & & \\
&\ldTo & & \rdTo & \\
\Omega^{1}_{1\tildeotimes_{\nabla '}D}& &\rTo^{\sim} & &\Omega^{1}_{\nabla '_{D}},\end{diagram} 
 which shows that $\nabla_{D}=\nabla_{1\tildeotimes_{\nabla '}D}.$
\end{proof}
\begin{corollary}\label{conn3} Let $E,F,G$ be stably rigged $\mathcal{A}$-, $(\mathcal{A},\mathcal{B})$-,$(\mathcal{B},\mathcal{C})$ -(bi)modules,  with
universal connections $\nabla,\nabla ',\nabla ''$ respectively. The natural isomorphism
\[E\tildeotimes_{\mathcal{A}}(F\tildeotimes_{\mathcal{B}}G)\xrightarrow{\sim}(E\tildeotimes_{\mathcal{A}}F)\tildeotimes_{\mathcal{B}}G\] 
intertwines the product connections
$1\otimes_{1\otimes_{\nabla}\nabla '}\nabla ''$ and $1\otimes_{\nabla}(1\otimes_{\nabla '}\nabla '')$.
\end{corollary}

The upshot of theorem \ref{conn2} and its corollary \ref{conn3} is that there is a category whose objects are operator algebras, and whose morphisms 
$\Mor(\mathcal{A},\mathcal{B})$ are given by pairs $(E,\nabla)$ consisting of  a stably rigged right  $(\mathcal{A},\mathcal{B})$-bimodule $E$ with a 
universal $\mathcal{B}$ connection. The identity morphisms are the pairs 
$(\mathcal{A},d)$ consisiting of the trivial bimodule $\mathcal{A}$ and the universal derivation 
$d:\mathcal{A}\rightarrow \Omega^{1}(\mathcal{A})$. Of course this category is described equivalently as the category of pairs $(E,s)$ of 
bimodules together with a splitting $s$ of the universal exact sequence.
One can proceed to enrich the category described above by considering triples $(E,T,\nabla)$ consisting of stably rigged bimodules with connection and a
distinguished endomorphism $T\in \Endst_{\mathcal{B}}(E).$ 
The composition law then becomes
\[(E,S,\nabla)\circ(F,T,\nabla '):=(E\tildeotimes_{\mathcal{B}}F,S\tildeotimes 1+1\tildeotimes_{\nabla}T,1\tildeotimes_{\mathcal{B}}\nabla').\]

\subsection{Smooth connections}
We now consider $C^{k}$-modules over a $C^{k}$-algebra $B$.   If $\mathcal{B}$ is an involutive operator algebra, $\Omega^{1}(\mathcal{B})$ carries a natural involution, defined by
\begin{equation}\label{frominv} (adb)^{*}:=-(-1)^{\partial b}(db^{*})a^{*}.\end{equation}
The inner product on $E^{k}$ induces a pairing
\[\begin{split}E^{k}\times E^{k}\tildeotimes_{\mathcal{B}_{k}}\Omega^{1}(\mathcal{B}_{k})&\rightarrow \Omega^{1}(\mathcal{B}_{k})\\
(e_{1},e_{2}\otimes\omega) &\mapsto \langle e_{1},e_{2}\rangle\omega.\end{split}\]
By abuse of notation we write $\langle e_{1},e_{2}\otimes\omega\rangle$ for this pairing.
A pairing \[E^{k}\tildeotimes_{\mathcal{B}_{k}}\Omega^{1}(\mathcal{B}_{k})\times E^{k}\rightarrow\Omega^{1}(\mathcal{B}_{k}),\] is obtained by setting
$\langle e_{1}\otimes\omega,e_{2}\rangle:=\langle e_{2},e_{1}\otimes\omega\rangle^{*}.$
A connection \[\nabla:E^{k}\rightarrow E^{k}\tildeotimes_{\mathcal{B}_{k}}\Omega^{1}(\mathcal{B}_{k}),\] is a *-\emph{connection} if there is a connection $\nabla^{*}$ on $E^{k}$ for which
\begin{equation}\label{*}\langle e_{1},\nabla(e_{2})\rangle-\langle\nabla^{*} (e_{1}),e_{2}\rangle=d\langle e_{1},e_{2}\rangle.\end{equation}
The connection is \emph{Hermitian} if we can choose $\nabla^{*}=\nabla$ in the above equation.

\begin{lemma} Let $\nabla$ be a *-connection on a  $C^{k}$-module $E^{k}$. Then $\nabla^{*}$ is unique, and $\nabla^{**}=\nabla$.
\end{lemma} 
\begin{proof}Let $\tilde{\nabla}$ be a connection satisfying \ref{*}. By stabilizing and replacing $\nabla,\nabla^{*}$ and $\tilde{\nabla}$ by $\nabla\oplus d$, $\nabla^{*}\oplus d$ and $\tilde{\nabla}\oplus d$, we may assume $E^{k}=\mathpzc{H}_{\mathcal{B}_{k}}$. For any connection we have
\[\nabla(e)=\sum_{i\in \Z\setminus\{0\}} e_{i}\langle e_{i},\nabla(e)\rangle,\]
where $\{e_{i}\}_{i\in\Z\setminus\{0\}}$ is the standard basis of $\mathpzc{H}_{\mathcal{B}_{k}}$. Therefore it suffices to show that $\langle e_{i},\nabla^{*}(e)\rangle=\langle e_{i},\tilde{\nabla}(e)\rangle$. This follows immediately from (the adjoint of) \ref{*}.
\end{proof}
We now address smoothness and transversality of connections on smooth $C^{*}$-modules.
\begin{lemma} Let $\mathpzc{E}\leftrightharpoons B$ be a $C^{k}$-module over a $C^{k}$-algebra $B$, and $\nabla:E^{k}\rightarrow E^{k}\tildeotimes_{\mathcal{B}_{k}}\Omega^{1}(\mathcal{B}_{k})$ a *-connection. Then $\nabla$ uniquely extends to a *-connection on $E^{i}$ for all $i\leq k$.
\end{lemma}
\begin{proof} Recall the identification $E^{i}=E^{k}\tildeotimes_{\mathcal{B}_{k}}\mathcal{B}_{i}$ from proposition \ref{useful}, and observe that there is a canonical isomorphism
\[\begin{split}\mathcal{B}_{i}\tildeotimes_{\mathcal{B}_{k}}\Omega^{1}(\mathcal{B}_{k})\tildeotimes_{\mathcal{B}_{k}}\mathcal{B}_{i}&\rightarrow\Omega^{1}(\mathcal{B}_{i})\\
a\otimes db\otimes c &\mapsto a(db)c,\end{split}\]
compatible with $d$.
This allows us to define
\[\nabla(e\otimes b):=\nabla(e)\otimes b+e\otimes db,\]
which is easily checked to be a *-connection. Uniqueness follows from the fact that $E^{k}$ is dense in $E^{i}$ for $k\geq i$.
 \end{proof}

The operator spaces $\Sob_{n}^{D}(E^{k},E^{k}\tildeotimes_{\mathcal{B}_{k}}\Omega^{1}(\mathcal{B}_{k}))$ are defined via representations $\pi_{n}$ (\ref{prep}) and $\theta_{n}$ (\ref{trep}) on $\Hom^{c}_{\C}(E^{k},E^{k}\tildeotimes_{\mathcal{B}_{k}}\Omega^{1}(\mathcal{B}_{k}))$, relative to a regular operator in $E^{k}$, using the unbounded operator $D\otimes 1$ on $E^{k}\tildeotimes\Omega^{1}(\mathcal{B}_{k})$, to make sense of the commutators $[D,\cdot]$. This allows us to speak of transverse $C^{k}$-connections.
\begin{definition} Let $\mathpzc{E}\leftrightharpoons B$ be a  $C^{k}$-module over a $C^{k}$-algebra $B$, $D$ an almost selfadjoint regular operator in $E^{k}$ and $\nabla$ a *-connection in $E^{k}$. Then $\nabla$ is said to be \emph{transverse} $C^{n}$ if $\nabla\in\Sob_{n}^{D}(E^{k},E^{k}\tildeotimes_{\mathcal{B}_{k}}\Omega^{1}(\mathcal{B}_{k}))$. \end{definition}
Note that in this definition, $n$ and $k$ are independent of one another. This definition can be phrased equivalently by saying that $\nabla\in\mathfrak{Dom}\pi_{n-1}$ and $[D,\theta_{n}(\nabla)]$  extends to a
completely bounded operator $E^{k}\rightarrow E^{k}\tildeotimes_{\mathcal{B}_{k}}\Omega^{1}(\mathcal{B}_{k})$. Yet another equivalent way of phrasing this (cf. proposition \ref{smoothorder}) is to say that the operators
\[\textnormal{ad}(D)^{n}(\nabla)(D\pm i)^{-n+1},\quad\textnormal{and}\quad (D\pm i)^{-n+1}\textnormal{ad}(D)^{n}(\nabla),\]
extend to completely bounded operators $E^{k}\rightarrow E^{k}\tildeotimes_{\mathcal{B}_{k}}\Omega^{1}(\mathcal{B}_{k})$.
It is clear from this definition, that a transverse $C^{n}$-connection induces a connection on the Sobolev chain of $D$ up to degree $n$. 
\subsection{Induced operators and their graphs}\label{graph}

As we have seen, a *-connection
$\nabla:E^{k}\rightarrow E^{k}\tildeotimes_{\mathcal{B}_{k}}\Omega^{1}(\mathcal{B}_{k})$ can be used
to transfer operators on $F^{k}$ to $E^{k}\otimes_{\mathcal{B}_{k}} F^{k}$.
We now show that this algebraic procedure is well behaved for selfadjoint
regular operators $T$ in $F^{k}$, and describe the Sobolev chain, i.e. the graphs $\mathfrak{G}(1\tildeotimes_{\nabla}T)_{j}\subset E^{k}\tildeotimes_{\mathcal{B}_{k}}F^{k}\oplus E^{k}\tildeotimes_{\mathcal{B}_{k}}F^{k}$ as  topological
$C^{k}$-modules, in terms of the graph of $T_{j}$, as well as the graph representations $\chi_{j}^{k}$ (corollary \ref{graphrep}). Note that 
by proposition \ref{useful} $E^{i}\tildeotimes_{\mathcal{B}_{i}}F^{j}\cong E^{j}\tildeotimes_{\mathcal{B}_{j}}F^{j}$, whenever $i\geq j$. If $\mathpzc{E}$ carries a $C^{k}$- left module structure from another smooth $C^{*}$-algebra $A$, then $\mathpzc{E}\tildeotimes_{B}\mathpzc{F}$ carries a canonical $C^{k}$ left $A$-module structure.

\begin{theorem}\label{indgraph}  Let $k\geq 1$, $A,B,C,$ $C^{k}$-algebras, $\mathpzc{E},\mathpzc{F}$ a $C^{k}$-$(A,B)$, $(B,C)$-bimodules, respectively. Let $T:\mathfrak{Dom}(T)\rightarrow F^{k}$ be selfadjoint and  regular and transverse $C^{k}$ in $F^{k}$. If $\nabla:E^{k}\rightarrow E^{k}\tildeotimes_{\mathcal{B}_{k}}\Omega^{1}(\mathcal{B}_{k})$ is a *- connection, then the operator $t:=1\tildeotimes_{\nabla}T$ is almost selfadjoint and regular in $E^{k}\tildeotimes_{\mathcal{B}_{k}}F^{k}$. If $\nabla$ is Hermitian, then $1\tildeotimes_{\nabla}T$ is selfadjoint. With $g_{0}^{\chi}=\id_{\mathpzc{E}\tildeotimes_{B}\mathpzc{F}}$, for $j\leq i\leq k$, the inductively defined map 
\[\begin{split}g_{j}^{\chi}:E^{i}\tildeotimes_{\mathcal{B}_{i}}\mathfrak{G}(T_{j})^{i}&\rightarrow \mathfrak{G}(t_{j})^{i}\\
e\otimes(f,Tf) &\mapsto (g_{j-1}^{\chi}(e\otimes f) , 1\tildeotimes_{\nabla} T_{j}(g_{j-1}^{\chi} (e\otimes f))) \end{split}\]
is a topological isomorphism of $C^{i}$-modules. Moreover, we have $g_{j}^{\chi}\circ(a\otimes 1)=\chi^{t}_{j}(a)g^{\chi}_{i}$, thus $\mathcal{A}_{j}\rightarrow\Sob^{i}_{j}(1\otimes_{\nabla}T)$ completely boundedly, for $j\leq i\leq k$. So in particular $(\mathpzc{E}\tildeotimes_{B}\mathpzc{F}, 1\otimes_{\nabla}T)$ is transverse $C^{k}$.
\end{theorem}
\begin{proof}  To see that $t_{j}:=1\otimes_{\nabla} T_{j}$ is selfadjoint regular, stabilize $\mathpzc{E}$, and denote by $d$ the Grassmannian connection on $\mathpzc{H}_{B}$. Then, via the
stabilization isomorphism $\nabla ':=\nabla\oplus d$ defines a  $C^{k}$-*-connection on $\mathpzc{H}_{B}\cong \mathpzc{E}\oplus \mathpzc{H}_{B}.$ Since the difference $R:=\nabla'_{T}-d_{T}$ is an element of $\Endst_{\mathcal{C}_{k}}(E^{k}\tildeotimes_{\mathcal{B}_{k}}F^{k})$, it suffices to prove regularity of $t$ when $\nabla$ is the Grassmannian connection $d$ on $\mathpzc{H}_{B}$. In that case, define $t_{j}$ on the domain $\im (1\otimes_{d}(T_{j}\pm i)^{-1})$, which is dense. Then, for $e=\sum_{m\in\Z\setminus 0}e_{m}\otimes b_{m}$,
\[\begin{split}t_{j}:\mathpzc{H}_{\mathcal{B}_{i}}\tildeotimes_{\mathcal{B}_{i}} \mathfrak{G}(T_{j-1})^{i}&\rightarrow\mathpzc{H}_{\mathcal{B}_{i}}\tildeotimes_{\mathcal{B}_{i}}\mathfrak{G}(T_{j-1})^{i}\\
e\otimes f &\mapsto \sum_{m\in\Z\setminus 0}e_{m}\otimes T(b_{m}f).\end{split}\]
For $j=1$ this is symmetric by a standard argument. For $j>1$, $\mathcal{B}_{j}$ is represented on $\mathfrak{G}(T_{j-1})$ by a non $*$-homomorphism. But then, cf. \ref{innerprod}
\[\begin{split}\langle t_{j}(e \otimes f),e'\otimes f'\rangle &=\sum_{n\in\Z\setminus\{0\}}\langle e_{n}\otimes T_{j}b_{n}f, e'\otimes f'\rangle \\
&= \sum_{n\in\Z\setminus\{0\}}\sum_{m\in\Z\setminus\{0\}}\langle \langle e_{m},e_{n} \rangle T_{j}b_{n}f, \langle e_{m},e'\rangle f'\rangle \\
&=\sum_{n\in\Z\setminus\{0\}}\langle T_{j}b_{n}f, b'_{n}f'\rangle\\
&=\sum_{n\in\Z\setminus\{0\}}\langle b_{n}f, T_{j}b'_{n}f'\rangle \\
&=  \sum_{n\in\Z\setminus\{0\}}\sum_{m\in\Z\setminus\{0\}}\langle \langle e_{m},e\rangle f,\langle e_{m},e_{n} \rangle T_{j}b'_{n}f'\rangle \\
&=\langle e\otimes f, t_{j}(e'\otimes f')\rangle . \end{split}\]
Furthermore it is closed, selfadjoint and regular, because $t_{j}\pm i$ are surjective by construction:  $(t_{j} \pm i)(1\otimes_{d}(T_{j}\pm i)^{-1})=1$ for the connection $d$.\newline 

For the statement on the topological type of $\mathfrak{G}(t_{j}),$ it again suffices to consider the Grassmannian connection on $\mathpzc{H}_{\mathcal{B}_{k}}$: according to theorem \ref{ckperturb}, we have 
\[\begin{split}g_{j}:\mathfrak{G}((t+R)_{j})&\xrightarrow{\sim}\mathfrak{G}(t_{j})\\
(x,(t+R)x)&\mapsto (g_{j-1}x,tg_{j-1}x)\end{split}\] $C^{k}$-topologically, once we show that \[R=1\otimes_{d}T-1\otimes_{\nabla}T=d_{T}-\nabla_{T}\in\Dom \theta_{j}^{k}.\] To this end we compute
\[\begin{split}(\textnormal{ad}(1\otimes_{d} T))^{j}(R)(1\otimes_{d}T\pm i)^{-j}e_{m}\otimes f 
&=(\textnormal{ad}(1\otimes_{d} T))^{j}(\nabla_{T}-d_{T})e_{m}\otimes(T\pm i)^{-j}f\\
&=\sum_{n\in\Z\setminus\{0\}}(\textnormal{ad}(1\otimes_{d} T))^{j}e_{n}\otimes \omega^{m}_{n}(T\pm i)^{-j}f\\
&=\sum_{n\in\Z\setminus\{0\}}e_{n}\otimes (\textnormal{ad}T)^{j}(\omega^{m}_{n})(T\pm i)^{-j}f\end{split}\]
which is an element of  $\Endst_{\mathcal{C}_{k}}(E^{k}\otimes_{\mathcal{B}_{k}}F^{k}),$ since the connection is $C^{k}$. Here the $\omega^{m}_{n}\in\Omega^{1}_{T}$ are such that
\[\nabla_{T}(e_{m})=\sum_{n\in\Z\setminus\{0\}}e_{n}\otimes\omega^{m}_{n}.\]


Viewing $\mathfrak{G}(T_{i})$ as a submodule of $\mathfrak{G}(T_{i-1})\oplus\mathfrak{G}(T_{i-1})$, the
representations
\[\chi_{j}^{i}:\mathcal{B}_{j}\rightarrow \mathfrak{G}(T_{j})^{i},\]
from corollary \ref{graphrep} have the form 
\[\chi_{j}^{i}(b)(f,Tf)=(\chi_{j-1}^{i}(b)f,T\chi_{j-1}^{i}(b)f),\] by transversality.
For convenience, we suppress the $\chi_{j}^{i}$ in the notation.
By \eqref{innerprod}, the inner product (inducing an equivalent operator space structure)
on $\mathpzc{H}_{\mathcal{B}_{i}}\tildeotimes_{\mathcal{B}_{i}}\mathfrak{G}(T_{i})$ is thus given by 
\[\begin{split}\langle e\otimes (f,Tf), e'\otimes (f',Tf')\rangle:&=\sum_{n\in\Z\setminus\{0\}}\langle\langle e_{n},e\rangle (f,Tf), \langle e_{n},e'\rangle (f',Tf)\rangle\\
&=\sum_{n\in\Z\setminus\{0\}}\langle (b_{n} f, T b_{n}f), (b'_{n}f',Tb'_{n}f')\rangle\\
&=\sum_{n\in\Z\setminus\{0\}}\langle b_{n} f, b'_{n} f'\rangle + \langle Tb_{n} f, Tb'_{n} f'\rangle.\end{split}\]
Therefore the map
\[\begin{split}\mathpzc{H}_{\mathcal{B}_{i}}\tildeotimes_{\mathcal{B}_{i}}\mathfrak{G}(T_{j})^{i}&\rightarrow \mathfrak{G}(t_{j})^{i}\\
e\otimes (f,Tf) &\mapsto (e\otimes f, t(e\otimes f)),\end{split}\]
is unitary. \newline 

From this it follows that the $\mathcal{A}_{i}\rightarrow\Sob^{i}_{j}(1\otimes_{\nabla} T)$. For $j=1$ this holds because $[1\otimes_{\nabla}T,a]=[\nabla_{T},a]\otimes 1,$ which is a completely bounded derivation from $\mathcal{A}_{1}$ into $\Endst_{C}(E^{1}\tildeotimes_{\mathcal{B}_{1}}F^{1})$. Therefore, 
\[\pi_{1}^{1\otimes_{\nabla}T}:\mathcal{A}_{1}\rightarrow M_{2}(\Endst_{C}(E^{1}\tildeotimes_{\mathcal{B}_{1}}F^{1})),\] is completely bounded.
Suppose we have proven $\mathcal{A}_{i}\rightarrow\Sob^{i}_{j}(1\otimes_{\nabla}T)$ completely boundedly. The isomorphism $g_{j}$ intertwines the $\mathcal{A}_{i}$ representations, i.e., $g_{j}$ is a bimodule map. For $a\in\mathcal{A}_{i+1}$,
\[ [t_{j+1},\chi_{j}^{i+1}(a)]=g_{j}([ \nabla_{T_{j+1}},a]\otimes 1),\]
which is adjointable, and the same holds for $a^{*}$. Thus by corollary \ref{astar} $[t_{j+1},\theta_{j}^{i+1}(a)]$ is adjointable. It follows that $a\mapsto [t_{j+1},\theta_{j}^{i+1}(a)]$ is a completely bounded derivation $\mathcal{A}_{i+1}\rightarrow M_{2}(\mathfrak{G}(t_{j})$. Thusfor $j\leq i$,  $\mathcal{A}_{i+1}\rightarrow\Sob^{i+1}_{j+1}(1\otimes_{\nabla}T)$ completely boundedly.
\end{proof}

\begin{corollary}\label{resolvent} Let $k\geq 1$, $A,B,C$ be $C^{k}$-algebras, $(\mathpzc{E},S)$ and $(\mathpzc{F},T)$ transverse $C^{k}$-bimodules for $(A,B)$ and $(B,C)$ respectively, and $\nabla$ a Hermitian $C^{k}$-connection on $\mathpzc{E}$. If $\nabla$ is transverse $C^{i}$, the operators  $1\tildeotimes_{\nabla_{i}}T$ are almost selfadjoint, regular and transverse $C^{k}$ in $\mathfrak{G}(S_{i})\tildeotimes_{B}\mathpzc{F}$, with Sobolev modules $\mathfrak{G}((1\otimes_{\nabla_{i}}T)_{j})$ topologically isomorphic to $\mathfrak{G}(S_{i})^{j}\tildeotimes_{\mathcal{B}_{j}}\mathfrak{G}(T_{j})$.
\end{corollary}
\begin{proof} The connections $\nabla_{i}:E^{i}_{i}\rightarrow E^{i}_{i}\tildeotimes_{\mathcal{B}_{i}}\Omega^{1}(\mathcal{B}_{i})$ are $C^{k}$*-connections, so the statement follows from the previous theorem. 
\end{proof}

\subsection{Endomorphism algebras}
We are now able to show that the notions of left- and right-smoothness of a $C^{*}$-module can be treated on equal footing. The notion of connection links the the two concepts in an elegant way. 

Recall the representations
\[\theta_{n}:\mathcal{A}_{n}\rightarrow \Endst(\mathfrak{G}(D_{n}))\oplus\Endst(v_{[n]}\mathfrak{G}(D_{n})),\]
\[\pi_{n+1}:\mathcal{A}_{n+1}\rightarrow \Endst(\mathfrak{G}(D_{n})\oplus\mathfrak{G}(D_{n})).\]
The $\theta_{n}$ are endomorphisms of $X_{n}^{D}:=\mathfrak{G}(D_{n})\oplus v_{[n]}\mathfrak{G}(D_{n})$, respecting the direct sum decomposition. The $\pi_{n}$ act on $X_{n}^{D}\oplus X_{n}^{D}$, but do not respect the direct sum decomposition. 
\begin{theorem}\label{pithe} Let $\mathpzc{E}\lrh B$ be a $C^{k}$-module, $\nabla:E^{k}\rightarrow E^{k}\tildeotimes_{\mathcal{B}_{k}}\Omega^{1}(\mathcal{B}_{k})$ a *-connection and $\mathpzc{F}$ a transverse $C^{k}$ $(B,C)$-bimodule. There are canonical topological isomorphisms
\[\begin{split}g^{\pi}_{k}:E^{k}\tildeotimes_{\pi_{k}}(X_{k-1}\oplus X_{k-1})&\rightarrow E^{k}\tildeotimes_{\theta_{k-1}}X_{k-1}\oplus E^{k}\tildeotimes_{\theta_{k-1}}X_{k-1},\\
e\otimes\begin{pmatrix} x \\ y\end{pmatrix} &\mapsto\begin{pmatrix} e\otimes x \\ e\otimes y +\nabla_{T}(e)x\end{pmatrix},\end{split}\]
\[\begin{split} g^{\theta}_{k}:E^{k}\tildeotimes_{\theta_{k}}X_{k}&\rightarrow \mathfrak{G}(t_{k})\oplus v_{[k]}\mathfrak{G}(t_{k}),\\
e\otimes\begin{pmatrix} x\\ Tx\end{pmatrix}\oplus \begin{pmatrix} -Ty \\ y\end{pmatrix} &\mapsto \begin{pmatrix}  g^{\theta}_{k-1}(e\otimes x)\\ t_{k}g^{\theta}_{k-1}(e\otimes x)
\end{pmatrix}+v_{[k]}p^{t_{k}}v_{[k]}^{*} \begin{pmatrix} -g^{\theta}_{k-1}(e\otimes Ty )\\ g^{\theta}_{k-1}(e\otimes y +\nabla_{T}(e)Ty)\end{pmatrix},\end{split}\]
where $t_{i}=(1\otimes_{\nabla}T)_{i}$. Moreover, if $\mathpzc{E}$ is a transverse $C^{k}$ $(A,B)$ -bimodule, then we have \[g_{k}^{\pi}\circ (a\otimes 1)=\pi_{k}^{t}(a)g_{k}^{\pi},\quad g_{k}^{\theta}\circ (a\otimes 1)=\theta_{k}^{t}(a)g_{k}^{\theta},\] that is they are bimodule maps for the respective $\mathcal{A}_{k}$ module structures. 
\end{theorem}
\begin{proof} Well-definedness of $g^{\pi}_{k}$ is straightforward to check.
\[\begin{split}g_{k}^{\theta}(eb\otimes\begin{pmatrix} x \\ y\end{pmatrix} )&=\begin{pmatrix} eb\otimes x \\ eb\otimes y +\nabla_{T}(eb)x\end{pmatrix}\\
&=\begin{pmatrix} e\otimes x \\ e\otimes y +\nabla_{T}(e)\theta_{k}^{T}(b)x+e\otimes[T,\theta_{k}^{T}(b)]x\end{pmatrix}\\
&=g_{k}^{\pi}(e\otimes \pi_{k}^{T}(b)\begin{pmatrix}x \\ y\end{pmatrix}).\end{split}.\]
Its inverse is the map 
\[e\otimes\begin{pmatrix} x\\ y\end{pmatrix}\mapsto e\otimes \begin{pmatrix} x\\ y\end{pmatrix}-\nabla_{T}(e)\begin{pmatrix} 0 \\ x\end{pmatrix},\]
as is checked by computation.

 For $g_{k}^{\theta}$, well definedness is more of a surprise. First observe that $g_{k}^{\chi}$ is the first component of $g_{k}^{\theta}$, so we know this is a well defined topological isomorphism.  In case $\begin{pmatrix} -Ty \\ y\end{pmatrix}\in\mathfrak{Dom} T_{k+1}$, we can write
\[v_{[k]}p^{t_{k}}v_{[k]}^{*} \begin{pmatrix} -g^{\theta}_{k-1}(e\otimes Ty )\\ g^{\theta}_{k-1}(e\otimes y +\nabla_{T}(e)Ty)\end{pmatrix}=\begin{pmatrix} -t(1+t^{2})^{-1}e\otimes (1+T^{2})y\\ (1+t^{2})^{-1}e\otimes (1+T^{2})y\end{pmatrix}.\]
For such elements, we also have the expression
\[\theta_{k}^{T}(b)\begin{pmatrix} -Ty \\ y\end{pmatrix}=\begin{pmatrix} -T(1+T^{2})^{-1}\theta_{k-1}(b)(1+T^{2})y \\ (1+T^{2})^{-1}\theta_{k-1}(b)(1+T^{2})y\end{pmatrix},\]
from which well definedness follows directly. Thus, $g_{k}^{\theta}$ is well defined and completely bounded on a dense subset of $E^{k}\tildeotimes_{\theta_{k}}(\mathfrak{G}(T_{k}\oplus v_{[k]}\mathfrak{G}(T_{k})$, and hence extends to a well defined map on the entire tensor product module.
The fact that $g^{\theta}_{k}$ is a topological isomorphism is proved in a similar fashion as for $g^{\chi}_{k}$, by first stabilizing and considering the Grassmannian connection. The same computation as in the proof of \ref{indgraph} then shows that for this connection the map is unitary.
\end{proof}

\begin{corollary}\label{Kopalg} Let $B$ be a $C^{k}$-algebra, with defining $C^{k}$ spectral triple $(\mathpzc{H},D)$, $p$ the projection onto $\overline{B\mathpzc{H}}$ and $\mathpzc{E}\lrh B$ a $C^{k}$-module with connection. There are completely bounded isomorphisms
\[\K_{\mathcal{B}_{i}}(E^{i})\cong\Sob_{i}(1\otimes_{\nabla}D)\cap\K_{B}(\mathpzc{E})\otimes p,\] of involutive operator algebras, for all $i\leq k$, and similarly for $\Endst_{\mathcal{B}_{i}}(E^{i})$. In particular  $\K_{B}(\mathpzc{E})$ is completely boundedly isomorphic to a $C^{k}$-algebra and
\[\K_{\mathcal{B}_{i}}(E^{i})=\K_{B}(\mathpzc{E})\cap\Endst_{\mathcal{B}_{i}}(E^{i}).\]
\end{corollary} 
\begin{proof} The algebra $\mathcal{B}_{k}$ is completely isometrically isomorphic to a closed subalgebra of $\bigoplus _{i=0}^{k}M_{2^{i}}(B(\mathpzc{H}))$, via the defining representation 
$\pi_{[k]}$. Then by theorem \ref{projprop}, $p\in\Sob_{k}(D)$, and by theorem \ref{essrep}, $\K_{\mathcal{B}_{k}}(E^{k})$ is completely boundedly isomorphic to \[\K_{\mathcal{B}_{k}}(E^{k})\otimes \pi_{[k]}(p)\subset \bigoplus_{i=0}^{k}B(E^{i}\tildeotimes_{\mathcal{B}_{i}}(\bigoplus_{j=1}^{2^{i}}\mathpzc{H})).\] Choose a connection $\nabla:E^{k}\rightarrow E^{k}\tildeotimes_{\mathcal{B}_{k}}\Omega^{1}(\mathcal{B}_{k})$. We have $\K_{\mathcal{B}_{k}}(E^{k})\otimes \pi_{[k]}(p)\subset\Sob_{k}(1\otimes_{\nabla}D)$, by theorem \ref{indgraph}, and we have to show that it is closed.
Consider the map $g_{1}^{\pi}$ from theorem \ref{pithe}. It intertwines the representations $a\otimes 1$ and $\pi_{1}^{t}$, and since $\K_{\mathcal{B}_{1}}(E^{1})\otimes\pi_{[1]}(p)$ is a closed subalgebra on the left side, it is so on the right side. The same argument works for $\Endst_{\mathcal{B}_{1}}(E^{1})$. Next, assume that we have proven the result for $\Endst_{\mathcal{B}_{i-1}}(E^{i-1})$ and $\K_{\mathcal{B}_{i-1}}(E^{i-1})$. The operator space structure on $\mathcal{B}_{i}$  is given by the representation $\pi_{[i]}$. Applying theorem \ref{pithe} to $\pi_{i}$, and using the induction hypothesis on $\bigoplus_{j=0}^{i-1}\pi_{j}$, we see that $\K_{\mathcal{B}_{i}}(E^{i})\otimes p$ is a closed subset of $\Sob_{i}(1\otimes_{\nabla}T)$. The same reasoning applies to $\Endst_{\mathcal{B}_{i}}(E^{i})$. Since $\K_{\mathcal{B}_{i}}(E^{i})$ is dense in $\K_{B}(\mathpzc{E})$, it is a $C^{k}$-algebra.
\end{proof}
A $C^{k}$-module with connection $(\mathpzc{E},\nabla)$ can thus be viewed as a $C^{k}$- Morita equivalence between $\K_{\mathcal{B}_{k}}(E^{k})$, and the ideal $\langle E^{k},E^{k}\rangle\subset \mathcal{B}_{k}$.
\begin{corollary}\label{Endstab} Let $\mathpzc{E}\lrh B$ be a $C^{k}$-module over a $C^{k}$-algebra $B$. Then $\Endst_{\mathcal{B}_{k}}(E^{k})$ is spectral invariant in $\Endst_{B}(\mathpzc{E})$ and $\K_{B}(\mathpzc{E})$ is holomorphically dense in $\K_{\mathcal{B}_{k}}(E^{k})$.
\end{corollary}
\begin{proof} The argument from theorem \ref{hol}, does not directly apply since $\Endst_{\mathcal{B}_{k}}(E^{k})$ need not be dense in $\Endst_{B}(\mathpzc{E})$. When we replace $\Endst_{B}(\mathpzc{E})$ by the unital $C^{*}$-subalgebra $A=\overline{\Endst_{\mathcal{B}_{k}}(E^{k})},$ we find that $\Endst_{\mathcal{B}_{k}}(E^{k})$ is spectral invariant in $A$, which is spectral invariant in $\Endst_{B}(\mathpzc{E})$.
\end{proof}

\section{Correspondences}
We have seen how to employ connections as a tool in constructing products of unbounded selfadjoint operators. This observation leads to the construction of a category of spectral triples. They give a notion of morphism of noncommutative geometries, in such a way that the bounded transform induces a functor from correspondences to $KK$-groups. By considering several levels of
differentiability and smoothness on correspondences, one gets subcategories of correspondences of $C^{k}$- and smooth $C^{*}$-algebras.
\subsection{Almost anticommuting operators}
In this section we describe conditions on two almost selfadjoint regular operators, implying that  they induce almost selfadjoint regular operators in each others graphs. In the case of selfadjoint operators, this can be used to show that their sum is selfadjoint on the intersection of their domains. In the next section, where we introduce connections, pairs of such operators will be constructed in a natural way.
\begin{definition}\label{anti} Let $\mathpzc{E}\lrh B$ be a $C^{k}$-module and $s$ and $t$ almost selfadjoint regular operators in $E^{k}$. The operators $s$ and $t$ \emph{almost anticommute} if \begin{enumerate}\item There exists $\lambda > 0$ such that the operators \[(s\pm \lambda i)^{-1}(t\pm\lambda i)^{-1}\quad\textnormal{and}\quad (s\pm \lambda i)^{-1}(t\mp\lambda i)^{-1}\] and their adjoints all have the same range.
\item The operator $st+ts$, defined on $\im(s\pm \lambda i)^{-1}(t\mp\lambda i)^{-1}$,  extends to an operator in $\Endst_{\mathcal{B}_{k}}(E^{k})$; 
\end{enumerate}
\end{definition}
We need the following lemma.
\begin{lemma}\label{tildes} Let $s:\Dom s\rightarrow E^{k}$ be a closed densely defined operator, such that: 
\begin{enumerate} \item There exists $\lambda\in\R\setminus \{0\}$ such that $s\pm\lambda i$ is surjective and $(s\pm\lambda i)^{-1}\in \Endst_{\mathcal{B}_{k}}(E^{k})$;
\item $\Dom s\subset \Dom s^{*}$ and $s-s^{*}$ extends to an operator in $\Endst_{\mathcal{B}_{k}}(E^{k})$.
\end{enumerate}
Then $s$ is almost selfadjoint and regular in $E^{k}$.
\end{lemma}
\begin{proof} Write $R$ for the extension of $s-s^{*}$ to all of $E^{k}$. The operator $s+\frac{1}{2} R$ is symmetric on $\Dom s$, and
\[x=(s+\frac{1}{2}R\pm\lambda i)(s\pm\lambda i)^{-1}=1+\frac{1}{2}R(s+\lambda i)^{-1},\]
so, increasing $\lambda$ if necessary, $0\notin\Sp(x)$, so by corollary \ref{Endstab}, $x$ is an  invertible operator in $\Endst_{\mathcal{B}_{k}}(E^{k})$. Therefore $s+\frac{1}{2}R\pm\lambda i$ is surjective. Thus $s+\frac{1}{2}R$ is selfadjoint regular on $\Dom s$ by theorem \ref{dplusmini}.
\end{proof}
For a pair $(s,t)$ of almost anticommuting operators, we can define
\[\mathfrak{Dom} \chi_{1}^{t}(s):=\{(e,te)\in\mathfrak{G}(t): e\in\im(s\pm \lambda i)^{-1}(t\pm\lambda i)^{-1}\subset \mathfrak{Dom} s\cap\mathfrak{Dom t}\},\]
and $\chi_{1}^{t}(s)(e,te):=(se,tse)$, for $e\in\mathfrak{Dom}\chi_{1}^{t}(s)$. The notation $\chi^{t}_{1}(s)$ indicates the analogy with the bounded case.
$\chi_{1}^{s}(t):\mathfrak{Dom}\chi_{1}^{s}(t)\rightarrow\mathfrak{G}(s)$ is defined similarly, by switching $s$ and $t$.

\begin{proposition}\label{antigraph}Let $s$ and $t$ be almost anticommuting operators. Then $\chi_{1}^{t}(s)$ and $\chi^{s}_{1}(t)$ are almost selfadjoint and regular. Moreover, the map
\[\begin{split}\mathfrak{G}(\chi_{1}^{t}(s))&\rightarrow\mathfrak{G}(\chi_{1}^{s}(t))\\
(e,te,se, ste)&\mapsto (e,se,te,tse),\end{split}\]
is a topological isomorphism of $C^{k}$-modules.
\end{proposition}
\begin{proof}First we prove $\chi_{1}^{t}(s)$ is almost selfadjoint. By definition, $\chi_{1}^{t}(s)+\lambda i:\Dom\chi_{1}^{t}(s)\rightarrow \mathfrak{G}(t)$  is surjective. Moreover since
\[ [t,(s+\lambda i)^{-1}]=(s+\lambda i)^{-1}[s,t](s-\lambda i)^{-1}\in\Endst_{\mathcal{B}_{k}}(E^{k}),\]
we can write
\[(\chi_{1}^{t}(s)+\lambda i)^{-1}=\chi_{1}^{t}((s+\lambda i)^{-1})\in\Endst_{\mathcal{B}_{k}}(\mathfrak{G}(t)).\]
Thus, by lemma \ref{tildes}, $\chi_{1}^{t}(s)$ is almost selfadjoint and regular in $\mathfrak{G}(t)$.
The statement on the topological type follows by observing that the map $\mathfrak{G}(\chi_{1}^{t}(s))\rightarrow\mathfrak{G}(\chi_{1}^{s}(t))$ defined above can be written as
\[p^{\chi_{1}^{s}(t)}\begin{pmatrix} 1& 0& 0& 0\\0& 0&1 & 0\\0&1&0&0\\ [s,t] &0 &0 &-1\end{pmatrix} p^{\chi_{1}^{t}(s)}\in\Hom^{*}_{\mathcal{B}_{k}}(\mathfrak{G}(\chi_{1}^{t}(s)),\mathfrak{G}(\chi_{1}^{s}(t))),\]
and its inverse is obtained by interchanging $p^{\chi_{1}^{s}(t)}$ and $p^{\chi_{1}^{t}(s)}$.
\end{proof}
\begin{corollary} For almost anticommtuing operators $s$ and $t$ and $\lambda,\mu$ sufficiently large
\[\im(s\pm\lambda i)^{-1}(t\pm\mu i)^{-1}=\im(t\pm\lambda i)^{-1}(s\pm\mu i)^{-1},\]
\[\im(s\mp\lambda i)^{-1}(t\pm\mu i)^{-1}=\im(t\pm\lambda i)^{-1}(s\mp\mu i)^{-1}.\]
\end{corollary}
\begin{proof}It is immediate that $\im(s+\lambda i)^{-1}(t+\lambda i)^{-1}=\im(s+\lambda i)^{-1}(t+\mu i)^{-1}$. The equality $\im(t+\lambda i)^{-1}(s+\lambda i)^{-1}=\im(t+\mu i)^{-1}(s+\lambda i)^{-1}$ follows from almost selfadjointness of $\chi^{t}_{1}(s)$.
\end{proof}

One may define almost selfadjoint operators $\chi^{t}_{j}(s)$ in $\mathfrak{G}(t_{j})$ inductively whenever 
$\chi_{j-1}^{t}(s)$ and $t$ almost anticommute in $\mathfrak{G}(t_{j-1})$.\newline\newline

We now turn to the subject of the sum of almost anticommuting selfadjoint operators in $C^{*}$-modules.To this end we will use the following positivity result: Whenever $h,k\in\Endst_{B}(\mathpzc{E})$ are positive, $h$ has dense range, and $h\leq k$, then $k$ has dense range. The reader can consult \cite{Lan} for a proof of this statement, which plays a crucial r\^{o}le in the subsequent discussion.
\begin{lemma}\label{denseran} Let $(s,t)$ be a pair of almost anticommuting selfadjoint operators in a $C^{*}$-module $\mathpzc{E}$. For $\lambda,\mu$ positive and sufficiently large, the operators 
\[x=(t-\mu i)^{-1}-(s+\lambda i)^{-1} \quad\textnormal{and}\quad y=(t+\mu i)^{-1}-(s-\lambda i)^{-1}\] have dense range.
\end{lemma}
\begin{proof} We show that $xx^{*}$ has dense range, which implies $x$ has dense range. Write $a=(s+\lambda i)^{-1}$ and $b=(t-\mu i)^{-1}$ and compute
\[ xx^{*} =(a-b)(a^{*}-b^{*})=aa^{*}+bb^{*}-ab^{*}-ba^{*}.\]
We know that both $aa^{*}$ and $bb^{*}$, and hence also $aa^{*}+bb^{*}$ have dense range. Now observe that
\[-ab^{*}-ba^{*}=ab^{*}(\lambda\mu-[s,t])ba^{*}\geq 0,\]
and therefore $xx^{*}\geq aa^{*}+bb^{*}$, so $xx^{*}$ has dense range.
\end{proof}

\begin{lemma}\label{closed}  Let $(s,t)$ be a pair of almost anticommuting operators in a $C^{*}$-module $\mathpzc{E}$. Then the sum $s+t$ is closed and symmetric on $\Dom s \cap \Dom t$, and $\im(s+\lambda i)^{-1}(t+\mu i)^{-1}$ is a core for $s+t$.
\end{lemma}
\begin{proof}The sum $s+t$ is symmetric, and it is closed on $\Dom s\cap\Dom t$, which can be seen as follows. Let $x_{n}$ be a sequence in $\mathfrak{Im}(t\pm\lambda i)^{-1}(s\pm\lambda i)^{-1}\subset \Dom s\cap\Dom t$ converging to $x\in \mathpzc{E}$, and such that $(s+t)x_{n}$ is Cauchy in $E^{k}$. Then, for $y=x_{n}-x_{m}$, 
\[\begin{split}\langle (s+t)y,(s+t)y\rangle & = \langle sy,sy\rangle +\langle ty,ty\rangle +\langle sy, ty \rangle +\langle ty,sy\rangle \\ &=  \langle sy,sy\rangle +\langle ty,ty\rangle +\langle [s,t]y, y \rangle,\end{split}\]
and since $[s,t]$ is bounded on $\mathfrak{Im} (t\pm\lambda i)^{-1}(s\pm\lambda i)^{-1}$, we have \[\langle [s,t](x_{n}-x_{m}),(x_{n}-x_{m})\rangle\rightarrow 0\] for $n\geq m\rightarrow\infty$. Since the other two terms are positive, they must converge to zero as well (since the left hand side does so). Thus,  both $sx_{n}$ and $tx_{n}$ are convergent, and since both $s$ and $t$ are closed, we have $x\in \mathfrak{Dom} s\cap\Dom t$ and $(s+t)x_{n}$ must converge to $(s+t)x$. So $s+t$ is closed on $\Dom s\cap\Dom t$, and $\mathfrak{Im}(t\pm\lambda i)^{-1}(s\pm\lambda i)^{-1}$ is a core for $s+t$.
\end{proof}
\begin{lemma}\label{selfperturb} Let $D$ be a closed symmetric operator operator in $C^{*}$-module $\mathpzc{E}$. Suppose there exist $R_{\pm}\in\Endst_{B}(\mathpzc{E})$ and $\lambda > \max\{(\|R_{\pm}\|+1)^{2}\}$  such that the operators \[D+R_{\pm}\pm \lambda i:\Dom D\rightarrow \mathpzc{E},\] have dense range in $\mathpzc{E}$. Then $D\pm \lambda i$ are surjective and $D$ is selfadjoint and regular in $\mathpzc{E}$.
\end{lemma}
\begin{proof} Write $\tilde{D}$ for the closed operator $D+R_{+}$ and $b=R_{+}-R_{+}^{*}$. For $e\in\Dom D$ we can estimate
\[\langle (\tilde{D}+\lambda i) e, (\tilde{D}+\lambda i) e\rangle=\langle \tilde{D}e,\tilde{D}e\rangle+\lambda^{2}\langle e,e\rangle+\lambda\langle ibe,e\rangle\geq \lambda(\lambda-\|b\|)\langle e,e\rangle. \]
This shows that $(\tilde{D}+\lambda i)^{-1}$ is injective and extends to an operator $r\in\Endst_{B}(\mathpzc{E})$, with \begin{equation}
\label{normr}
\|r\|\leq\frac{1}{\sqrt{\lambda(\lambda-\|b\|)}}<\frac{1}{\sqrt{\lambda}}<\frac{1}{\|R_{+}\|},\end{equation}
because $\lambda > \max\{(\|R_{\pm}\|+1)^{2}\}\geq\|R_{+}\|^{2}+ \|b\|+1$. Let $e\in\mathpzc{E}$ be arbitrary and $e_{n}\in\im(\tilde{D}+\lambda i)$ be a sequence converging to $e$. Then $(\tilde{D}+\lambda i)^{-1}e_{n}\rightarrow re$ and \[\tilde{D}(\tilde{D}+\lambda i)^{-1}e=(1-\lambda i(\tilde{D}+\lambda i)^{-1})e\rightarrow e-\lambda ire.\] Since $D$ is closed we have $re\in\Dom D$ and $(\tilde{D}+\lambda i)re=e$. That is, $r=(\tilde{D}+\lambda i)^{-1}$ and $\tilde{D}+ \lambda i:\Dom D\rightarrow \mathpzc{E}$ is bijective. Now $\|R_{+}(D+R_{+}+\lambda i)^{-1})\|<1$ by \eqref{normr}, so we see that
\[(D+\lambda i)(D+R_{+}+\lambda i)^{-1}=1-R_{+}(D+R_{+}+\lambda i)^{-1},\]
 is invertible. Hence $D+\lambda i:\Dom D\rightarrow \mathpzc{E}$ is bijective. Using $D+R_{-}-\lambda i$, one shows in the same way that $D-\lambda i$ is bijective too, so $D$ is selfadjoint and regular.

\end{proof}
\begin{theorem}\label{selfsum}  Let $(s,t)$ be a pair of almost anticommuting operators in a $C^{*}$-module $\mathpzc{E}$. Then the sum $s+t$ is selfadjoint and regular on $\Dom s\cap \Dom t$.
\end{theorem}
\begin{proof} The operator $s+t$ is closed and symmetric by lemma \ref{closed}.
Consider the operators $x$ and $y$ from lemma \ref{denseran}. We can factor these operators as
\[x=(s+t+(\mu-\lambda)i-(s+\lambda i)^{-1}([s,t]-2\lambda\mu))(s-\lambda i)^{-1}(t-\mu i)^{-1},\]
\[y=(s+t+(\lambda-\mu)i-(s-\lambda i)^{-1}([s,t]-2\lambda\mu))(s-\mu i)^{-1}(t-\lambda i)^{-1},\]
by a standard algebraic computation. By lemma \ref{denseran}, $x$ and $y$ have dense range, and therefore the operators \[s+t\pm (\mu-\lambda)i - (s\pm\lambda i)^{-1}([s,t]-2\lambda\mu):\Dom s\cap\Dom t\rightarrow\mathpzc{E}\] have dense range. Choosing $\lambda,\mu$ positive, and such that \[\lambda>(2+2\mu)^{2}>\|[s,t]\|,\] we find that 
\[R_{\pm}=-(s\pm\lambda i)^{-1}([s,t]-2\lambda\mu)\in\Endst_{B}(\mathpzc{E}), \]  have norm \[\|R_{\pm}\|<\frac{\|[s,t]\|}{\lambda}+2\mu\leq 1+2\mu,\]
and $\lambda-\mu>(2+2\mu)^{2}\geq (\|R_{\pm}\|+1)^{2}$. Thus we can apply lemma \ref{selfperturb} to the closed symmetric operator $s+t$, the operators $R_{\pm}$ and the constant $\lambda-\mu$, to find that $s+t$ is selfadjoint and regular. 

\end{proof}
Note that for almost selfadjoint almost anticommuting operators, almost selfadjointness of the sum is not guaranteed by the above considerations. \newline

Next, we consider triples of almost anticommuting operators in a $C^{*}$-module. A triple of operators $(s,t,\partial)$ is said to \emph{almost anticommute} if each pair of them almost anticommutes and if
\[\mathfrak{Im}(s+\lambda i)^{-1}(t+\lambda i)^{-1}(\partial +\lambda i)^{-1}=\mathfrak{Im}(\partial+\lambda i)^{-1}(s+\lambda i)^{-1}(t +\lambda i)^{-1}.\]
Note that this implies that any order of resolvent products will have the same range, using that pairs almost anticommute.
\begin{proposition}\label{triple} Let $(s,t, \partial)$ is an almost anticommuting triple of almost selfadjoint regular operators in a $C^{*}$-module $\mathpzc{E}$. Suppose that $s+t$ is almost selfadjoint regular with core $\mathfrak{Im} (s+\lambda i)^{-1}(t+\lambda i)^{-1}$.Then $s+t$ and $\partial$ almost anticommute. 
\end{proposition}
\begin{proof} For notational convenenience we write
\[a=(\partial+\lambda i)^{-1},\quad b=(s+\lambda i)^{-1},\quad c=(t+\lambda i)^{-1},\quad d=(s+t+\lambda i)^{-1}.\]
We have to show that
\begin{equation}\label{imeq}\mathfrak{Im}(s+t+\lambda i)^{-1}(\partial +\lambda i)^{-1}=\mathfrak{Im}(\partial +\lambda i)^{-1}(s+t+\lambda i)^{-1},\end{equation}
and that the commutator $[s+t,\partial],$ is bounded on this set. Note that
\[\begin{split}\mathfrak{Im}(\partial +\lambda i)^{-1}(s+t+\lambda i)^{-1}&=\mathfrak{Im}(\partial +\lambda i)^{-1}(s+\lambda i)^{-1}\cap \mathfrak{Im}(\partial +\lambda i)^{-1}(t+\lambda i)^{-1}\\
&=\mathfrak{Im}(s+\lambda i)^{-1}(\partial +\lambda i)^{-1}\cap \mathfrak{Im}(t+\lambda i)^{-1}(\partial +\lambda i)^{-1}\\
 &\subset \mathfrak{Im}(s+t+\lambda i)^{-1},\end{split}\]
and that $[\partial,s+t]$ is bounded on this subset. Consider $s+t$ as an operator in the graph $\mathfrak{G}(\partial)$, defined on the above domain, and denote it by $\chi_{1}^{\partial}(s+t)$. Then $\chi_{1}^{\partial}(s+t)$ is a closed operator: Suppose
\begin{enumerate}
\item $adx_{n}\xrightarrow{\partial} ax,$
\item $ (s+t)adx_{n}\xrightarrow{\partial} ay,$
\end{enumerate}
where $\xrightarrow{\partial}$ means convergence in $\mathfrak{G}(\partial)$. (1) implies that $dx_{n}$ is a convergent sequence in $\mathpzc{E}$. Moreover we have 
\[(s+t)adx_{n}=-a^{*}(s+t)dx_{n}+[s,a]dx_{n}+[t,a]dx_{n},\]
from which it follows that $-a^{*}(s+t)dx_{n}$ is convergent (in $\mathpzc{E}$). Moreover, since
\[ [s,a]=sa+a^{*}s\quad\textnormal{on }\Dom s,\quad [t,a]=ta+a^{*}t\quad\textnormal{on }\Dom t,\]
in particular on $\im d=\Dom s\cap\Dom t$ we get
\[ [s,a]=a^{*}[s,\partial]a,\quad  [t,a]=a^{*}[t,\partial]a.\]
Therefore the term
\[ [s,a]dx_{n}+[t,a]dx_{n},\]
is convergent in $\mathfrak{G}(\partial)$, and hence $a^{*}(s+t)dx_{n}$ is so too.
Then from
\[a^{*}(s+t)dx_{n}=a^{*}x_{n}-ia^{*}dx_{n},\]
it follows that $a^{*}x_{n}$ is convergent in $\mathfrak{G}(\partial)$, which means that $x_{n}$ is a convergent sequence in $\mathpzc{E}$. From this it follows readily that $\chi^{\partial}_{1}(s+t)$ is closed in $\mathfrak{G}(\partial)$. It is almost symmetric on its domain since $\chi^{\partial}_{1}(s)$ and $\chi^{\partial}_{1}(t)$ are almost selfadjoint. Thus it remains to show that $\chi^{\partial}_{1}(s+t)+\lambda i$ has dense range for some $\lambda$. To this end we use that
\[\mathfrak{Im} abc\subset\Dom \chi^{\partial}_{1}(s+t).\]
Then since
\[\begin{split}(s+t+\lambda i)abc &= a^{*}(-s-t +\lambda i)bc + [s,a]bc+[t,a]bc\\
&=a^{*}(-s-t+\lambda i +[s,\partial]a+[t,\partial]a)bc,\end{split}\]
and $[s,\partial]a+[t,\partial]a$ is bounded, for $\lambda$ large enough we have that
\[\mathfrak{Im} (-s-t+\lambda i +[s,\partial]a+[t,\partial]a)bc,\]
is dense in $\mathpzc{E}$. Hence
\[\mathfrak{Im} a^{*}(-s-t+\lambda i +[s,\partial]a+[t,\partial]a)bc,\]
is dense in $\mathfrak{G}(\partial)$. Thus $\chi^{\partial}_{1}(s+t)$ is almost selfadjoint $\mathfrak{G}(\partial)$. This implies \eqref{imeq}, and the commutator properties are immediate, so $s+t$ and $\partial$ almost anticommute.
\end{proof}
By the same methods, we can prove the following proposition.
\begin{proposition}\label{st} Let $s,t$ be almost anticommuting almost selfadjoint operators in a $C^{k}$-module $E^{k}$. Suppose $s+t$ is almost selfadjoint on $\Dom s\cap \Dom t$, with core $\im(s+\lambda i)^{-1}(t+\lambda i)^{-1}$. Then 
\begin{enumerate}

\item $\chi^{s}(s+t)$ is almost selfadjoint in $\mathfrak{G}(s)$ and $\chi^{s}(s+t)=\chi^{s}(s)+\chi^{s}(t)$, i.e.
\[\Dom \chi^{s}(s+t)=\Dom \chi^{s}(s)\cap\Dom \chi^{s}(t).\] 
\item $\chi^{t}(s+t)$ is almost selfadjoint in $\mathfrak{G}(t)$ and $\chi^{t}(s+t)=\chi^{t}(s)+\chi^{t}(t)$, i.e.
\[\Dom \chi^{t}(s+t)=\Dom \chi^{t}(s)\cap \Dom\chi^{t}(t). \]
\item $\mathfrak{Dom}(s+t)^{2}=\Dom s^{2}\cap\im(s+\lambda i)^{-1}(t+\lambda i)^{-1}\cap\Dom t^{2}$.
\end{enumerate}
\end{proposition}
\begin{proof} Statements (1) and (2) are proved as in the previous proposition. For (3), note that
\[\begin{split}\Dom (s+t)^{2}&=(s+t+\lambda i)^{-1}\im (s+\lambda i)^{-1}\cap\im(t+\lambda i)^{-1}\\
&=\im (s+\lambda i)^{-1}(s+t+\lambda i)^{-1}\cap \im(t+\lambda i)^{-1}(s+t+\lambda i)^{-1}\\
&=\im (\lambda^{2}+s^{2})^{-1}\cap\im (s+\lambda i)^{-1}(t+\lambda i)^{-1}\cap\im (\lambda^{2}+t^{2})^{-1},\end{split}\]
where the second equality follows by using (1) and (2).

\end{proof}
Suppose we have a pair $(s,t)$ of almost selfadjoint operators in $E^{k}$, whose sum $s+t$ is almost selfadjoint on $\mathfrak{Dom}s\cap\mathfrak{Dom} t$. The graphs of $s$ and $t$ both map completely boundedly to $E^{k}$, by projection onto the first factor. Hence the pullback $\mathfrak{G}(s)*\mathfrak{G}(t)$ is defined, as the universal solution to the diagram
\begin{diagram} \mathfrak{G}(s)*\mathfrak{G}(t)&\rTo & \mathfrak{G}(s) \\
\dTo & & \dTo \\
\mathfrak{G}(t)&\rTo & E^{k}.
\end{diagram}
It can be identified (as a topological $C^{k}$-module) with the submodule of $\mathfrak{G}(s)\oplus\mathfrak{G}(t)$ given by
\[\mathfrak{G}(s)*\mathfrak{G}(t):=\{(e,se,e,te):e\in\Dom s\cap\Dom t\}.\]

\begin{proposition}\label{sumpull} If $s$ and $t$ are almost anticommuting almost selfadjoint regular operators in $E^{k}$ such that $s+t$ is almost selfadjoint regular on $\mathfrak{Dom}s\cap\mathfrak{Dom} t \subset E^{k},$ with core $\im(s+\lambda i)^{-1}(t+\lambda i)^{-1}$ then there is a topological isomorphism of $C^{k}$-modules
\[\begin{split}g:\mathfrak{G}((s+t))&\xrightarrow{\sim}\mathfrak{G}(s)*\mathfrak{G}(t)\\
(e,(s+t)e)&\mapsto (e,se,e,te).\end{split}\]
\end{proposition}
\begin{proof} By (3) of proposition \ref{st}, \[\Dom (s+t)^{2}=\Dom s^{2}\cap\im(s+\lambda i)^{-1}(t+\lambda i)^{-1}\cap\Dom t^{2},\] and since $s,t$ almost anticommute, $[s,t]$ is bounded on $\Dom(s+t)^{2}$ and so $s^{2}+t^{2}$ is a bounded perturbation of $(s+t)^{2}$. Hence it is almost selfadjoint regular on $\Dom(s+t)^{2}$, and the operator $\lambda^{2}+s^{2}+t^{2}$ is a bijection for $\lambda$ sufficiently large. In the following we take $\lambda=1$, which can always be achieved by rescaling. The module $\mathfrak{G}(s)*\mathfrak{G}(t)\subset\bigoplus_{j=1}^{4}E^{k}$ is the range of the (non-selfadjoint) idempotent
\[q:=\begin{pmatrix} a & as & a & at \\
sa & sas & sa & sat \\
a & as & a & at \\
ta & tas & ta & tat\end{pmatrix},\]
where $a=(1+s^{2}+t^{2})^{-1}$. Thus, by corollaries \ref{Kopalg} and \ref{rangeproj} there is a projection $p$ with \[p   \bigoplus_{j=1}^{4}E^{k}=\mathfrak{G}(s)*\mathfrak{G}(t),\]
and the map $g$ can be written as 
\[g=p\begin{pmatrix}b & (s+t)b\\sb & s(s+t)b\\b & (s+t)b\\tb & t(s+t)b\end{pmatrix} p^{s+t},\]
where $b=(1+(s+t)^{2})^{-1}$ and $p^{s+t}$ the Woronowicz projection. Moreover, we have
\[g^{-1}=p^{s+t}\begin{pmatrix}\frac{1}{2} & 0 & \frac{1}{2} & 0\\ 0& 1 & 0 & 1\end{pmatrix}p,\]
showing that $g$ is a topological isomorphism.
\end{proof}

\subsection{The product of transverse modules}
We now show that the operators $S\otimes 1$ and $1\otimes_{\nabla} T$ almost anticommute in the Sobolev modules of $1\otimes_{\nabla}T$.
From now on, write $s=S\tildeotimes 1$ and $t=1\tildeotimes_{\nabla} T$. The resolvents of $s$ and $t$ satisfy the following crucial compatibility. 
\begin{lemma}\label{resprod}Let $(\mathpzc{E},S)$ and $(\mathpzc{F},T)$ be transverse $C^{k}$ $(A,B)$ and $(B,C)$ bimodules respectively, and $\nabla:E^{k}\rightarrow E^{k}\tildeotimes_{\mathcal{B}_{k}}\Omega^{1}(\mathcal{B}_{k})$ a transverse $C^{n}$-connection. The $C^{k}$-endomorphisms
\[(t\mp\lambda i)^{-1}(s\pm\lambda i)^{-1}, (t\pm\lambda i)^{-1}(s\pm\lambda i)^{-1},(s\mp\lambda i)^{-1}(t\pm\lambda i)^{-1},(s\pm\lambda i)^{-1}(t\pm\lambda i)^{-1},\]  all have the same range in $\mathfrak{G}(S_{i})\tildeotimes_{\mathcal{B}_{j}}\mathfrak{G}(T_{j})$, for $i\leq n-1$,$ j\leq k-1$.
\end{lemma}
\begin{proof}Denote by $\pr^{S}:\mathfrak{G}(S_{i+1})\rightarrow \mathfrak{G}(S_{i})$ and $\pr^{T}:\mathfrak{G}(T_{j+1})\rightarrow \mathfrak{G}(T_{j})$ the adjointable operators given by projection on the first coordinate of the graph. $\pr^{T}$ is a $\mathcal{B}_{j}$-module map, and hence by theorem \ref{tensop}, 
\[\pr^{S}\otimes\pr^{T}:\mathfrak{G}(S_{i+1})\tildeotimes_{\mathcal{B}_{j+1}}\mathfrak{G}(T_{j+1})\rightarrow \mathfrak{G}(S_{i})\tildeotimes_{\mathcal{B}{j}}\mathfrak{G}(T_{j}),\] is an adjointable operator. We will show that all operators have range $\mathfrak{Im}\pr^{S}\otimes\pr^{T}$.\newline

For any $\lambda>0$, $(s\pm\lambda i)^{-1}$ maps  $\mathfrak{G}(S_{i})\tildeotimes_{\mathcal{B}_{j}}\mathfrak{G}(T_{j})$ bijectively onto $\mathfrak{Dom} S_{i+1}\otimes 1$, which is in bijection with $\mathfrak{G}(S_{i+1})\tildeotimes_{\mathcal{B}_{j}} \mathfrak{G}(T_{j})\cong \mathfrak{G}(1\otimes_{\nabla_{i+1}}T)_{j}$. Since $(1\otimes_{\nabla_{i+1}}T)_{j}$ is almost selfadjoint in $\mathfrak{G}(1\otimes_{\nabla_{i+1}}T)_{j-1}$, $(t\pm\lambda i)^{-1}$ maps this module bijectively onto \[\mathfrak{Dom} t\subset\mathfrak{G}(1\otimes_{\nabla_{i+1}}T)_{j-1}\] for $\lambda$ sufficiently large. This domain in turn is in bijection with $\mathfrak{G}(S_{i+1})\tildeotimes_{\mathcal{B}_{j}}\mathfrak{G}(T_{j})$, by theorem \ref{indgraph}. The diagram
\begin{diagram} \mathfrak{Dom} s &\rTo^{(t\pm\lambda i)^{-1}} & \mathfrak{G}(S_{i})\tildeotimes_{\mathcal{B}_{j}}\mathfrak{G}(T_{j})\\
\dTo & & \uTo^{\pr^{S}\otimes\pr^{T}}\\
\mathfrak{G}(s)&\rTo &\mathfrak{G}(S_{i+1})\tildeotimes_{\mathcal{B}_{j+1}}\mathfrak{G}(T_{j+1}),\end{diagram}
commutes, which means we have shown
\[\mathfrak{Im} (t\pm\lambda i)^{-1}(s\pm\lambda i)^{-1}=\mathfrak{Im}\pr^{S}\otimes\pr^{T}.\]
The map
\[\begin{split} r:\mathfrak{G}(S_{i})&\rightarrow\mathfrak{G}(S_{i+1})\\
e &\mapsto ((S+\lambda i)^{-1}e, S(S+\lambda i)^{-1}e)\end{split}\]
is a topological isomorphism, and hence, by theorem \ref{tensop}, 
\[r\otimes 1:\mathfrak{G}(S_{i})\tildeotimes_{\mathcal{B}{j}}\mathfrak{G}(T_{j})\rightarrow \mathfrak{G}(S_{i+1})\tildeotimes_{\mathcal{B}_{j}}\mathfrak{G}(T_{j}),\] is a topological isomorphism. Moreover the diagram
\begin{diagram}\mathfrak{Dom} t &\rTo^{(s+\lambda i)^{-1}} & \mathfrak{G}(S_{i})\tildeotimes_{\mathcal{B}_{j}} \mathfrak{G}(T_{j})\\
\dTo & &\uTo^{\pr^{S}\otimes\pr^{T}}\\
\mathfrak{G}(S_{i})\tildeotimes_{\mathcal{B}_{j+1}}\mathfrak{G}(T_{j+1}) &\rTo^{r\otimes 1} & \mathfrak{G}(S_{i+1})\tildeotimes_{\mathcal{B}_{j+1}}\mathfrak{G}(T_{j+1}),\end{diagram}
commutes, where the downward arrow is the bijection $\mathfrak{Dom} t\rightarrow \mathfrak{G}(t)^{i}$ composed with the map from theorem \ref{indgraph}. This proves that
\[\mathfrak{Im} \pr^{S}\otimes\pr^{T}=\mathfrak{Im}(s\pm\lambda i)^{-1}(t\pm\lambda i)^{-1}.\]

\end{proof}
Given a selfadjoint regular $C^{k}$-operator $S$ in $E^{k}$, we get naturally induced operators $S\otimes 1$ in all the modules $E^{i}\tildeotimes_{\mathcal{B}_{i}}\mathfrak{G}(T_{i})$, for $i\leq k$. Although these operators need not be almost selfadjoint they are still regular. 
\begin{lemma}\label{asatop}  Let $\mathpzc{E},\mathpzc{F}$ be $C^{k}$-modules over $B$, $g:E^{k}\rightarrow F^{k}$ a topological isomorphism, and $D$ an almost selfadjoint regular operator in $E^{k}$. Then $gDg^{-1}$ is regular in $F^{k}$.
\end{lemma}
\begin{proof} Since $D$ is almost selfadjoint, $D\pm\lambda i :\mathfrak{D}\rightarrow E^{k}$ are bijections, and $(D\pm\lambda  i)^{-1}\in\Endst_{\mathcal{B}_{k}}(E^{k})$. Therefore both $g(D+\lambda i)^{-1}g^{-1}$ and $g^{-1 *}(D^{*}-\lambda i)^{-1}g^{*}$ have dense range in $F^{k}$ and by corollary \ref{invregular}, their inverses are regular. Thus
\[gDg^{-1}=g(D+\lambda i)g^{-1} -\lambda i,\]
is a bounded perturbation of a regular operator, hence regular.
\end{proof}
\begin{proposition}\label{sreg} Let $(\mathpzc{E},S)$ and $(\mathpzc{F},T)$ be transverse $C^{k}$ $(A,B)$ and $(B,C)$ bimodules respectively, and $\nabla:E^{k}\rightarrow E^{k}\tildeotimes_{\mathcal{B}_{k}}\Omega^{1}(\mathcal{B}_{k})$ a transverse $C^{k}$-connection. For $i,j\leq k-1$, write $t=1\otimes_{\nabla_{i}}T$. The operators $\chi_{j}^{t}(s)$ and $t_{j}$ almost anticommute in $\mathfrak{G}(t_{j})$. Consequently, the operators $S_{i+1}\otimes 1$ are regular in each $\mathfrak{G}(S_{i})^{k}\tildeotimes_{\mathcal{B}_{k}}\mathfrak{G}(T_{j})^{k}$, with graph $\mathfrak{G}(S_{i+1})^{k}\tildeotimes_{\mathcal{B}_{k}}\mathfrak{G}(T_{j})^{k}$.
\end{proposition}
\begin{proof} By theorem \ref{indgraph}, the modules $\mathfrak{G}(S_{i})^{j}\tildeotimes_{\mathcal{B}_{j}}\mathfrak{G}(T_{j})$ are topologically isomorphic to the Sobolev modules $\mathfrak{G}((1\otimes_{\nabla_{i}}T)_{j})$ of $1\otimes_{\nabla_{i}}T$, an almost selfadjoint operator in $\mathfrak{G}(S_{i})^{k}\tildeotimes_{\mathcal{B}_{k}}F^{k}\cong\mathfrak{G}(s_{i})$. By lemma \ref{resprod} \[\mathfrak{Im}(s\pm\lambda i)^{-1}(t\pm \lambda i)^{-1}=\mathfrak{Im}(t\pm\lambda i)^{-1}(s\pm \lambda i)^{-1},\]
in $\mathfrak{G}(S_{i})\tildeotimes_{\mathcal{B}_{k}}F^{k}$, and $[s,t]=[\nabla_{i},S]\otimes 1$, so $s$ and $t$ almost anticommute cf. definition \ref{anti}, so $\chi_{1}^{t}s$ is almost selfadjoint in $\mathfrak{G}(1\otimes_{\nabla_{i}}T)$ by \ref{antigraph}. The topological isomorphism \[g^{\chi}: \mathfrak{G}(S_{i})\tildeotimes_{\mathcal{B}_{1}}\mathfrak{G}(T)\rightarrow \mathfrak{G}(1\otimes_{\nabla_{i}}T),\]
satisfies $g^{\chi}\chi_{1}^{t}(s)(g^{\chi})^{-1}=S\otimes 1$. So by lemma \ref{asatop}, $S\otimes 1$ is regular in $\mathfrak{G}(S_{i})\tildeotimes_{\mathcal{B}_{1}}\mathfrak{G}(T)$, and its graph is $\mathfrak{G}(S_{i+1})\tildeotimes_{\mathcal{B}_{1}}\mathfrak{G}(T)$. Proceeding by induction, suppose we have shown that $\chi^{t}_{j}(s)$ is almost selfadjoint in $\mathfrak{G}((1\otimes_{\nabla_{i}}T)_{j})$ and hence $S\otimes 1$ is regular in $\mathfrak{G}(S_{i})\otimes_{\mathcal{B}_{j}}\mathfrak{G}(T_{j})$. By lemma \ref{resprod} we have
\[\mathfrak{Im}(s\pm\lambda i)^{-1}(t\pm \lambda i)^{-1}=\mathfrak{Im}(t\pm\lambda i)^{-1}(s\pm \lambda i)^{-1},\]
in $\mathfrak{G}(S_{i})\otimes_{\mathcal{B}_{j}}\mathfrak{G}(T_{j})$ and hence also in $\mathfrak{G}((1\otimes_{\nabla_{i}}T)_{j})$, as $g^{\chi}_{j}$ intertwines these operators. Moreover
\[[\chi_{j}^{t}(s),t]=g_{j}^{\chi}([\nabla_{i},S]\otimes 1)(g_{j}^{\chi})^{-1},\]
which is bounded on $\mathfrak{Im}(s\pm\lambda i)^{-1}(t\pm \lambda i)^{-1}$, so $\chi^{t}_{j}(s)$ and $t$ almost anticommute in $\mathfrak{G}((1\otimes_{\nabla_{i}}T)_{j})$, and $\chi^{t}_{j+1}(s)$ is almost selfadjoint in $\mathfrak{G}((1\otimes_{\nabla_{i}}T)_{j+1})$, by \ref{antigraph}. The topological isomorphism $g^{\chi}_{j+1}$ intertwines $\chi^{t}_{j+1}$ and $S\otimes 1$, so the latter operator is regular in $\mathfrak{G}(S_{i})\tildeotimes_{\mathcal{B}_{j+1}}\mathfrak{G}(T_{j+1})$ by lemma \ref{asatop}.
\end{proof}
\begin{lemma}\label{connectionlemma} Let $(\mathpzc{E},S,\nabla)$ and $(\mathpzc{F},T,\nabla ')$ be $C^{k}$ bimodules with transverse $C^{k}$-connection. We have
\[ [S\otimes 1,1\otimes_{\nabla_{i}}\nabla_{j} ']=[S,\nabla_{i}]\otimes 1,\]
\[ [1\otimes_{\nabla_{i}}T, 1\otimes_{\nabla_{i}}\nabla_{j} ']=1\otimes_{\nabla_{i}}[\nabla_{j},T]+[\nabla_{i\nabla_{j}'} ,\nabla_{iT}],\]
where the left hand sides are defined on \[\mathfrak{Dom} S\otimes1\subset \mathfrak{G}(S_{i})^{k}\tildeotimes_{\mathcal{B}_{k}}\mathfrak{G}(T_{j})^{k}\textnormal{ and  } \mathfrak{Dom}1\otimes_{\nabla_{i}}T\subset \mathfrak{G}(S_{i})^{k}\tildeotimes_{\mathcal{B}_{k}}\mathfrak{G}(T_{j})^{k},\] respectively. Thus, these commutators extend to completely bounded maps 
\[\mathfrak{G}(S_{i})^{k}\tildeotimes_{\mathcal{B}_{k}}\mathfrak{G}(T_{j})^{k}\rightarrow\mathfrak{G}(S_{i})^{k}\tildeotimes_{\mathcal{B}_{k}}\mathfrak{G}(T_{j})^{k}\tildeotimes_{\mathcal{C}_{k}}\Omega^{1}(\mathcal{C}_{k}).\]
\end{lemma}
\begin{proof}The conditions imply we have transverse $C^{k+1-i}$ connections
\[\nabla_{i}:\mathfrak{G}(S_{i})^{k}\rightarrow \mathfrak{G}(S_{i})^{k}\tildeotimes_{\mathcal{B}_{k}}\Omega^{1}(\mathcal{B}_{k}),\quad \nabla'_{i}:\mathfrak{G}(T_{j})^{k}\rightarrow \mathfrak{G}(T_{j})^{k}\tildeotimes_{\mathcal{C}_{k}}\Omega^{1}(\mathcal{C}_{k}),\]
with the property that $[\nabla_{i},S]$, $[\nabla'_{j},T]$ are bounded endomorphisms of the respective modules.
These as well define product connections $1\otimes_{\nabla_{i}}\nabla_{j}'$ on $\mathfrak{G}(S_{i})^{k}\tildeotimes_{\mathcal{B}_{k}}\mathfrak{G}(T_{j})^{k},$ $i,j\leq k$. We show such connections boundedly commute with $S\otimes 1$ and $1\otimes_{\nabla}T$.
Since
\[[1\tildeotimes_{\nabla}\nabla',S\tildeotimes 1+1\tildeotimes_{\nabla}T]=[1\tildeotimes_{\nabla}\nabla',S\tildeotimes 1]+[1\tildeotimes_{\nabla}\nabla',1\tildeotimes_{\nabla}T],\]
and $[1\tildeotimes_{\nabla}\nabla',S\tildeotimes 1]=[\nabla,S]\tildeotimes 1$, which is completely bounded, we compute 
\[(-1)^{\partial e}[1\tildeotimes_{\nabla}\nabla',1\tildeotimes_{\nabla}T](e\otimes f)\] to find
\[e\otimes[\nabla',T]f+\nabla_{\nabla'}(e)Tf+1\tildeotimes_{\nabla}\nabla '(\nabla_{T}(e)f)-\nabla_{T}(e)\nabla'(f)-1\tildeotimes_{\nabla}T(\nabla_{\nabla'}(e)f).\]
The first term is completely bounded, and in working out the last four terms write $\nabla(e)=\sum e_{i}\otimes d b_{i}$. Then
\begin{eqnarray}\label{eerste}\nabla_{\nabla'}(e)Tf&=&\sum e_{i}\otimes [\nabla',b_{i}] Tf,\\ \label{tweede}\nabla_{T}(e)\nabla'(f)&=&\sum e_{i}\otimes [T,b_{i}]\nabla '(f),\\
\label{derde}1\tildeotimes_{\nabla}\nabla '(\nabla_{T}(e)f)&=&\sum e_{i}\otimes\nabla'[T,b_{i}]f +\nabla_{\nabla'}(e_{i})[T,b_{i}]f,\\
\label{vierde}1\tildeotimes_{\nabla}T (\nabla_{\nabla'}(e)f)&=&\sum_{i}e_{i}\otimes T[\nabla',b_{i}]f+\nabla_{T}(e_{i})[\nabla',b_{i}]f.\end{eqnarray}
Combining \ref{eerste},\ref{tweede} and the first terms on the right hand sides of \ref{derde} and \ref{vierde} give a term 
\[\sum_{i}e_{i}\otimes[[\nabla',T],b_{i}]f=\nabla_{[\nabla',T]}(e)f,\]
and the terms remaining from \ref{derde} and \ref{vierde} give a term
\[(\nabla_{\nabla'}\nabla_{T}-\nabla_{T}\nabla_{\nabla'})(e\otimes f).\]
Thus, we have shown that
\[[1\tildeotimes_{\nabla}\nabla',1\tildeotimes_{\nabla}T]=1\tildeotimes_{\nabla}[\nabla',T]+[\nabla_{\nabla'},\nabla_{T}],\]
which is a completely bounded map $\mathfrak{G}(S_{i})^{k}\tildeotimes_{\mathcal{B}_{k}}\mathfrak{G}(T_{j})^{k}\rightarrow \mathfrak{G}(S_{i})^{k}\tildeotimes_{\mathcal{B}_{k}}\mathfrak{G}(T_{j})^{k}\tildeotimes_{\mathcal{C}_{k}}\Omega^{1}(\mathcal{C}_{k})$. 

\end{proof}

\begin{theorem}\label{regular} Let $k\geq 1$, and $A,B,C$ be $C^{k}$-algebras, $(\mathpzc{E},S,\nabla)$ and $(\mathpzc{F},T,\nabla')$ transverse $C^{k}$-bimodules with connection, for $(A,B)$ and $(B,C)$ respectively. If $\nabla$ is transverse $C^{1}$, then the operator 
\[S\otimes 1+1\otimes_{\nabla} T\]
is selfadjoint and regular in $E^{k}\tildeotimes_{\mathcal{B}_{k}}F^{k}$. 
\end{theorem}
\begin{proof} Since $\nabla$ is transverse $C^{1}$ the operators $s=S\otimes 1$ and $t=1\otimes_{\nabla}T$ almost anticommute in $\mathpzc{E}\tildeotimes_{B}\mathpzc{F}$, and hence by theorem \ref{selfsum} $S\otimes 1 +1\otimes_{\nabla} T$ is selfadjoint and regular on $\Dom s\cap \Dom t$. Since both $s$ and $t$ are $C^{1}$, $s+t$ maps the $C^{1}$-domain $\mathfrak{Dom} s\cap \mathfrak{Dom} t\subset E^{1}\tildeotimes_{\mathcal{B}_{1}}F^{1}$ into  $E^{1}\tildeotimes_{\mathcal{B}_{1}}F^{1}$. To show that $s+t$ is selfadjoint in $E^{k}\tildeotimes_{\mathcal{B}_{k}}F^{k}$, it suffices to show that $(s+t+i)^{-1}$  is an element of $\Endst_{\mathcal{C}_{k}}(E^{k}\tildeotimes_{\mathcal{B}_{k}}F^{k})$. \newline\newline
Let $k=1$, and $(C,\mathpzc{H},D)$ be the defining $C^{1}$-spectral triple for $C$. Consider the map
\[g^{\pi}_{1}:E^{1}\tildeotimes_{\mathcal{B}_{1}}F^{1}\tildeotimes_{\pi}(\mathpzc{H}\oplus\mathpzc{H})\rightarrow \mathpzc{E}\tildeotimes_{B}\mathpzc{F}\tildeotimes_{C}(\mathpzc{H}\oplus\mathpzc{H}),\]
from theorem \ref{pithe}. Write $\partial:= 1\otimes_{1\otimes_{\nabla}\nabla '}D$. We have 
\[g \Endst_{\mathcal{B}_{1}}(E^{1}\tildeotimes_{\mathcal{B}_{1}}F^{1})g^{-1}\subset \Sob_{1}(\partial),\]
as a closed subalgebra. Thus it suffices to show that under this isomorphism \[(s+t+i)^{-1}\in \Sob_{1}(\partial).\]
This means we have to show that $(s+t+i)^{-1}$ preserves the domain of $\partial$, and that $[(s+t+i)^{-1},\partial]$ is bounded on the domain. By construction, the operators $(s,t,\partial)$ form an almost anticommuting triple in the Hilbert space $\mathpzc{E}\tildeotimes_{B}\mathpzc{F}\tildeotimes_{C}(\mathpzc{H}$. Moreover $s+t$ is selfadjoint with core $\mathfrak{Im}(s+\lambda i)^{-1}(t+\lambda i)^{-1}$. Thus, by proposition \ref{triple}, we find that $s+t$ almost anticommutes with $\partial$, and in particular that $(s+t+i)^{-1}\in\Sob_{1}(\partial)$. Proceeding by induction on $k$, $(\mathpzc{H},D)$ is the defining $C^{k}$ spectral triple for $C$, and suppose we have shown that $(i+s+t)^{-1}\in\Sob_{k-1}(\partial)$, and that 
\[\mathfrak{Im}\chi_{k-1}^{\partial}((s+\lambda i)^{-1}(t+\lambda i)^{-1}),\]
is a core for $\chi_{k-1}^{\partial}(s+t)$. Using the isomorphism
\[g^{\pi}_{k}:E^{k}\tildeotimes_{\mathcal{B}_{k}}F^{k}\tildeotimes_{\pi_{k}}(\mathpzc{H}_{k}\oplus\mathpzc{H}_{k})\rightarrow E^{k}\tildeotimes_{\mathcal{B}_{k}}F^{k}\tildeotimes_{\theta_{k-1}}(\mathpzc{H}_{k-1}\oplus\mathpzc{H}_{k-1}),\]
where $\mathpzc{H}_{k}$ is the $k$-th Sobolev space of $D$, we see by theorem \ref{indgraph} that this yields two copies of the $k-1$-th Sobelev space of $\partial$. Now applying proposition \ref{triple} again to $\chi_{k-1}^{\partial}(s),\chi_{k-1}^{\partial}(t)$ and $\partial_{k-1}$, we see that $(s+t\pm i)^{-1}\in\Sob_{k}(\partial)$. Hence the operator $S\otimes 1 + 1\otimes_{\nabla} T$ is selfadjoint and regular in $E^{k}\tildeotimes_{\mathcal{B}_{k}}F^{k}$.
\end{proof}
Since the product operator is $C^{k}$, its Sobolev chain can be canonically smoothened. Also, it is obviously transverse $C^{1}$. We now proceed to show higher order transverse smoothness.
\begin{lemma}\label{gsgt}Let $k\geq 1$, and $A,B,C$ be $C^{k}$-algebras, $(\mathpzc{E},S,\nabla)$ and $(\mathpzc{F},T,\nabla')$ transverse $C^{k}$-bimodules with connection, for $(A,B)$ and $(B,C)$ respectively. If $\nabla$ is transverse $C^{i}$, with $i\geq 1$, then for all $j\leq k$, the operator 
\[S_{i}\otimes 1+1\otimes_{\nabla_{i-1}} T_{j}\]
is regular in $\mathfrak{G}(S_{i-1})^{k}\tildeotimes_{\mathcal{B}_{k}}\mathfrak{G}(T_{j-1})^{k}$. Moreover its graph is topologically isomorphic to
\[\mathfrak{G}(S_{i})^{k}\tildeotimes_{\mathcal{B}_{k}}\mathfrak{G}(T_{j-1})^{k}*\mathfrak{G}(S_{i-1})^{k}\tildeotimes_{\mathcal{B}_{k}}\mathfrak{G}(T_{j})^{k}.\]
\end{lemma}
\begin{proof} For $i=1$, $j=1$, the operator  $S\otimes 1+1\otimes_{\nabla}T$ is selfadjoint, regular and $C^{k}$ by theorem \ref{regular}. Now proposition \ref{sumpull} gives the graph isomorphism 
\[\mathfrak{G}(S\otimes 1+1\otimes_{\nabla}T)^{k}\rightarrow \mathfrak{G}(S)^{k}\otimes_{\mathcal{B}_{k}}F^{k}*E^{k}\tildeotimes_{\mathcal{B}_{k}}\mathfrak{G}(T)^{k}.\]
Moreover, $s=S\otimes 1$, $t=1\otimes_{\nabla}T$ and $s+t$ are almost selfadjoint in $\mathfrak{G}(s)$ and $\mathfrak{G}(t)$ by proposition \ref{st}. Next we proceed by induction on $j\leq k$. Suppose we have shown that $S\otimes 1,1\otimes_{\nabla}T_{j-1}$ and $S\otimes 1+1\otimes_{\nabla}T_{j-1}$ are almost selfadjoint regular in $\mathfrak{G}(1\otimes_{\nabla}T)_{j-2}^{k}$, which is topologically isomorphic to $ E^{k}\tildeotimes_{\mathcal{B}_{k}}\mathfrak{G}(T_{j-2})^{k}$. Since the connection is transverse $C^{1}$, $[S,1\otimes_{\nabla}T_{j-1}]$ is bounded in these modules, and proposition \ref{st} now gives that $S\otimes 1, 1\otimes_{\nabla}T_{j}$ and their sum are almost selfadjoint regular in $\mathfrak{G}(1\otimes_{\nabla}T)_{j-1}^{k}$. Using the  proposition \ref{sumpull} and theorems \ref{indgraph} and \ref{sreg} gives
\[\mathfrak{G}(S\otimes 1 +1\otimes_{\nabla}T_{j})^{k}\xrightarrow{\sim}\mathfrak{G}(S)^{k}\otimes_{\mathcal{B}_{k}}\mathfrak{G}(T_{j-1})^{k}*E^{k}\tildeotimes_{\mathcal{B}_{k}}\mathfrak{G}(T_{j})^{k}.\]
We proceed by induction on $i$, so suppose we have proven, that for a transverse $C^{i}$-connection, $S_{i}\otimes 1,$ $1\otimes_{\nabla_{i-1}}T_{j}$ and $S_{i}\otimes 1 +1\otimes_{\nabla_{i-1}}T_{j}$ are almost selfadjoint in the Sobolev module $\mathfrak{G}(1\otimes_{\nabla_{i-1}}T)_{j-1}^{k}$ which is topologically isomorphic to $ \mathfrak{G}(S_{i-1})^{k}\tildeotimes_{\mathcal{B}_{k}}\mathfrak{G}(T_{j-1})^{k}.$ Since the connection is transverse $C^{i+1}$ the commutator $[S_{i+1}\otimes 1,1\otimes_{\nabla_{i}}T]=[S_{i+1},\nabla_{i}]$ is bounded, and by proposition \ref{st} $S_{i+1}\otimes 1, 1\otimes_{\nabla_{i}}T_{j}$ and $S_{i+1}\otimes 1+1\otimes_{\nabla_{i}}T_{j}$ give rise to almost selfadjoint operators in the graph of $S_{i+1}\otimes 1$, viewed as an operator in $\mathfrak{G}(1\otimes_{\nabla_{i}}T)_{j}$. By theorems \ref{indgraph} and \ref{sreg}, this graph is isomorphic to $\mathfrak{G}(S_{i+1})^{k}\tildeotimes_{\mathcal{B}_{k}}\mathfrak{G}(T_{j})^{k}$. Combining this with proposition \ref{sumpull} gives the graph isomorphism
\[\mathfrak{G}(S_{i+1}\otimes 1+1\otimes_{\nabla_{i}}T_{j})^{k}\xrightarrow{\sim} \mathfrak{G}(S_{i+1})^{k}\tildeotimes_{\mathcal{B}_{k}}\mathfrak{G}(T_{j-1})^{k}*\mathfrak{G}(S_{i})^{k}\tildeotimes_{\mathcal{B}_{k}}\mathfrak{G}(T_{j})^{k}.\]
\end{proof}
The next result can be regarded as a type of K\"{u}nneth formula for smooth products.
\begin{theorem}\label{smregular} Let $k\geq 1$, and $A,B,C$ be $C^{k}$-algebras, $(\mathpzc{E},S,\nabla)$ and $(\mathpzc{F},T,\nabla ')$ transverse $C^{k}$-bimodules for $(A,B)$ and $(B,C)$ respectively. For all $i\leq k$, there are natural topological isomorphisms
\[\begin{split}g_{i}:\mathfrak{G}((S\otimes 1+1\otimes_{\nabla}T)_{i})^{i}&\xrightarrow{\sim}\begin{array}{cc} ^{i} \\ \ast \\  _{j=0}\end{array}\mathfrak{G}(S_{j})^{i}\tildeotimes_{\mathcal{B}_{i}}\mathfrak{G}(T_{i-j})^{i},\end{split}\]
where the successive pullbacks are over the maps $\pr_{1}^{S}\otimes\pr_{1}^{T}$. Moreover \[g_{i}\circ\chi^{s+t}_{i}(a)=(a\otimes 1)\circ g_{i},\] for all $i\leq k$, and consequently $\mathcal{A}_{i}\rightarrow\Sob^{i}_{i}(S\otimes 1+1\otimes_{\nabla}T)$ completely boundedly, and the connection $1\otimes_{\nabla}\nabla '$ is transverse $C^{k}$. That is \[(\mathpzc{E}\tildeotimes_{B}\mathpzc{F}, S\otimes 1+1\otimes_{\nabla}T, 1\otimes_{\nabla}\nabla')\] is a transverse $C^{k}$-bimodule with connection.
\end{theorem}
\begin{proof}  Write $s=S\otimes 1$, $t=1\otimes_{\nabla}T$. From proposition \ref{sumpull} we get that $\mathfrak{G}(s+t)\cong \mathfrak{G}(s)*\mathfrak{G}(t)$. 
This is the theorem for $k=1$. Suppose the theorem has been proven for $k$. The map
\[g_{k}\oplus g_{k}: \mathfrak{G}(s+t)_{k}^{k}\oplus \mathfrak{G}(s+t)_{k}^{k}\rightarrow  \bigoplus_{i=1}^{2}\begin{array}{cc} ^{k}  \\ \ast  \\  _{j=0}\end{array}\mathfrak{G}(S_{j})^{k}\tildeotimes_{\mathcal{B}_{k}}\mathfrak{G}(T_{k-j})^{k},\]
maps $\mathfrak{G}(s+t)_{k+1}^{k+1}$ onto the graphs $\mathfrak{G}(s+t)\subset\mathfrak{G}(S_{j})^{k+1}\tildeotimes_{\mathcal{B}_{k+1}}\mathfrak{G}(T_{k-j})^{k+1}$. By lemma \ref{gsgt}, these graphs are isomorphic to \[\mathfrak{G}(S_{j+1})^{k+1}\tildeotimes_{\mathcal{B}_{k+1}}\mathfrak{G}(T_{k-j})^{k+1}*\mathfrak{G}(S_{j})^{k+1}\tildeotimes_{\mathcal{B}_{k+1}}\mathfrak{G}(T_{k+1-j})^{k+1}.\]
Eliminating the double terms from the pullback, this gives the theorem for $k+1$. Now let $a\in\mathcal{A}_{k+1}$, and assume by induction that $g_{k}\chi^{s+t}_{k}(a)=(a\otimes 1) g_{k}$ (this is obvious for $k=1$). Then 
\[g_{k}[s+t,\chi^{s+t}_{k}(a)]g_{k}^{-1}=[S\otimes 1+1\otimes_{\nabla}T,a\otimes 1]=[S,a]\otimes 1+[\nabla,a]\otimes 1,\]
in each $\mathfrak{G}(S_{j})\tildeotimes_{\mathcal{B}_{i}}\mathfrak{G}(T_{k-j})$. This is a bounded operator in $\mathfrak{G}(S_{j})\tildeotimes_{\mathcal{B}_{k+1}}\mathfrak{G}(T_{k-j})^{k+1}$, similarly for $a^{*}$, so by lemma \ref{astar} $a\in\Sob^{k+1}_{k+1}(s+t)$. It is immediate that then $g_{k+1}\chi^{s+t}_{k+1}(a)=a\otimes 1 g_{k+1}$, and the map $\mathcal{A}_{k+1}\rightarrow \Sob^{k+1}_{k+1}(s+t)$ is completely bounded. The statement on the connection $1\otimes_{\nabla}\nabla'$ follows by a similar argument applying \ref{connectionlemma}.
\end{proof}

As a consequence, we see that for $k\geq 1$, transverse $C^{k}$-triples $(\mathpzc{E},S,\nabla)$, with $C^{k}$-connection can be composed according to the rule
\[(\mathpzc{E},S,\nabla)\circ(\mathpzc{F},T,\nabla '):=(\mathpzc{E}\tildeotimes_{B}\mathpzc{F},S\otimes 1+1\otimes_{\nabla} T,1\otimes_{\nabla}\nabla '),\]
and that this composition is associative up to unitary equivalence inducing topological isomorphisms on the graphs and smooth structures.

\subsection{The $KK$-product} Now we establish that compact resolvents are preserved under taking products. Then we will see that the product operator satisfies Kucerovsky's conditions for an
unbounded Kasparov product. Thus, for smooth modules the $KK$-product is given by an explicit algebraic formula. Let us
put the pieces together.
\begin{lemma}\label{cptperturb} Let $s,t$ be selfadjoint regular operators on a $C^{k}$-module $\mathpzc{E}$, and $R,a\in\End^{*}_{\mathcal{B}_{k}}(E^{k})$ with $R\in\Endst_{\mathcal{B}_{k}}(E^{k})$ a selfadjoint element. If $a(s+i)^{-1}(t+i)^{-1}\in\K_{\mathcal{B}_{k}}(E^{k}),$ then
$a(s+i)^{-1}(t+R+i)^{-1}\in\K_{\mathcal{B}_{k}}(E^{k}).$
\end{lemma} 
\begin{proof} One has the identity
\[a(s+i)^{-1}(i+t+R)^{-1}=a(s+i)^{-1}(i+t)^{-1}(1-R(t+i)^{-1}),\]
which is a compact operator.
\end{proof}
 We now show that the product of cycles is a cycle. Note that this result is a generalization of the stability property of spectral triples proved in
 \cite{Mathai}. There it was shown that tensoring a given spectral triple by a finitely generated projective module yields again a spectral triple.
\begin{lemma}\label{sumprodcpt} Let $s,t$ be a pair of almost anticommuting selfadjoint regular operators in a $C^{k}$-module $E^{k}$, such that $s+t$ is selfadjoint and regular in $E^{k}$. If \[a(s+i)^{-1}(t+i)^{-1},a(s-i)^{-1}(t-i)^{-1}\in \K_{\mathcal{B}_{k}}(E^{k}),\] for some $a\in\Endst_{\mathcal{B}_{k}}(E^{k})$, then $a(s+t\pm i)^{-1}\in\K_{\mathcal{B}_{k}}(E^{k})$.
\end{lemma}
\begin{proof} By assumption, $a(s+t\pm i)^{-1}\in\Endst_{\mathcal{B}_{k}}(E^{k})$, so in view of corollary \ref{Kopalg}, it suffices to show that $a(s+t\pm i)^{-1}\in\K_{B}(\mathpzc{E})$. This in turn is the case if and only if $a(1+(s+t)^{2})^{-1}a^*\in\K_{B}(\mathpzc{E})$, since for a closed ideal $I$ in a $C^{*}$-algebra $C$ it is the case that $c^{*}c\in I\Leftrightarrow c\in I$ for all $c\in C$.\newline\newline
We have the identities
\[\begin{split}-ia(s+t+i)^{-1}&=a(s+i)^{-1}(t+i)^{-1}+a(s+i)^{-1}(t+i)^{-1}st(s+t+i)^{-1}\\
ia(s+t-i)^{-1}&=a(t-i)^{-1}(s-i)^{-1}-a(s-i)^{-1}(t-i)^{-1}st(s+t-i)^{-1},\end{split}\]
and combining these yields
\[\begin{split}-2a(1+(s+t)^{2})^{-1}&=a((s+i)^{-1}(t+i)^{-1}+(s+i)^{-1}(t+i)^{-1}st(s+t+i)^{-1}\\ &\quad+(t-i)^{-1}(s-i)^{-1}-(t-i)^{-1}(s-i)^{-1}ts(s+t-i)^{-1}).\end{split}\]
Now use that $(s+i)^{-1}(t+i)^{-1}+(t-i)^{-1}(s-i)^{-1}$ equals 
\[(t-i)^{-1}(s-i)^{-1}(2-[s,t])(t+i)^{-1}(s+i)^{-1},\]
and 
\[(s+t-i)^{-1}=(s+t+i)^{-1}+2i(1+(s+t)^{2})^{-1},\]
and write $(s+i)^{-1}=x, (t+i)^{-1}=y,$ to find
\[\begin{split}-2a(1+(s+t)^{2})^{-1}&=a(xy+y^{*}x^{*}+xy[s,t](s+t+i)^{-1} +2iy^{*}x^{*}ts(1+(s+t)^{2})^{-1}\\ &\quad -x^{*}y^{*}(2+[s,t])yxts(s+t+i)^{-1}).\end{split}\]
Since all terms on the right hand side are compact, the left hand side is compact. Therefore $a(1+(s+t)^{2})^{-1}$ and also $ a(1+(s+t)^{2})^{-1}a^{*}\in \K_{B}(\mathpzc{E})$ as desired.

\end{proof}
\begin{theorem}\label{compactresolvent}Let $k\geq 1$, $A,B,C$ be $C^{k}$-algebras, $(\mathpzc{E},S)$ a transverse $C^{k}$ $KK$-cycle for  $(A,B)$ and $(\mathpzc{F},S)$ a transverse $C^{k}$ $KK$-cycle for $(B,C)$. Let
$\nabla:E^{k}\rightarrow E^{k}\tildeotimes_{\mathcal{B}_{k}}\Omega^{1}(\mathcal{B}_{k})$ be a transverse $C^{1}$-connection on $\mathpzc{E}$. Then the operator
\[S\otimes 1+1\otimes_{\nabla} T\]
has $\mathcal{A}_{k}$-locally compact resolvent. That is, for $a\in\mathcal{A}_{k}$ we have \[a(S\otimes1 + 1\otimes_{\nabla}T\pm i)^{-1}\in \K_{\mathcal{B}_{k}}(E^{k}).\]
\end{theorem}

\begin{proof}
By lemma \ref{sumprodcpt} it suffices to show that $a(s+i)^{-1}
(i+t)^{-1}$ and $a(s-i)^{-1}(t-i)^{-1}$ are compact for $a\in\mathcal{A}_{k}$. Since the operator $s+t:=S\otimes 1+1\otimes_{\nabla}T$ is $C^{k}$, we have that
\[a(s+t\pm i)^{-1}\in\Endst_{\mathcal{B}_{k}}(E^{k}).\]
By corollary \ref{Kopalg} we have
\[\K_{\mathcal{B}_{k}}(E^{k})=\Endst_{\mathcal{B}_{k}}(E^{k})\cap\K_{B}(\mathpzc{E}),\]
so it suffices to show that $a(s+t\pm i)^{-1}\in\K_{B}(\mathpzc{E})$. By lemma \ref{cptperturb}, we only have to check this in case $\nabla$ is the Grassmann connection on $\mathpzc{H}_{B}$. Denote by $\{e_{j}\}_{j\in\Z\setminus\{0\}}$ the standard orthonormal basis of $\mathpzc{H}_{B}$. Note that
\[(s\pm i)^{-1}(e_{j}\otimes f)=(S\pm i)^{-1}e_{j}\otimes f,\quad (t\pm i)^{-1}(e_{j}\otimes f)=e_{j}\otimes(T\pm i)^{-1}f.\] Choose a countable, increasing, contractive approximate unit for $B$, such that for all $1\leq |j| \leq n\leq m$ we have $\|e_{j}(u_{n}-u_{m})\|\leq\frac{1}{n}$. The sequence
\[x_{n}=\sum_{1\leq |j|\leq n}a(S+i)^{-1}e_{j}\otimes u_{n}(i+T)^{-1}\otimes e_{j}\in\K_{C}(\mathpzc{H}_{B}\tildeotimes\mathpzc{F})\cong\mathpzc{H}_{B}\tildeotimes\K_{C}(\mathpzc{F})\tildeotimes\mathpzc{H}_{B}^{*},\]
converges pointwise to $a(s+i)^{-1}
(i+t)^{-1}$. We show it converges in norm. 
We have
\[\begin{split}x_{m}-x_{n}=\sum_{1\leq |j|\leq n} & a(S+i)^{-1}e_{j}(u_{m}-u_{n})\otimes (i+T)^{-1}\otimes e_{j}\\ &+\sum_{n+1\leq |j|\leq m}a(S+i)^{-1}e_{j}\otimes u_{m} (i+T)^{-1}\otimes e_{j}.\end{split}\]
A computation in the linking algebra yields
\[ \sum_{1\leq |j|\leq n}a(S+i)^{-1}e_{j}(u_{m}-u_{n})\tildeotimes (a(S+i)^{-1}e_{j}(u_{m}-u_{n}))^{*}\leq \|a\|^{2} \frac{2n}{n^{2}}=\|a\|^{2} \frac{2}{n},\]
and therefore
\[\|\sum_{1\leq |j|\leq n}a(S+i)^{-1}e_{j}(u_{m}-u_{n})\otimes (i+T)^{-1}\otimes e_{j}\|^{2}\leq \|a\|^{2}\frac{2}{n}\rightarrow 0.\]
For the tail
\[\sum_{n+1\leq |j|\leq m}a(S+i)^{-1}e_{j}\otimes u_{m}(i+T)^{-1}\otimes e_{j},\]
it is enough to observe that $\|u_{m}(i+T)^{-1}\|\leq 1$ and
\[\|\sum_{n+1\leq |j|\leq m}a(S+i)^{-1}e_{j}\|\rightarrow 0,\]
because $a(S+i)^{-1}$ is compact.
\end{proof}
Recall that $\Psi_{0}(A,B)$ denotes the set of unbounded $KK$-cycles up to unitary equivalence. For $k\geq 1$, we denote by $\Psi_{0}^{k}(A,B)$ the set of $C^{k}$ $KK$-cycles with transverse $C^{k
}$-connection on them, up to transverse $C^{k}$ unitary equivalence. Note that this requires fixing $C^{k}$-spectral triples for $A$ and $B$.
\begin{theorem}\label{KKproduct}For $k\geq 1$, the diagram
\begin{diagram}\Psi_{0}^{k}(A,B)\times\Psi_{0}^{k}(B,C) &\rTo^{(S,T)\mapsto S\otimes 1+1\otimes_{\nabla}T} &\Psi_{0}^{k}(A,C) \\
\dTo^{\mathfrak{b}} & & \dTo^{\mathfrak{b}}\\
KK_{0}(A,B)\otimes KK_{0}(B,C) & \rTo^{\otimes_{B}} & KK_{0}(A,C)\end{diagram}
commutes.
\end{theorem}
\begin{proof} We just need to check that the KK-cycles $(\mathpzc{E},S)$, $(\mathpzc{F},T)$ and $(\mathpzc{E}\tildeotimes_{B}\mathpzc{F}, S\tildeotimes 1+1\otimes_{\nabla}T)$ satisfy the conditions of theorem \ref{Kuc}.
If we write $D$ for $S\otimes 1+1\otimes_{\nabla}T=s+t$, we have to check that
\[J:=\left[\begin{pmatrix}D & 0\\ 0& T\end{pmatrix},\begin{pmatrix} 0 &
T_{e}\\T^{*}_{e} & 0\end{pmatrix}\right]\]
is bounded on $\Dom(D\oplus T)$. This is a straightforward calculation:
\[\begin{split}J\begin{pmatrix}e'\otimes f' \\ f\end{pmatrix} 
&=\begin{pmatrix}Se\otimes f +(-1)^{\partial e}\nabla_{T}(e)f\\ \langle e,Se'\rangle f + [T,\langle e,e'\rangle]f +(-1)^{-\partial e'}\langle e,\nabla_{T}(e')\rangle f\end{pmatrix}\\
&=\begin{pmatrix}Se\otimes f +(-1)^{\partial e}\nabla_{T}(e)f\\ \langle Se,e'\rangle f + \langle \nabla_{T}(e),e'\rangle f\end{pmatrix}.\end{split}\]
This is valid whenever $e\in \Dom S\cap E^{1}$, which is dense in $\mathpzc{E}$.\newline\newline
The second condition $\mathfrak{Dom}(D)\subset\mathfrak{Dom}(S\tildeotimes 1)$ is obvious, so we turn the semiboundedness condition 
\begin{equation}\label{com}\langle S\tildeotimes 1 x,Dx\rangle + \langle Dx,S\tildeotimes 1 x\rangle \geq\kappa\langle x,x\rangle,\end{equation} 
must hold for all $x$ in the domain. On $\mathfrak{Im}(s+\lambda i)^{-1}(t+\lambda i)^{-1}$, which is a common core for $s$ and $D$, the expression $\ref{com}$ is equal to 
\[\langle [D,S\tildeotimes 1]x,x\rangle=\langle [s+t,s]x,x\rangle=\langle sx,sx\rangle + \langle [s,t]x,x\rangle\geq -\|[s,t]\|\langle x,x\rangle,\]
and the last estimate is valid since $[s,t]$ is in $\Endst_{C}(\mathpzc{E}\tildeotimes_{B}\mathpzc{F})$. Thus, it holds for all $x$ in the domain.
\end{proof}
The functor $KK$ forgets all the smoothness assumptions imposed on the cycles in $\Psi^{k}_{0}(A,B)$. The problem of smoothening given cycles and equipping them with a connection shall be dealt with elsewhere. Also note that for forming unbounded Kasparov products, it is enough to have a $C^{1}$-cycle with transverse $C^{1}$-connection.  It is relevant for the categorical considerations of the next section.

\subsection{A category of spectral triples}

Let $A$ and $B$ be smooth $C^{*}$-algebras. We saw that triples $(\mathpzc{E},D,\nabla)$ consisting of a smooth $(A,B)$-bimodule equipped with a smooth regular
operator $D$ and a smooth connection $\nabla$ form a category, in which the composition law is
\[(E^{k},S,\nabla)\circ(F^{k}, T, \nabla'):=(E^{k}\tildeotimes_{\mathcal{B}_{k}}F^{k},S\otimes 1+1\otimes_{\nabla} T,1\otimes_{\nabla}\nabla').\]
This can be naturally interpreted as a category of spectral triples.

\begin{definition} Let $A$ and $B$ be $C^{*}$-algebras, and $(\mathpzc{H},D)$ and $(\mathpzc{H}',D')$ be $C^{k}$ spectral triples for $A$ and $B$ respectively, with $k\geq 1$. A $C^{k}$-\emph{correspondence} $(\mathpzc{E},S,\nabla)$ between $(\mathpzc{H},D)$ and
$(\mathpzc{H}',D')$ is a class $[(E^{k}, S,\nabla)]\in\Psi^{k}_{0}(A,B)$ of a  $C^{k}$-$(A,B)$-bimodule with transverse $C^{k}$-connection, such that there is a unitary isomorphism of spectral triples $(\mathpzc{H},D)\cong(\mathpzc{E}\tildeotimes_{B}\mathpzc{H}',S\otimes 1+1\otimes_{\nabla}D')$ 
and $\mathfrak{G}(D_{i})\cong\mathfrak{G}(S\tildeotimes 1+1\tildeotimes_{\nabla}D')_{i}$ for $i=0,...,k$ under this isomorphism. The correspondence is \emph{smooth} if it is $C^{k}$ for all $k$. Two correspondences are said to be equivalent
if they are $C^{k}$- or smoothly unitarily isomorphic such that the unitary intertwines the operators and connections and induces isomorphisms on the graphs and the smooth structure up to degree $k$. The set of isomorphism classes of such correspondences is denoted by $\mathfrak{Cor}_{k}(D,D')$ or
$\mathfrak{Cor}(D,D')$ in the smooth case.
\end{definition}
We can reformulate the previous results as a categorical statement.
\begin{theorem}\label{functor} For $k\geq 1$, there is a category whose objects are $C^{k}$-spectral triples and whose morphisms are the sets $\Cor_{k}(D,D')$. The bounded transform $\mathfrak{b}(\mathpzc{E},D,\nabla)=(\mathpzc{E},\mathfrak{b}(D))$
defines a functor $\mathfrak{Cor}_{k}\rightarrow KK$.
\end{theorem}

A category with unbounded $C^{k}$-cycles as objects can be constructed in a similar way. A morphism of unbounded cycles $A\rightarrow(\mathpzc{E},D)\leftrightharpoons B$ and
$A'\rightarrow(\mathpzc{E}',D')\leftrightharpoons B'$ is given by a correspondence $A\rightarrow(\mathpzc{F},S,\nabla)\leftrightharpoons A'$ and a bimodule $B\rightarrow\mathpzc{F}'\leftrightharpoons B'$, where $B$ is
represented by compact operators. The bounded transform functor then takes values in the \emph{morphism category} $KK^{2}$.\newline

Furthermore, we would like to note that the category of spectral triples constructed is a 2-category. A morphism of morphisms $f:(\mathpzc{E},D,\nabla)\rightarrow (\mathpzc{E}',D',\nabla ')$ is given by an element
$F\in\Hom^{*}_{\mathcal{B}_{k}}(E^{k},F^{k})$, inducing morphisms
\[\mathfrak{G}(D_{i})^{i}\rightarrow \mathfrak{G}(D'_{i})^{i},\]
commuting with the left $\mathcal{A}_{i}$-module structure, intertwining the connections and the operators.\newline

The external product of correspondences is defined in the expected way:
\[(\mathpzc{E},D,\nabla)\otimes(\mathpzc{E}',D'\nabla '):=(\mathpzc{E}\minotimes\mathpzc{E}',D\minotimes 1+1\minotimes D',\nabla\minotimes 1+1\minotimes\nabla).\]
In this way, $\mathfrak{Cor}$ becomes a symmetric monoidal category.

\appendix
\setcounter{equation}{0}
\renewcommand{\thedefinition}{\Alph{section}.\arabic{definition}}
\renewcommand{\thetheorem}{\Alph{section}.\arabic{theorem}}
\renewcommand{\thecorollary}{\Alph{section}.\arabic{corollary}}
\renewcommand{\theexample}{\Alph{section}.\arabic{example}}
\renewcommand{\thelemma}{\Alph{section}.\arabic{lemma}}
\renewcommand{\theproposition}{\Alph{section}.\arabic{proposition}}
\renewcommand{\theequation}{\Alph{section}.\arabic{equation}}
\renewcommand{\theremark}{\Alph{section}.\arabic{remark}}
\section{Smoothness and regularity}

 Recall that  a spectral triple $(A,\mathcal{H},D)$ is \emph{regular} \cite{Conspec} if there is a dense subalgebra $\mathscr{A}\subset A$ such that
$\mathscr{A}$ and $[D,\mathscr{A}]$ are in $\mathfrak{Dom}^{\infty}\textnormal{ad}|D|$. We now proceed to show that the notion of smoothness introduced above is weaker than regularity.\newline\newline
For a regular bimodule we introduce representations $\pi_{i}':\mathscr{A}\rightarrow M_{2^{i}}(\Endst_{B}(\mathpzc{E}))$ inductively by setting $\pi_{0}'(a):=a$ and
\[\pi'_{i+1}(a):=\begin{pmatrix}\pi_{i}'(a) & 0\\
[|D|,\pi_{i}'(a)] & \pi_{i}'(a)\end{pmatrix}.\]
Subsequently, define representations $\theta_{i}':\mathscr{A}\rightarrow M_{2^{i}}(\Endst_{B}(\mathpzc{E}))$ by
\[\theta_{i}'(a):=p_{[i]}^{|D|}\pi_{i}'(a)p_{[i]}^{|D|}+v_{[i]} p_{[i]}^{|D|}v_{[i]}^{*}\gamma^{i}\pi'_{i}(a)\gamma^{i} v_{[i]}p_{[i]}^{|D|}v_{[i]}^{*}.\]
Here we use $\gamma$ to denote the usual diagonal grading on $\bigoplus_{j=0}^{2^{i}}\mathpzc{E}$. Notice that for $i$ even, the $\gamma$'s disappear from the formula. Both the $\pi'_{i}$ and $\theta'_{i}$ are graded representations for the diagonal grading, i.e. $\pi'_{i}(\hat{\gamma}(a))=\gamma\pi'_{i}(a)\gamma$, because $|D|$ is even.
\begin{lemma}\label{dom} Let $\mathpzc{E}$ be an $(A,B)$-bimodule, $D$ a selfadjoint regular operator in $\mathpzc{E}$. For all $i$ there exist unitaries $u_{i}$ such that $u_{i}\theta^{D}_{i}u_{i}^{*}=\theta_{i}'$. In particular \[\mathfrak{Dom}\theta_{i}^{D}=\mathfrak\Dom\theta'_{i}.\]
\end{lemma}
\begin{proof}The operator
\[U_{i}:=\begin{pmatrix}(1+|D|D)\mathfrak{r}(D)^{2}&(D-|D|) \mathfrak{r}(D)^{2} \\
(|D|-D)\mathfrak{r}(D)^{2}&(1+|D|D)\mathfrak{r}(D)^{2}\end{pmatrix}\in M_{2^{i+1}}(\Endst_{B}(\mathpzc{E})),\]
is unitary and maps $\mathfrak{G}(D)$ to $\mathfrak{G}(|D|)$. Moreover it commutes with both $D$ and $v_{i}$ and intertwines the Woronowicz projections:
\begin{equation}U_{i}p^{D}_{i}U^{*}_{i}=p_{i}^{|D|}.\end{equation}\label{pro}
Set $u_{1}:=U_{1}$, and inductively define
\[u_{i+1}:=\begin{pmatrix} u_{i} & 0\\0& u_{i}\end{pmatrix} U_{i},\]
so that $u_{i+1}p^{D}_{[i+1]}u_{i+1}^{*}=p^{|D|}_{[i+1]}$.
The $u_{i}$ intertwine the $\theta_{i}$'s:
\begin{equation}\label{theta} u_{i}\theta_{i}^{D}(a)u_{i}^{*}=\theta'_{i}(a).\end{equation}
To see this, note that $\theta_{i}=p_{[i]}\pi_{i} p_{[i]} +v_{[i]} p_{[i]}v_{[i]}^{*}\pi_{i}v_{[i]}p_{[i]}v_{[i]}^{*}$, and that it is clear that \[u p^{D}\pi^{D}_{1}p^{D}u^{*}=p^{|D|}\pi^{|D|}_{1}p^{|D|}.\]
Then
\[\begin{split}up^{D\perp}\pi_{1}^{D}(a)p^{D\perp}u^{*}&=uvp^{D}v^{*}\pi_{1}^{D}(a)vp^{D}v^{*}u^{*}\\
&=uvp^{D}\gamma_{1}\pi_{1}^{D}(a^{*})^{*}\gamma_{1}p^{D}v^{*}u^{*}\\
&=uvp^{D}\pi_{1}^{D}(\hat{\gamma}(a^{*}))^{*}p^{D}v^{*}u^{*}\\
&=vp^{|D|}\pi_{1}^{|D|}(\hat{\gamma}(a^{*}))^{*}p^{|D|}v^{*}\\
&=p^{|D|\perp}\gamma\pi_{1}^{|D|}(a)\gamma p^{|D|\perp}.\end{split}\]
So \ref{theta} holds for $i=1$. Suppose that \ref{theta} holds for $i$. Then since
\[U_{i+1}p^{D}_{i+1}p^{D}_{[i]}\pi^{D}_{i+1}(a)p^{D}_{i+1}p^{D}_{i}U_{i+1}^{*}=p^{|D|}_{i+1}p^{D}_{[i]}\begin{pmatrix}\theta_{i}(a) & 0\\ [|D|,\theta_{i}(a)] &\theta_{i}(a)\end{pmatrix}p_{[i]}^{D} p^{|D|}_{i+1},\]
and $p^{|D|}_{i+1}$ commutes with $\begin{pmatrix} u_{i} & 0 \\ 0 & u_{i}\end{pmatrix}$, it follows that
\[\begin{split}u_{i+1}p^{D}_{[i+1]}\pi^{D}_{i+1}(a)p^{D}_{[i+1]}u_{i+1}^{*} & =p^{|D|}_{i+1}\begin{pmatrix}u_{i}p^{D}_{[i]}\theta_{i}(a)p^{D}_{[i]}u_{i}^{*} & 0\\ [|D|,u_{i}p^{D}_{[i]}\theta_{i}(a)p^{D}_{[i]}u_{i}^{*} ] & u_{i}p^{D}_{[i]}\theta_{i}p^{D}_{[i]}(a)u_{i}^{*}\end{pmatrix} p^{|D|}_{i+1}\\
&=p^{|D|}_{i+1}p^{|D|}_{[i]}\begin{pmatrix}\theta'_{i}(a) & 0\\ [|D|,\theta'_{i}] & \theta'_{i}(a)\end{pmatrix}p^{|D|}_{[i]}p^{|D|}_{i+1}\\
&=p^{|D|}_{[i+1]}\pi'_{i+1}(a)p^{|D|}_{[i+1]}.\end{split}\]
Using either \ref{inveven} or \ref{invodd} and the fact that $u_{i}$ and $v_{[i]}$ commute, one obtains that
\[\begin{split}u_{i+1}v_{[i+1]}p^{D}_{[i+1]}v^{*}_{[i+1]}\pi_{i+1}^{D}(a)v_{[i+1]}p^{D}_{[i+1]} & v_{[i+1]}^{*}u_{i+1}^{*}\\ &=v_{[i+1]}p^{|D|}_{i+1}v_{[i+1]}^{*}\gamma^{i}\pi'_{i}(a)\gamma^{i} v_{[i+1]}p^{|D|}_{i+1}v^{*}_{[i+1]},\end{split}\]
in the same way as for $i=1$. Thus, $u_{i+1}\theta_{i+1}^{D}u_{i+1}^{*}=\theta'_{i+1}$.
\end{proof}

\begin{theorem} Let $(\mathpzc{E},D)$ be a regular unbounded $(A,B)$-bimodule. Then $(\mathpzc{E},D)$ is smooth.
\end{theorem}
\begin{proof} We will show that $\mathscr{A}\subset\mathcal{A}_{n}$ for all $n$. By definition, $\mathscr{A}\subset\mathcal{A}_{1}$, so suppose $\mathscr{A}\subset \mathcal{A}_{n}$. Then $\theta_{n}(a)$ is well defined, and we have to show that $[D,\theta_{n}(a)]$ extends to an adjointable operator. From lemma \ref{dom} it follows that for $a\in\mathscr{A}$,
\[\begin{split} [D,\theta_{n}(a)]&=u_{n}[D,\theta'_{n}(a)]u_{n}^{*}\\
&=u_{n}(p_{[n]}[D,\pi'_{n}(a)]p_{[n]}+v_{[n]}p_{[n]}v_{[n]}^{*}[D,\gamma^{n}\pi'_{n}(a)\gamma^{n}]v_{[n]}p_{[n]}v_{[n]}^{*})u_{n}^{*}\\
&=u_{n}(p_{[n]}[D,\pi'_{n}(a)]p_{[n]}+(-1)^{n}v_{[n]}p_{[n]}v_{[n]}^{*}\gamma^{n}[D,\pi'_{n}(a)]\gamma^{n}v_{[n]}p_{[n]}v_{[n]}^{*})u_{n}^{*}.\end{split}\]
Since $(\mathpzc{E},D)$ is regular, 
\[[D,\mathscr{A}]\subset\mathfrak{Dom}(\textnormal{ad }|D|)^{n},\]
which is the same as saying that 
\[(\textnormal{ad }|D|)^{n}(\mathscr{A})\subset \mathfrak{Dom}(\textnormal{ad } D).\]
Therefore we have that $[D,\pi'_{n}(a)]\in M_{2^{n}}(\Endst_{B}(\mathpzc{E}))$ for $a\in\mathscr{A}$. It follows that $\mathscr{A}\subset\mathcal{A}_{n+1}$ as desired.
\end{proof}

\section{Nonunital $C^{k}$-algebras}
In this appendix we describe a principle to reduce the defining representation of a $C^{k}$ algebra in $\Sob_{k}(D)$ to the essential case. Note that a homomorphism $\pi:\mathcal{A}\rightarrow\mathcal{B}$ between operator algebras is \emph{essential} when $\pi(\mathcal{A})\mathcal{B}\pi(\mathcal{A})$ is dense in $\mathcal{B}$. In that case $\pi$ extends to a map $\mathscr{M}(\mathcal{A})\rightarrow\mathscr{M}(\mathcal{B})$ (see \cite{Blechbook}). When $\mathcal{B}$ is unital and $\mathcal{A}$ is cb-isomorphic to $\pi(\mathcal{A})$, we have
\[\mathscr{M}(\mathcal{A})\cong\{T\in \mathcal{B}:T\pi(\mathcal{A}),\pi(\mathcal{A})T\subset\pi(\mathcal{A})\}.\]
In what follows,  we will use some well known properties of the strong and weak topologies on the algebra of bounded operators on a separable infinite dimensional Hilbert space.\newline 

Recall that, for a Hilbert space $\mathpzc{H}$, the \emph{strong operator topology} on $B(\mathpzc{H})$ is the topology of pointwise norm convergence: $T_{i}\rightarrow T$ strongly if $T_{i}\xi\rightarrow T\xi$ in norm for all $\xi\in\mathpzc{H}$. The \emph{weak operator topology} is the weakest topology that makes the functionals $T\mapsto\langle T\xi,\eta\rangle$ continuous, so $T_{i}\rightarrow T$ weakly when $\langle T_{i}\xi,\eta\rangle\rightarrow \langle T\xi,\eta\rangle$ for all $\xi,\eta\in\mathpzc{H}$. \newline 

These topologies are complete on the closed unit ball of $B(\mathpzc{H})$. Hence a sequence $T_{i}$ with $\sup_{i}\|T_{i}\|\leq C$ that is strongly or weakly  Cauchy has a limit in $B(\mathpzc{H})$ in the respective topologies. Moreover, operator multiplication is separately continuous for these topologies: if $T_{i}\rightarrow T$ strongly or weakly, then $ST_{i}\rightarrow ST, T_{i}S\rightarrow TS$ strongly or weakly, respectively. 
\begin{lemma}\label{ws} Let $a_{n},b\in\Sob_{i}(D)$ be a sequence such that $\pi_{i}(a_{n})\rightarrow\pi_{i}(b)$ strongly resp. weakly. Then $\theta_{i}(a_{n})\rightarrow\theta_{i}(b)$ strongly resp. weakly.
\end{lemma}
\begin{proof} We prove the statement for the weak topology: According to \eqref{trep} we have
\[\langle\theta_{i}(a_{n})\xi,\eta\rangle=\langle p_{i}p_{i-1}\pi_{i}(a_{n})p_{i-1}p_{i}\xi,\eta\rangle +\langle p_{i}^{\perp}p_{i-1}^{\perp}\pi_{i}(a_{n})p_{i-1}p_{i}\xi,\eta\rangle,\]
from which the statement is immediate.
\end{proof}
\begin{proposition}\label{projprop}Let $B$ be a nonunital $C^{k}$ algebra with spectral triple $(\mathpzc{H},D)$. Let $p$ be the projection onto the essential subspace $\overline{B\mathpzc{H}}$. Then $p\in\Sob_{k}(D)$ and $\pi_{k}(b)=\pi_{k}(p)\pi_{k}(a)\pi_{k}(p)$.
\end{proposition}
\begin{proof} When $B$ is unital there is nothing to prove. A nonunital $C^{k}$-algebra $B$ by definition posesses an even, increasing, contractive approximate unit $u_{n}$, which is completely bounded in $\mathcal{B}_{k}$. Denote by $p\in B(\mathpzc{H})$ the projection onto $\overline{B\mathpzc{H}}$, the essential subspace of $B$. It is well known that $u_{n}\rightarrow p$ in the weak operator topology on $\mathpzc{H}$. Since $u_{n}(1-p)\mathpzc{H}=u_{n}(B\mathpzc{H})^{\perp}=0$, and $u_{n}b\rightarrow b$ in norm, it follows that
\[(u_{n}-p)h=(u_{n}-p)(ph+(1-p)h)=(u_{n}-p)ph\rightarrow 0,\]
that is, $u_{n}\rightarrow p$ strongly.\newline
We will show that $p\in\Sob_{k}(D)$ by induction on $k$. Suppose that for $k-1$ we have shown that $p\in\Sob_{k-1}(D)$ and $\pi_{k-1}(u_{n})\rightarrow \pi_{k-1}(p)$ strongly. By lemma \ref{ws} we get that $\theta_{k-1}(u_{n})\rightarrow\theta_{k-1}(p)$ strongly. Then, we need to show that $\theta_{k-1}(p)\Dom D\subset\Dom D$ and $[D,\theta_{k-1}(p)]$ is bounded on $\Dom D$. To this end observe that $[D,\theta_{k-1}(u_{n})]$ is a uniformly bounded sequence of operators that is weakly Cauchy: 
\[\begin{split} \langle [D,\theta_{k-1}(u_{n})]\xi,\eta \rangle &=\langle \theta_{k-1}(u_{n})\xi, D\eta\rangle -\langle \theta_{k-1}(u_{n})D\xi,\eta\rangle \\ & \rightarrow \langle \theta_{k-1}(p)\xi,D\eta\rangle -\langle \theta_{k-1}(p)D\xi,\eta\rangle,\end{split}\]
for $\xi,\eta\in \Dom D$, which is dense. Therefore $\pi_{k}(u_{n})$ has a weak limit in $B(\bigoplus_{i=1}^{2^{k}}\mathpzc{H})$, which we denote by $q$. Denote by $p_{k}$ the projection onto $\overline{\pi_{k}(\mathcal{B}_{k})(\bigoplus_{i=1}^{2^{k}}\mathpzc{H})}$. Then since $\pi_{k}(u_{n})\pi_{k}(b)\rightarrow \pi_{k}(b)$ in norm, we find that $\pi_{k}(u_{n})\rightarrow 1$ strongly on $\im p_{k} $, and hence $qp_{k}=p_{k}$. On the other hand, taking weak limits we find \[p_{k}q=\lim_{n} p_{k}\pi_{k}(u_{n})=\lim_{n} \pi_{k}(u_{n})=q,\] and so \[\im p_{k}=\im q=\overline{\pi_{k}(\mathcal{B}_{k})(\bigoplus_{i=1}^{2^{k}}\mathpzc{H})}.\] Therefore $q$ is an idempotent:
\[q^{2}h=\lim \pi_{k}(u_{n})qh=qh.\]
Applying the above procedure to $\pi_{k}(u_{n})^{*}$ yields $\im q^{*}=\overline{\pi_{k}(\mathcal{B}_{k})^{*}(\bigoplus_{i=1}^{2^{k}}\mathpzc{H})}$ and thus 
\[\im(1-q)=(\im q^{*})^{\perp}=\overline{\pi_{k}(\mathcal{B}_{k})^{*}(\bigoplus_{i=1}^{2^{k}}\mathpzc{H})}^{\perp},\] from which we see that
\[\langle \pi_{k}(b)(1-q)\xi,\eta\rangle=\langle (1-q)\xi,\pi_{k}(b)^{*}\eta\rangle=0,\] for all $\xi,\eta$ so $\pi_{k}(b)(1-q)=0$. Consequently \[\pi_{k}(u_{n})\xi=\pi_{k}(u_{n})q\xi\rightarrow q\xi,\] i.e. $\pi_{k}(u_{n})\rightarrow q$ strongly. Taking $\xi=\begin{pmatrix}
\eta \\ D\eta\end{pmatrix}$, we find
\[\pi_{k}(u_{n})\begin{pmatrix}
\eta \\ D\eta\end{pmatrix}=\begin{pmatrix}
\theta_{k-1}(u_{n})\eta \\ D\theta_{k-1}(u_{n})\eta\end{pmatrix}\rightarrow \begin{pmatrix}
\theta_{k}(p)\eta \\ D\theta_{k}(p)\eta
\end{pmatrix},\]  so $\theta_{k-1}(p)$ preserves $\Dom D$ and $[D,\theta_{k-1}(p)]$ is a densely defined operator that is the weak limit of the bounded sequence $[D,\theta_{k-1}(u_{n})]$, so it must be bounded. Hence $q=\pi_{k}(p)$ and  $\pi_{k}(u_{n})\rightarrow \pi_{k}(p)$ in the strong topology, completing the induction step.
\end{proof}
\begin{theorem}\label{multipliersob} The inclusion $\mathcal{B}_{k}\subset\Sob_{k}(D)$ naturally extends to an inclusion  $\mathscr{M}(\mathcal{B}_{k})\subset\Sob_{k}(D)$ of involutive operator algebras. In particular $\mathscr{M}(\mathcal{B}_{k})$ is spectral invariant in its $C^{*}$-closure.
\end{theorem}
\begin{proof} Let $p$ be the projection on the essential subspace $\overline{B\mathpzc{H}}$. The inclusion $\mathcal{B}_{k}\subset\Sob_{k}(D)$ contracts to an inclusion $\mathcal{B}_{k}\subset \pi_{k}(p)\Sob_{k}(D)\pi_{k}(p)$. This inclusion is essential, and the operator algebra $\pi_{k}(p)\Sob_{k}(D)\pi_{k}(p)$ is unital. Therefore \[\mathscr{M}(\mathcal{B}_{k})=\{a\in \pi_{k}(p)\Sob_{k}(D)\pi_{k}(p): ab,ba\in\mathcal{B}_{k}\},\]
which is by definition a subalgebra of $\Sob_{k}(D)$.
\end{proof}
\begin{theorem}\label{essrep} Let $\mathpzc{E}\lrh B$ be a $C^{k}$-module over the $C^{k}$-algebra $B$, with spectral tiple $(\mathpzc{H},D)$. The map
\[\begin{split}\Endst_{\mathcal{B}_{k}}(E^{k})&\rightarrow \bigoplus_{i=0}^{k}B(E^{k}\tildeotimes_{\mathcal{B}_{k}} \bigoplus_{j=1}^{2^{i}}\mathpzc{H}) \\
T&\mapsto T\otimes\pi_{[k]}(p),\end{split}\]
is a cb-isomorphism onto its image. The same holds for $\K_{\mathcal{B}_{k}}(E^{k})$
\end{theorem}
\begin{proof} By proposition \ref{multipliersob}, the projection $p$ can be used to turn \[\pi_{[k]}:\mathcal{B}_{k}\rightarrow \bigoplus_{i=0}^{k}B(\bigoplus_{j=1}^{2^{i}}\mathpzc{H}),\](cf.\eqref{directpi}) into an essential representation on the Hilbert space 
\[\pi_{[k]}(p)\bigoplus_{i=0}^{k}(\bigoplus_{j=1}^{2^{i}}\mathpzc{H}).\] Realizing $E^{k}$ cb-isomorphically as a submodule of $\mathpzc{H}_{\mathcal{B}_{k}}$ and then using  \cite{Blech2}, theorem 6.10, we find that $T\mapsto T\otimes 1$ is a cb isomorphism for tensoring with this essential representation. Since $1=\pi_{[k]}(1)=\pi_{[k]}(p)+\pi_{[k]}(1-p)$ and $\pi_{[k]}(b)\pi_{[k]}(1-p)=0$ by proposition \ref{multipliersob}, this extends to a cb isomorphism $T\mapsto T\otimes\pi_{[k]}(p)$ on the full tensor product.
\end{proof}

\end{document}